\documentclass[a4paper,12pt,reqno]{amsart}

\usepackage{graphicx}
\usepackage{amsmath}
\usepackage{amsthm}
\usepackage{amssymb}
\usepackage{amstext}
\usepackage{mathtools}
\usepackage{multicol}
\usepackage{booktabs}
\usepackage{lmodern,textcomp}
\usepackage{eurosym}
\usepackage{bbm}
\usepackage{color}
\usepackage[shortlabels]{enumitem}
\usepackage{hyperref}
\usepackage[ansinew]{inputenc}
\usepackage{natbib}
\usepackage[a4paper,
left=2.5cm,
right=2.5cm,
top=2.7cm,
bottom=3cm]{geometry}

\newcommand{\floor}[1]{\left\lfloor #1 \right\rfloor}
\newcommand{\ceil}[1]{\lceil #1 \rceil}
\newcommand{\norm}[1]{\left|\left|#1\right|\right|}

\newcommand{\E}[1]{\mathrm E\left\{ #1 \right\}}

\newcommand{\Prob}[1]{\mathrm P\left(#1\right)}

\newcommand{\ER}[0]{\overline{\mathbb R}}
\newcommand{\EOuter}[1]{\mathrm E^*\hspace{-1mm}\left\{ #1 \right\}}
\newtheorem{theorem}{Theorem}
\newtheorem{lemma}{Lemma}
\newtheorem{remark}{Remark}

\newtheorem{proposition}{Proposition}

\newtheorem{corollary}{Corollary}
\newenvironment{AMS}{}{}
\newcommand{\affil}[1]{\textsuperscript{#1}}

\title[Sequential empirical processes under nonstationarity]{On the Weak Convergence of the Function-Indexed Sequential Empirical Process and its Smoothed Analogue under Nonstationarity}
\author{Florian Alexander Scholze\affil{1,2}}
\address[F.A. Scholze]{\affil{1}Otto-Friedrich-Universit\"at Bamberg, Institute of Statistics, Feldkirchenstra{\ss}e 21, D-96052 Bamberg, Germany}
\email[Corresponding author]{florian.scholze@uni-bamberg.de}

\author{Ansgar Steland\affil{2}}
\address[A. Steland]{\affil{2}RWTH Aachen University, Institute of Statistics, Pontdriesch 14-16, D-52062 Aachen, Germany}
\email{steland@stochastik.rwth-aachen.de}
\date{\today}

\begin{document}
	\maketitle
	
	\begin{abstract}
		We study the sequential empirical process indexed by general function classes and its smoothed set-indexed analogue. Sufficient conditions for asymptotic equicontinuity are provided for nonstationary arrays of time series. This yields comprehensive general results that are applicable to various notions of dependence, which is exemplified in detail for nonstationary $\alpha$-mixing series. Especially, we obtain the weak convergence of the sequential process under essentially the same mild assumptions as known for the classical empirical process. Core ingredients of the proofs are a novel maximal inequality for nonmeasurable stochastic processes, uniform chaining arguments and, for the set-indexed smoothed process, uniform Lipschitz properties.
	\end{abstract} 
	
	
	\textit{AMS:}

	\section{Introduction}\label{Section:Introduction}
	Let $(\Omega, \Upsilon, P)$ be a probability space, $(\mathcal X, \mathcal A)$ a Polish space, and $(X_{i,n}) = (X_{i,n}: (\Omega, \Upsilon) \to (\mathcal X, \mathcal A) ~|~ i = 1,...,n, n \in \mathbb N)$ be an array of $\mathcal X$-valued random variables. Denote by $\mathcal F$ a family of Borel measurable maps $\mathcal X \to \mathbb R$. We study the sequential empirical process, defined by
	\begin{align}\label{DEF:SEP}
		\mathbb Z_n(t,f) = \frac{1}{\sqrt{n}} \sum_{i = 1}^{\floor{nt}} \left( f(X_{i,n}) - \E{f(X_{i,n})} \right), ~ (t,f) \in [0,1]\times \mathcal F,
	\end{align}
	and  its smoothed version
	\begin{align}\label{DEF:SEP-SMOOTH}
		\mathbb Z_n^s(t,f) &= \mathbb Z_n(t,f) +  \frac{nt - \floor{nt}}{\sqrt{n}}\left( f(X_{\floor{nt}+1,n}) - \E{f(X_{\floor{nt}+1,n})} \right) , 
	\end{align}
	which we will generalize to set-indexed smoothed processes indexed by subsets of the time interval $ [0,1] $. The corresponding non-sequential empirical process is denoted by $ \mathbb G_n(f) = \mathbb Z_n(1,f) = \mathbb Z_n^s(1,f) $. 
	
	When it comes to nonstationary data, it is crucial to study conditions that ensure the weak convergence of these processes, either in the classical sense or in the ``relative'' sense recently proposed by \cite{PN25}. Of particular interest are investigations of their asymptotic tightness. Here, both the complexity of the indexing class and the dependence structure of the array matters, whereas finite-dimensional (fidi) convergence merely requires assumptions on the dependence and the elements of the class. The primary goal of this paper is therefore to establish and discuss sufficient conditions for the asymptotic tightness of $ \mathbb Z_n $ and $ \mathbb Z_n^s $. Specifically, we contribute to the literature on non-Borelian dependent processes, on sufficient regularity conditions for asymptotic equicontinuity in terms of moment bounds and Lipschitz-properties and on (sequential) empirical processes for strongly mixing arrays, all in the context of nonstationary time series. There are various obstacles to be overcome, especially in the context of nonstationary time series, and several tools and techniques, ranging from moment bounds to chaining arguments, need to be further developed for this setting. 
	
	The process $ \mathbb Z_n $, and to a somewhat smaller extent the smoothed process $ \mathbb Z_n^s $ as well,  has become a widely used tool in nonparametric statistics, e.g. in the field of change-point analysis (see, e.g., \cite{SN13}, \cite{AS2016}, \cite{PS2017}, \cite{AS2020}, \cite{MN20}, \cite{MN21}, and the references given in these papers), goodness-of-fit testing (see, e.g., \cite{Rem17}) and the construction of confidence intervals of estimators based on self-normalization (see \cite{BU15} and \cite{Shao2010}). The special case that $\mathcal F$ consists of indicators of $d$-dimensional intervals $(-\infty,x], x \in \mathbb R^d$ , has received by far the most attention (see, e.g., \cite{DDT14} for an overview and \cite{BU15}), but some recent applications also involve other and more general families $\mathcal F$ (see, e.g., \cite{HVS15}, \cite{AS2016}, \cite{PS2017}, \cite{MN20} and \cite{MN21}), and therefore the study of \eqref{DEF:SEP} and \eqref{DEF:SEP-SMOOTH} for general classes of functions $ \mathcal{F} $ is of interest.
	
	For i.i.d. and stationary observations, the weak limit theory of the empirical process and its sequential generalization are well established. We refer to \cite{DDT14} for a brief review and thus limit our discussion correspondingly.
	If $(X_{i,n}) = (X_n)$ is an i.i.d.-sequence and $\mathcal F$ is a set of square-integrable maps, then \cite[Thm.~2.12.1]{vdVW23} shows that $\mathbb Z_n$ converges weakly to a two-parameter process, the Kiefer process, if and only if $\mathbb G_n$ converges weakly to a one-parameter process, the $P$-Brownian bridge indexed by $ \mathcal F$. Since $\mathbb G_n = \mathbb Z_n(1,.)$, this is the best one can hope for as it reduces the problem of proving the weak convergence of $\mathbb Z_n$ to the task of proving that $\mathcal F$ is a Donsker class, which usually requires purely analytical considerations. It is natural to ask to which extent similar relations also hold in settings in which there are dependencies among the $(X_{i,n})$. Indeed, for a stationary sequence $(X_{i,n}) = (X_n)$ and a general family $\mathcal F$, \cite{DDT14} and \cite{MO20} establish the weak convergence of $\mathbb Z_n$ under multiple and strong mixing conditions on $(X_n)$, respectively, and \cite{BUCH2018} treats the case of long-range dependent stationary Gaussian sequences. Furthermore, \cite{VS14} prove the weak convergence of $\mathbb Z_n$ by imposing high-level assumptions on the empirical process $\mathbb G_n$ that can be verified under various combinations of short-range-dependence conditions on $(X_n)$ and conditions on the complexity of $\mathcal F$, thereby avoiding the need to specify a particular time series model and achieving a higher degree of generality. Given these results, the question arises whether the assumption of stationarity of the data can be dropped as well. However, to the best of our knowledge, weak convergence of $ \mathbb Z_n $ indexed by a family $ \mathcal F $ has not yet been studied for dependent nonstationary arrays, despite the growing body of literature on the weak convergence of the special case $\mathbb G_n$ in such settings (see \cite{AP94}, \cite{Hansen96}, and, more recently, \cite{MO20}, \cite{AS2020}, \cite{PR21}, \cite{PR22}, \cite{MS23} and \cite{BL24}). This reveals a gap between $\mathbb Z_n$ and $\mathbb G_n$ and providing suitable results for $\mathbb Z_n$ closes this gap and  contributes to sequential nonparametrics. For example, in the context of change-point analysis for nonparametric time series models, required assumptions on $ \mathbb Z_n $, as in \cite{AS2016}, \cite{MN20} or \cite{MN21}, can be simplified. Our work was  carried out independently of \cite{PN25}, which study a relative notion of weak convergence to handle nonstationarity and provided such result for the sequential empirical process under $ \beta$-mixing and square-root integrability of the bracketing entropy. 
	
	The main contributions of this paper on a comprehensive study of weak convergence, specifically tightness, of the sequential process and its smoothed version for nonstationary weakly dependent arrays are as follows. Similar to  \cite{VS14}, our Theorem \ref{UniformFCLT}  imposes high-level assumptions on $\mathbb G_n$ that can be verified under different combinations of dependency-restrictions on $(X_{i,n})$ and complexity-conditions on $\mathcal F$. A main tool for its proof is an asymptotically optimal inequality for maximum partial sums of \cite{MSS82}, which we generalize to nonmeasurable dependent processes. Similarly, we provide sufficient conditions for the smoothed sequential process indexed by $ \mathcal{A} \times \mathcal{F} $ for some suitable family $\mathcal A$ of subsets of $[0,1]$. Again, these conditions abstract from the dependency assumptions imposed on $(X_{i,n}) $ and complexity of $ \mathcal A \times \mathcal{F}$, and we apply our results to the case of a strongly mixing array $(X_{i,n})$, a well studied and widely used framework of weak dependence. For the sequential process $\mathbb Z_n$, we obtain extensions of \cite{MO20} and \cite{H05} that are essentially optimal in the sense that our tightness conditions are only marginally stricter than those imposed in these references. Moreover, our extension of \cite[Thm.~3]{H05} allows for exponentially large classes $\mathcal F$ and thereby significantly improves on existing results even for stationary sequences. Compared to \cite{PN25} our result considers the more general case of strong mixing and still allows for exponentially growing classes $ \mathcal{F} $ by requiring $ L_2$-integrable bracketing entropy. For the smoothed process $ \mathbb Z_n^s $, we obtain the to our knowledge first ever results that apply to a class $ \mathcal A$ of intervals. Lastly, we briefly discuss an application to change-point testing based on a class of probability metrics that includes the Wasserstein distances.
	
	The rest of the paper is organized as follows. In Section \ref{Section:UFCLT}, we introduce the theoretical framework, provide sufficient conditions for the weak convergence of $\mathbb Z_n$ under nonstationarity and, in a technical subsection, present an extension of \cite[Thm.~3.1]{MSS82} to nonmeasurable maps. Section \ref{Sec:Smoothed} presents the results for the smoothed process $\mathbb Z_n^s.$ The case of strongly mixing arrays is treated in Section \ref{Section:Examples-NEU}. \ref{Section:Application} gives an application to change-point testing. Lastly, Section \ref{Section:Discussion} provides a discussion and an outlook. All technical proofs are presented in Section \ref{Section:Appendix}.
	
	\textbf{Notation:}	If $A$ is a set, we denote by $\#(A)$ its cardinality and by $ \partial A $ its boundary. $A \triangle B$ denotes the symmetric difference of two sets $A$ and $B$. The extended real line is denoted by $\ER$. Metric spaces $(D,d)$ are endowed with their $d$-Borel $\sigma$-fields denoted by $\mathcal B(D)$ and measurability in a metric space is understood as Borel-measurability. If $D = M \times N$ and $\delta > 0$, we abbreviate $\sup_{(a,b) \in M \times N, d(a,b) \leq \delta}$ by $\sup_{d(a,b) \leq \delta}$ when no confusion can arise. Furthermore, for $k \in \mathbb N, p \in [1,\infty)$ and $x = (x_1,...,x_k) \in \mathbb R^k$, we denote by $|x|_p := (\sum_{i = 1}^k |x_i|^p)^{1/p}$ its $p$-norm and put $|x|_\infty = \max_{1 \leq i \leq k }|x_k|.$ If $ \psi \neq 0 $ is a Young function, i.e. a convex function on $ [0,\infty) $ with $ \psi(0)=0, $ we denote the associated Orlicz norm of a random variable $X$ by $ \| X \|_{L_\psi} $. The choice $\psi(x) = x^p, p\geq 1,$ corresponds to its $L_p$-norm and is denoted by $\norm{X}_{L_p}.$ If two sequences $(a_n), (b_n)$ satisfy $a_n \leq C b_n$ for all $n\in\mathbb N$ and a constant $C \geq 0$, we denote this as $a_n \lesssim b_n.$ The minimum of two real numbers $a,b$ is denoted by $a \land b$, their maximum by $a \lor b.$

	\section{The sequential empirical process}\label{Section:UFCLT}
	The framework of this paper is as follows: throughout, we assume the existence of a finite measurable and integrable function $F: \mathcal X \to \mathbb R$ called ``envelope'' that fulfills $\sup_{f \in \mathcal F} |f(x)| \leq F(x) < \infty$ for all $x \in \mathcal X$ and $\E{F(X_{i,n})} < \infty$ for all $1 \leq i \leq n, n \in \mathbb N.$ This entails that each $X_{i,n}$ induces a map $\Omega \to \ell^{\infty}(\mathcal F)$ which maps $ \omega \in \Omega $ to the function $ f \mapsto f(X_{i,n}(\omega)) - \E{f(X_{i,n})}$. That allows to view $\mathbb G_n$ and $\mathbb Z_n$ as random elements of $\ell^\infty(\mathcal F)$ and $\ell^\infty([0,1] \times \mathcal F)$, respectively, where, for $\Psi \neq \emptyset,$ $$ \ell^\infty(\Psi) = \left\{z: \Psi \to \mathbb R ~\bigg{|}~ \norm{z}_{\Psi} := \sup_{\psi \in \Psi} |z(\psi)| < \infty \right\}, $$ which is endowed with the $\norm{.}_\Psi$-Borel $\sigma$-field. Doing so is customary, but at the same time makes $\mathbb G_n$ and $\mathbb Z_n$ nonmeasurable. We therefore study their weak convergence in the sense of \cite[Def.~1.3.3]{vdVW23} that involves outer expectations and probabilities that will be denoted by $\EOuter{.}$ and $P^*(.)$, respectively. 
	
	It is well known (see \cite[Thm.~1.5.4 and 1.5.7]{vdVW23}) that $\mathbb Z_n$ converges weakly to a tight Borel map $\mathbb Z$ in the latter sense, in symbols $\mathbb Z_n \Rightarrow \mathbb Z$, if and only if the finite-dimensional marginals (fidis) of $\mathbb Z_n$ converge to those of $\mathbb Z$ and there exists a semimetric $\tau$ that makes $([0,1] \times \mathcal F,\tau)$ totally bounded and $\mathbb Z_n$ asymptotically uniformly equicontinuous in probability, i.e. 
	\begin{equation}\label{DEF:AEC}
	\forall ~ \varepsilon > 0: ~ \lim_{\delta \downarrow 0} \limsup_{n \to \infty} P^*\left( \sup_{\tau((s,f),(t,g)) \leq \delta}  |\mathbb Z_n(s,f) - \mathbb Z_n(t,g)| > \varepsilon \right) = 0. \tag{\text{AEC}}
	\end{equation}
	This holds true analogously for all other processes with bounded sample paths to appear in the course of this paper. Of these two conditions, condition \eqref{DEF:AEC} is usually what is more difficult to show, so we follow related work (see, e.g., \cite{AP94}, \cite{VS14} and \cite{MO20}) and focus on that part. Furthermore, there are already some results available from which the fidi-convergence of $\mathbb Z_n$ may be concluded under nonstationarity (see, e.g., \cite{Rio95}, \cite{DRW19}, \cite{Mies23} and \cite{AS2024}).
	
	Theorem \ref{UniformFCLT} below provides  sufficient conditions for \eqref{DEF:AEC} that apply to nonstationary arrays. As it might look somewhat complicated at first glance and its proof is fairly involved, we briefly sketch its underlying idea, first. A common first step towards \eqref{DEF:AEC} is to choose the semimetric $\tau$ as
	\begin{equation}\label{UFCLT-tau}
	\tau((s,f),(t,g)) = |s-f| + \rho(f,g)
	\end{equation}
	for some semimetric $\rho$ on $\mathcal F$, which allows to ``disentangle'' the two parameters of $\mathbb Z_n$ by means of the estimate
	\begin{align}\label{Disentangled}
	&\sup_{\tau((s,f),(t,g)) \leq \delta}  |\mathbb Z_n(s,f) - \mathbb Z_n(t,g)| \nonumber \\ 
	&\leq \sup_{t \in [0,1]} \sup_{\rho(f,g) \leq \delta} \left|\mathbb Z_n(t,f) - \mathbb Z_n(t,g) \right| + \sup_{|s-t| \leq \delta} \sup_{f \in \mathcal F} \left|\mathbb Z_n(t,f) - \mathbb Z_n(s,f) \right|.
	\end{align}
	Denoting $$\mathcal F_\delta = \{f-g ~|~ f,g \in \mathcal F, \rho(f,g) \leq \delta\}, ~ \delta > 0,$$ we may then write 
	\begin{align*}
	&\sup_{t \in [0,1]} \sup_{\rho(f,g) \leq \delta} \left|\mathbb Z_n(t,f) - \mathbb Z_n(t,g) \right| \\
	&= \max_{1 \leq k \leq n} \sup_{f \in \mathcal F_\delta} \left| n^{-\frac12} \sum_{i = 1}^{k} \left(f(X_{i,n}) - \E{f(X_{i,n})} \right) \right| = n^{-\frac12} \max_{1 \leq k \leq n} \norm{S_{n,1,k}}_{\mathcal F_\delta} 
	\end{align*}
	for $$ S_{n,i,j}(f) = \sum_{k = i}^{j} \left(f(X_{k,n}) - \E{f(X_{k,n})} \right), ~ 1 \leq i \leq j \leq n, ~ f \in \mathcal F \cup \bigcup_{\delta > 0} \mathcal F_\delta. $$ The rightmost term of \eqref{Disentangled} can be expressed similarly. It is therefore natural to approach the problem of verifying \eqref{DEF:AEC} by, firstly, applying an inequality to control the maximal partial sums indexed by $\mathcal F_\delta$ and, secondly, proving that the bound arising from such an inequality is of sufficient regularity in terms of $n$ and $\delta$ to imply \eqref{DEF:AEC}. So, informally, one needs to show that $$ n^{-\frac12} \max_{1 \leq k \leq n} \norm{S_{n,1,k}}_{\mathcal F_\delta} \leq g(n,\delta) \to 0, $$ if $n \to \infty$ followed by $\delta \downarrow 0,$ and do so analogously for the other term on the right-hand side of \eqref{Disentangled}. This method of proof has been used for stationary sequences in several works including \cite[Thm~2.12.1]{vdVW23} for i.i.d. data and in \cite[Lem.~2]{BU15} and \cite[Thm.~4.10]{VS14}. To some extent, we adopt this approach, but there are three major obstacles to be overcome in our setting. Firstly, despite the rich literature on maximum partial sums for real random variables (see, e.g, the overview in \cite{WU07} and also \cite{MSS82}), there are only few results for non-Borelian random maps (see \cite[App.~A.1]{vdVW23} and \cite{Zi97b} for independent processes and \cite[Prop.~1.(ii)]{WU07} for stationary ones), and neither of them applies to nonstationary dependent data. Secondly, while some results for real variables might be extendable to non-Borelian maps, the conditions one must verify to apply them might be natural for real variables, but practically infeasible for empirical processes. Lastly, once a bound $g(n,\delta)$ for the maximum partial sums is found, $g(n,\delta)$ should be simple enough in terms of $n$ and $\delta$ to give feasible sufficient conditions for $g(n,\delta) \to 0$ as $n \to \infty$ followed by $\delta \downarrow 0$, thereby resulting in easily applicable sufficient conditions for \eqref{DEF:AEC}. The primary contribution of the following result is to identify from the literature a type of bound for the increments of the empirical process $S_{n,i,j}$  that solves all these problems simultaneously. 

\begin{theorem}(Asymptotic equicontinuity)\\
\label{UniformFCLT}
For $\nu > 2, \kappa \in (0, 1/2 - 1/\nu), C \geq 0$ and finite functions $R,J: (0,\infty) \to [0,\infty)$ let
\begin{equation*}\label{Generic-gn-Equation}
\gamma(m, \delta) = C m \left(R(\delta) + J(\delta) m^{-\kappa} \right)^2, ~ m \in \mathbb N, \delta > 0.
\end{equation*}
Furthermore, let $(\mathcal F, \rho)$ be totally bounded and suppose that for all $\delta > 0$ and $n \in \mathbb N$, it holds
\begin{equation}\label{UniformFCLT-gammabound}
\EOuter{ \norm{ S_{n,i,j} }_{\mathcal F_\delta}^\nu } \leq \gamma^{\frac{\nu}{2}}(j-i+1,\delta), ~ \forall ~ 1 \leq i \leq j \leq n,
\end{equation}
and that there exists $f_0 \in \mathcal F$ with
\begin{equation}\label{UniformFCLT-f0bound}
\norm{S_{n,i,j}(f_0)}_{L_\nu} \leq C \sqrt{j-i+1}, ~ \forall ~ 1 \leq i \leq j \leq n.
\end{equation}  
If $\lim_{\delta \downarrow 0} R(\delta) = 0$, then \eqref{DEF:AEC} holds for $\tau$ from \eqref{UFCLT-tau} and fidi-convergence of $\mathbb Z_n$ implies weak convergence in $\ell^{\infty}([0,1] \times \mathcal F)$. 
\end{theorem}
\begin{remark}
That we need the bounds in \eqref{UniformFCLT-gammabound} to hold for some $\nu > 2$ is the price paid for control over the maximum partial sums instead of just $S_{n,1,n}$ and is a common requirement in the literature on asymptotically optimal bounds for maximum partial sums of dependent variables (see, e.g., \cite{Serf70} and \cite{MSS82}). An analogous condition also appears in \cite[Thm.~4.10]{VS14}. 
\end{remark}

To motivate and discuss the above conditions and relate them to known results, let us start by considering the case of i.i.d. data. If $\mathcal G$ is any set of measurable maps $\mathcal X \to \mathbb R$ with envelope $F$ that satisfies $\norm{F}_{L_2} \leq \delta$, then \cite[Thm.~19.34]{vdV98} gives 
\begin{align*}
&\EOuter{ \norm{ S_{n,1,n} }_{\mathcal G} } \\ 
&\lesssim \sqrt{n} \left( \int_0^\delta \sqrt{ 1 \lor \log N_{[]}(\varepsilon, \mathcal G, \norm{.}_{L_2}) } d\varepsilon + \sqrt{n} \norm{F \mathbbm{1}_{F > \sqrt{n} a(\delta)}}_{L_1}  \right).
\end{align*}
Here, the $L_p$-norms are with respect to $P$, the distribution of the data points, $N_{[]}$ is a bracketing number, and $a(\delta) = \delta / (1 \lor \log(N_{[]}(\delta, \mathcal G, \norm{.}_{L_2})))^{1/2}$. For weakly dependent nonstationary data, similar bounds have been established in \cite[Thm.~4.4]{PR21} for Bernoulli shifts, in \cite[Thm.~3.5]{PN25} for $\beta$-mixing arrays and, at least implicitly, in the proofs of \cite[Thm.~2.2]{AP94}, \cite[Thm.~3]{H05} and \cite[Thm.~2.5]{MO20} for strongly mixing sequences. Now, for i.i.d. data, the fact that the processes $S_{n,i,j}$ and $S_{n,1,j-i+1}$ are identically distributed immediately shows that for any $1 \leq i \leq j \leq n$, with $m = j-i+1$, it also holds that
\begin{align}\label{UniformFCLT_targetbound}
&\EOuter{ \norm{ S_{n,i,j} }_{\mathcal G} } \nonumber \\
&\lesssim \sqrt{m} \left( \int_0^\delta \sqrt{ 1 \lor \log N_{[]}(\varepsilon, \mathcal G, \norm{.}_{L_2}) } d\varepsilon + \sqrt{m} \norm{F \mathbbm{1}_{F > \sqrt{m} a(\delta)}}_{L_1}  \right).
\end{align}
Upon establishing these bounds for the $\nu$-th moment instead of the first one and proving that the Lindeberg-type remainder term is of the form $J(\delta) m^{-\kappa}$ (which follows from a moment condition on $F$ and H\"older's inequality, for instance), this bound is now of the form required by \eqref{UniformFCLT-gammabound}. Here, $R(\delta)$ is just the bracketing integral, thus simple bracketing conditions entail $R(1) < \infty$ and therefore $R(\delta) \to 0$ for $\delta \downarrow 0$ and $J(\delta) < \infty$ for any $\delta > 0.$ In this case, $(\mathcal F, \norm{.}_{L_2})$ is totally bounded and \eqref{UniformFCLT-f0bound} holds trivially, hence all conditions of Theorem \ref{UniformFCLT} are met.

For univariate nonstationary time series, it is possible to derive bounds for the $S_{n,i,j}$ under different dependency restrictions, including martingale differences and several notions of mixing (\cite{Serf70}). As these techniques also underlie the proofs of the existing results for the sums $S_{n,1,n}$ cited above, one can expect bounds of the form \eqref{UniformFCLT_targetbound} to be feasible in those settings as well. We therefore anticipate that Theorem \ref{UniformFCLT} is applicable to a wide range of situations. 

Lastly, it is worth mentioning that bounds similar to\eqref{UniformFCLT_targetbound} can also be derived without referring to bracketing numbers or entropy conditions. For unit balls of (generalized) Lipschitz functions, \cite[Thm.~8.1]{Rio17} shows
$$ \EOuter{ \sup_{\norm{ f-g }_{L_2} \leq \delta} \left| S_{n,1,n}(f) - S_{n,1,n}(g) \right|^2 } \lesssim n \delta^{2(1-\theta)} $$ for stationary sequences with summable $\alpha$-coefficients. Here, $\theta \in (0,1)$ is a constant that depends on the regularity and domain of the functions $f$. Clearly, the above bound is of the form required by \eqref{UniformFCLT-gammabound} if we take $R(\delta) = \delta^{1-\theta} \to 0$, if $\delta \downarrow 0,$ and $J(\delta) = 0.$ 

\subsection{Outline of the proof of Theorem \ref{UniformFCLT}}
Theorem \ref{UniformFCLT} is a consequence of a more general result on the asymptotic equicontinuity of sequential processes (see Theorem \ref{UniformFCLT-general} below) and a sequence of technical results. The first step towards its proof is to extend \cite[Thm.~3.1]{MSS82} to non-Borelian maps into the space $\ell^\infty(\Psi)$ for an arbitrary $\Psi \neq \emptyset$, which might be of independent interest. To state this result concisely, we introduce some additional notation. Let $\Psi \neq \emptyset$ and let $W_1,...,W_n: \Omega \to \ell^\infty(\Psi)$ be arbitrary processes. To avoid confusion with our preceeding notation, for $1 \leq i \leq j \leq n$ and $\psi \in \Psi$, we denote
\begin{align}\label{DEF:Sum_Max}
S_{i,j}^W(\psi) = \sum_{k = i}^{j} \left( W_k(\psi) - \E{W_k(\psi)} \right), ~~~ M_{i,j}^W(\psi) = \max_{k = i,...,j} \left|S_{i,k}^W(\psi) \right|,
\end{align}
and define $S_{i,j}^W(\psi) = 0 = M_{i,j}^W(\psi)$ for $j < i$. The quantity $M_{n,i,j}$ is defined analogously from $S_{n,i,j}$. Furthermore, for $\alpha > 1$ and a function $q: \mathbb N \to \mathbb R$, we say that the pair $(\alpha, q)$ fulfills condition \eqref{MSS-ConditionM} with index $Q \in [1,2^{(\alpha-1)/\alpha})$, if 
\begin{align}\label{MSS-ConditionM}
\mathrm{(i)}&:~ q \geq 0, \nonumber\\
\mathrm{(ii)}&:~ q \text{ is nondecreasing, } \tag{\text{S}} \\
\mathrm{(iii)}&:~ \text{for each } 1 \leq i < j, ~  q(i) + q(j-i) \leq Qq(j). \nonumber
\end{align} 
The conditions (i)-(iii) are essentially adopted from \cite{MSS82} (c.f. conditions (1.2a)-(1.2c) therein). The following result partially extends \cite[Thm.~3.1]{MSS82} to non-Borelian maps and is used in the proof of Theorem \ref{UniformFCLT}. A more general version can be found in Proposition \ref{MSSAdaption_general}.

\begin{proposition} (Maximum partial sums of processes)\\
\label{MSSAdaption}
Let $\nu \geq 1$, $\Psi \neq \emptyset$ and $W_1,...,W_n: \Omega \to \ell^\infty(\Psi)$ be arbitrary processes. If, for all $1 \leq i \leq j \leq n,$ it holds 
\begin{equation}\label{SumboundMSS}
\mathrm E^*\left\{ \norm{ S_{i,j}^W}_{\Psi}^\nu \right\} \leq q^{\alpha}(j-i+1)
\end{equation}
for a pair $(\alpha, q)$ that fulfills condition \eqref{MSS-ConditionM} with index $Q \in [1,2^{(\alpha-1)/\alpha})$ , then there exists a constant $A$ that only depends on $\alpha, \nu,$ and $Q$ for which we have
\begin{equation}\label{MaxBoundMSS}
\mathrm E^* \left\{ \norm{ M_{i,j}^W }_\Psi^\nu \right\} \leq A q^{\alpha}(j-i+1)
\end{equation}
for any $1 \leq i \leq j \leq n$.
\end{proposition}

We can now state and prove 
\begin{theorem}(Asymptotic equicontinuity - general case)\\\label{UniformFCLT-general}
Let $\nu > 2$ and let $\rho$ be a semimetric on $\mathcal F$. Assume the following conditions to hold: 	
\begin{enumerate}[ref = \arabic*]
\item \label{item:Uniform_FCLT_Cond1} There exists $g: \mathbb N \times (0, \infty) \to \mathbb R$ such that for each $n \in \mathbb N$, $\delta > 0$, it holds 
\begin{equation}\label{UniformFCLT_Cond1_Eq}
\EOuter{ \norm{  S_{n,i,j} }_{\mathcal F_\delta}^\nu } \leq g^{\frac{\nu}{2}}(j-i+1, \delta), ~ \forall ~ 1 \leq i \leq j \leq n,
\end{equation}
and there is a universal index $Q_g \in [1,2^{1-2/\nu})$ such that for each $\delta > 0,$ the pair $(\nu/2, g(.,\delta))$ fulfills condition \eqref{MSS-ConditionM} with index $Q_g$. 

\item \label{item:Uniform_FCLT_Cond2} There exists $h: \mathbb N \to \mathbb R$ such that for each $n \in \mathbb N$, it holds 
\begin{equation}\label{UniformFCLT_Cond2_Eq}
\EOuter{ \norm{ S_{n,i,j} }_{\mathcal F}^\nu } \leq h^{\frac{\nu}{2}}(j-i+1), ~ \forall ~ 1 \leq i \leq j \leq n,
\end{equation}
and the pair $(\nu/2, h)$ fulfills condition \eqref{MSS-ConditionM} with index $Q_h \in [1,2^{1-2/\nu})$.

\item \label{item:Uniform_FCLT_Cond3} It holds $$\lim_{\delta \downarrow 0} \limsup_{n \to \infty} \frac{g(n,\delta)}{n} = 0.$$	

\item \label{item:Uniform_FCLT_Cond4} There exists a constant $C_h \geq 0$ such that for each $0 < \varepsilon  \leq 1,$ it holds $$ \limsup_{n \to \infty}  \frac{h(\floor{n \varepsilon})}{n}  \leq C_h \varepsilon $$ (with the convention $h(0) = 0$). 
\end{enumerate}
Then \eqref{DEF:AEC} holds for $\tau((s,f),(t,g)) = |s-t| + \rho(f,g).$	If, in addition, $(\mathcal F, \rho)$ is totally bounded, then fidi-convergence of $\mathbb Z_n$ implies weak convergence in $\ell^{\infty}([0,1] \times \mathcal F)$.
\end{theorem}

\begin{proof}[Proof of Theorem \ref{UniformFCLT-general}]
We begin as in the proof of \cite[Thm.~2.12.1]{vdVW23}. By the triangle inequality, for any $n \in \mathbb N$ and $\delta > 0,$
\begin{align}\label{TwoSuprema}
&\sup_{|s-t| + \rho(f,g) \leq \delta} |\mathbb Z_n(s,f) - \mathbb Z_n(t,g)| \nonumber \\
&\quad \leq \sup_{|s-t| \leq \delta} \sup_{f \in \mathcal F} |\mathbb Z_n(s,f) - \mathbb Z_n(t,f)| + \sup_{t \in [0,1]} \sup_{\rho(f,g) \leq \delta} |\mathbb Z_n(t,f) - \mathbb Z_n(t,g)|.
\end{align}
We estimate both terms with the aid of Proposition \ref{MSSAdaption}. Regarding the left term on the right-hand side of  \eqref{TwoSuprema}, it suffices to show that for any $\varepsilon>0$,
\begin{equation*}
\limsup_{n \to \infty} \mathrm P^* \left(\max_{j \in \mathbb N_0, 0 \leq j\delta \leq 1} \sup_{s \in [j\delta, (j+1)\delta]} \sup_{f \in \mathcal F} |\mathbb Z_n(s,f) - \mathbb Z_n(j\delta, f)| > \varepsilon \right) \to 0,
\end{equation*}
as $\delta \downarrow 0.$ So, let $\varepsilon > 0$, $n \in \mathbb N$ and $0 < \delta \leq 1/2$. By a union bound and since $0 \leq j\delta \leq 1$ entails $0 \leq j \leq \ceil{\delta^{-1}}$, we have
\begin{align}\label{maxsupsup-bound}
&\mathrm P^* \left(\max_{j \in \mathbb N_0, 0 \leq j\delta \leq 1} \sup_{s \in [j\delta, (j+1)\delta]} \sup_{f \in \mathcal F} |\mathbb Z_n(s,f) - \mathbb Z_n(j\delta, f)| > \varepsilon \right) \nonumber \\ 
&\leq \sum_{j = 0}^{\ceil{\delta^{-1}}} \mathrm P^* \left(\sup_{s \in [j\delta, (j+1)\delta]} \sup_{f \in \mathcal F} |\mathbb Z_n(s,f) - \mathbb Z_n(j\delta, f)| > \varepsilon \right).
\end{align}	
Now, for each $j = 0,...,\ceil{\delta^{-1}}$ and using that  $\floor{n(j+1)\delta} \leq \floor{nj\delta} + \floor{n\delta} + 1$, 
\begin{align*}
&\sup_{s \in [j\delta, (j+1)\delta]} \sup_{f \in \mathcal F} |\mathbb Z_n(s,f) - \mathbb Z_n(j\delta, f)|\\
&= \frac{1}{\sqrt{n}} \sup_{s \in [j\delta, (j+1)\delta]} \sup_{f \in \mathcal F} \left|\sum_{i = \floor{nj\delta} + 1}^{\floor{ns}} \left( f(X_{i,n}) - \E{f(X_{i,n})} \right) \right|\\
&\leq \frac{1}{\sqrt{n}} \max_{k = 1,...,\floor{n\delta} + 1} \norm{ S_{n,\floor{nj\delta} + 1, \floor{nj\delta} + k } }_{\mathcal F} \\
&= \frac{1}{\sqrt{n}} \norm{ M_{n,a(j) + 1, a(j) + \floor{n\delta} + 1}}_{\mathcal F},
\end{align*}
where $a(j) := \floor{nj\delta} \in \mathbb N_0$. Hence, by Markov's inequality, condition \ref{item:Uniform_FCLT_Cond2} and Proposition \ref{MSSAdaption}, we obtain
\begin{align*}
&\mathrm P^*\left(\sup_{s \in [j\delta, (j+1)\delta]} \sup_{f \in \mathcal F} |\mathbb Z_n(s,f) - \mathbb Z_n(j\delta, f)| > \varepsilon \right) \\
&\leq \mathrm P^*\left(\norm{ M_{n,a(j) + 1, a(j) + \floor{n\delta} + 1}}_{\mathcal F} > \sqrt{n} \varepsilon \right)\\ 
&\leq \left(\varepsilon \sqrt{n}\right)^{-\nu} \mathrm E^*\left\{\norm{ M_{n,a(j) + 1, a(j) + \floor{n\delta} + 1}}_{\mathcal F}^\nu\right\} \\
&\leq \left(\varepsilon \sqrt{n}\right)^{-\nu} A h^{\frac{\nu}{2}} (\floor{n\delta} + 1)
\end{align*}
for a constant $A$ that only depends on $\nu$ and $Q_h.$ Since $(\nu/2, h)$ fulfills condition \eqref{MSS-ConditionM}, $h$ is nondecreasing, and so, as $1 \leq \floor{n\delta}$ holds for all $n$ large enough and $2 \floor{x} \leq \floor{2x}$ for all $x \geq 0$, we conclude that
\begin{align*}
\limsup_{n \to \infty} \left(\varepsilon \sqrt{n}\right)^{-\nu} A h^{\frac{\nu}{2}} (\floor{n\delta} + 1) &\leq \limsup_{n \to \infty} \left(\varepsilon \sqrt{n}\right)^{-\nu} A h^{\frac{\nu}{2}} (2 \floor{n\delta}) \\ 
&\leq A \varepsilon^{-\nu} \limsup_{n \to \infty} \frac{h^{\frac{\nu}{2}} (\floor{2n\delta})}{n^{\frac{\nu}{2}}} \\ 		
&\leq A \varepsilon^{-\nu} C^\frac{\nu}{2} 2^{\frac{\nu}{2}} \delta^{\frac{\nu}{2}},
\end{align*}
where in the last step we applied condition \ref{item:Uniform_FCLT_Cond4} combined with  $2\delta \leq 1$. By inserting the latter bound into \eqref{maxsupsup-bound} and using that $\nu/2 > 1$, we thus obtain
\begin{align*}
&\limsup_{n \to \infty} \mathrm P^* \left(\max_{j \in \mathbb N_0, 0 \leq j\delta \leq 1} \sup_{s \in [j\delta, (j+1)\delta]} \sup_{f \in \mathcal F} |\mathbb Z_n(s,f) - \mathbb Z_n(j\delta, f)| > \varepsilon \right) \\
&\leq \sum_{j = 0}^{\ceil{\delta^{-1}}} \limsup_{n \to \infty} \mathrm P^* \left(\sup_{s \in [j\delta, (j+1)\delta]} \sup_{f \in \mathcal F} |\mathbb Z_n(s,f) - \mathbb Z_n(j\delta, f)| > \varepsilon \right) \\
&\leq \left(\ceil{\delta^{-1}} + 1\right) A \varepsilon^{-\nu} (2C)^{\frac{\nu}{2}} \delta^{\frac{\nu}{2}} \leq 3 \delta^{\nu/2 - 1} A \varepsilon^{-\nu} (2C)^{\frac{\nu}{2}}  \to 0,
\end{align*}
as $\delta \downarrow 0.$

It therefore remains to discuss the right term on the right-hand side of \eqref{TwoSuprema}. So, let again $n \in \mathbb N$ and $\varepsilon, \delta > 0$, then by linearity,
\begin{align*}
&\sup_{t \in [0,1]} \sup_{\rho(f,g) \leq \delta} |\mathbb Z_n(t,f) - \mathbb Z_n(t,g)| \\
&= \sup_{t \in [0,1]} \sup_{\rho(f,g) \leq \delta} |\mathbb Z_n(t,f-g)| \\
&= \frac{1}{\sqrt{n}} \max_{k = 1,...,n} \sup_{\tilde f \in \mathcal F_\delta} \left|\sum_{i = 1}^k \left( \tilde f(X_{i,n}) - \E{\tilde f(X_{i,n})} \right) \right|\\ 
&= \frac{1}{\sqrt{n}} \norm{ M_{n,1, n}}_{\mathcal F_\delta}.
\end{align*}
Hence, by Markov's inequality, condition \ref{item:Uniform_FCLT_Cond1} and Proposition \ref{MSSAdaption}, we obtain
\begin{align*}
&\mathrm P^*\left( \sup_{t \in [0,1]} \sup_{\rho(f,g) \leq \delta} |\mathbb Z_n(t,f) - \mathbb Z_n(t,g)| > \varepsilon \right) \\
&\leq (\varepsilon \sqrt{n})^{-\nu} \mathrm E^*\left\{ \norm{ M_{n,1,n}}_{\mathcal F_\delta}^\nu\right\} \\
&\leq (\varepsilon \sqrt{n})^{-\nu} B g^{\frac{\nu}{2}}(n,\delta)
\end{align*}
for a constant $B$ that only depends on $\nu$ and $Q_g$. In particular, since $Q_g$ does not depend on $\delta$, so does $B$. Hence, by making use of condition \ref{item:Uniform_FCLT_Cond3}, we conclude that
\begin{align*}
&\lim_{\delta \downarrow 0} \limsup_{n \to \infty} \mathrm P^*\left( \sup_{t \in [0,1]} \sup_{\rho(f,g) \leq \delta} \left|\mathbb Z_n(t,f) - \mathbb Z_n(t,g) \right| > \varepsilon \right) \\ 
&\leq \varepsilon^{-\nu} B \lim_{\delta \downarrow 0} \limsup_{n \to \infty} \left(  \frac{g(n,\delta)}{n} \right)^{\frac{\nu}{2}} = 0.
\end{align*}
In view of \eqref{TwoSuprema}, we have thus shown \eqref{DEF:AEC}.

Finally, if $(\mathcal F, \rho)$ is totally bounded, then so is $([0,1] \times \mathcal F, \tau)$, and hence the weak convergence of $\mathbb Z_n$ follows from \cite[Thm.~1.5.4 and 1.5.7]{vdVW23}.
\end{proof}

To prove Theorem \ref{UniformFCLT} it now suffices to show that its conditions imply those of Theorem \ref{UniformFCLT-general}. This is accomplished by proving that the function $\gamma$ from the former allows to construct $g$ and $h$ for the latter (see Lemma \ref{Generic-gn}). The technical details are postponed to the appendix.

\section{The smoothed sequential empirical process}\label{Sec:Smoothed}
Let us now consider the smoothed version of $ \mathbb{Z}_n$ given by
\[
\mathbb{Z}_n^s(A,f) = \frac{1}{\sqrt{n}} \sum_{i=1}^n \lambda( R_i \cap n A ) ( f(X_{i,n}) - \E{f(X_{i,n})})
\]
for $ A \in \mathcal{A} $ and $ f \in \mathcal{F} $, where $ \mathcal{A} $ is a family of regular Borel sets of $ [0,1], $ $ R_i = (i-1,i] $, $ i \ge 1$, and $ \lambda $ denotes  the Lebesgue measure. Recall that a Borel set $A$ is called regular, if $ \lambda(\partial A) = 0 $. Specifically, the choice  $ \mathcal{A} = \mathcal A_{(0, \cdot]} = \{ (0, t] : t \in [0,1] \} $, which is a Vapnik-\v{C}ervonenkis (VC) class, corresponds to the interpolated version of $ \mathbb{Z}_n(t,f)$ as defined in \eqref{DEF:SEP-SMOOTH}, where $\mathbb Z_n^s((0,t],f) = \mathbb Z_n^s(t,f)$ . 

The smoothed sequential empirical process $ \mathbb{Z}_n^s $ is indexed by $ \mathcal{A} \times \mathcal{F} $, which becomes  a semimetric space when equipped with either the semimetric
\begin{equation}\label{DEF: Semimetric_smoothed}
\tau( (A,f), (B,g) ) = \lambda( A \triangle B ) + \rho(f,g), \qquad (A,f), (B,g) \in \mathcal{A} \times \mathcal{F},
\end{equation}
or
\begin{equation}\label{DEF: Semimetric_smoothed_sqrt}
\tau_s( (A,f), (B,g) ) = \sqrt{\lambda( A \triangle B )} + \rho(f,g), \qquad (A,f), (B,g) \in \mathcal{A} \times \mathcal{F}.
\end{equation}
We will use both of them in what follows, depending on the problem at hand. Note that for $ \mathcal A = \mathcal A_{(0, \cdot]} $  the metric $ \tau $ coincides with the  semimetric from Theorem \ref{UniformFCLT}. As in Section \ref{Section:UFCLT}, the existence of an envelope $F$ allows to view $\mathbb Z_n^s$ as a map $\Omega \to \ell^\infty(\mathcal A \times \mathcal F)$ and ensures $\mathbb Z_n^s(A,f) \in L_1(P)$ for all $(A,f) \in \mathcal A \times \mathcal F$. Again, we are interested in its asymptotic tightness and weak convergence and thus seek sufficient conditions for \eqref{DEF:AEC}. 

It is intuitively clear that $\mathbb Z_n^s$ cannot be a tight map into $\ell^\infty(\mathcal A \times \mathcal F)$ if its indexing set is arbitrarily complex, or large. Therefore, the two results of this Section, Theorem \ref{UFCLT_smoothed_Lip} and \ref{UFCLT_smoothed}, consider two different kinds of scenarios. The first result restricts the complexity of $ \mathcal{A} \times \mathcal{F} $ in terms of its covering numbers and assumes separability of $\mathbb Z_n^s$. The second result, Theorem \ref{UFCLT_smoothed}, considers  the special case that $\mathcal A$ is a set of intervals and shows that, essentially under the conditions of Theorem \ref{UniformFCLT}, $\mathbb Z_n^s$ is tight in $\ell^\infty(\mathcal A \times \mathcal F)$ as soon as $\mathbb Z_n$ is in $\ell^\infty([0,1] \times \mathcal F)$.

In order to present our first result, note that smoothing increases the regularity of the process, and, depending on the choice of $ \mathcal{F} $, $\mathbb Z_n^s$ can be measurable with continuous paths. It turns out that in this case, one can essentially rely on results on the regularity of random processes indexed by metric spaces to verify \eqref{DEF:AEC}, provided the entropy of the indexing space can be suitably controlled. Specifically, Theorem~\ref{UFCLT_smoothed_Lip} below assumes that the  $\tau_s$-covering numbers, $ N(\mathcal A \times \mathcal{F}, \tau_s, \varepsilon )$, of $ \mathcal A \times \mathcal F $ are such that $ \psi^{-1} \circ N(\mathcal A \times \mathcal{F}, \tau_s, \varepsilon )$ is integrable over $ \varepsilon \in [0, \Delta_s] $, where $ \Delta_s $ is the $\tau_s$-diameter of $\mathcal A \times \mathcal F $ and $ \psi $ is a Young function satisfying  some weak regularity conditions given below. Recall that  $  N(\mathcal A \times \mathcal{F}, \tau_s, \varepsilon ) $ is the number of open $ \varepsilon $-balls with respect to the semimetric $\tau_s $ needed to cover $\mathcal A \times \mathcal{F}$. It is well known that a VC-class of sets is a polynomial class with respect to any $ L_r(Q) $-norm and any probability measure $Q$ (see \cite[Thm. 2.6.4]{vdVW23}). Analogously, if $ \mathcal F $ is a VC-subgraph class, i.e., if the family of sets $ \{ (x,t) ~|~ t < f(x)\} $ is a VC-class, and possesses a measurable envelope, then the covering numbers of $ \mathcal F $ are polynomial with respect to any $ L_r(Q) $-norm and any probability measure $Q$ (see \cite[Thm. 2.6.7]{vdVW23}). Note that if both $\mathcal A$ and $\mathcal F$ are polynomial classes, then so is $ \mathcal A \times \mathcal F $. But much larger classes with exponential growth arise in application, and here our results provide sufficient conditions when the assumptions hold for an exponential Young function such as $ \psi(x) = \exp(x^p) - 1$.

The following Theorem \ref{UFCLT_smoothed_Lip} provides sufficient conditions for \eqref{DEF:AEC} in terms of $ L_\psi $-regularity under the assumption that $\mathbb Z_n^s$ is a separable process. Especially, for $\psi(x) = x^p$, it shows that if $ \mathcal A \times \mathcal F $ is a polynomial class, then the mild condition of $L_p$-Lipschitz continuity of $\mathbb Z_n^s$ for a suitable $p > 1$ guarantees \eqref{DEF:AEC}. For the general approach we refer to \cite{LT2002}, and to \cite{MVW2013} and \cite{AS2024} for related results  concerning the asymptotic tightness and weak convergence to a Brownian motion, respectively, of the set-indexed process for stationary and nonstationary random fields. 

\begin{theorem}(Asymptotic equicontinuity under Lipschitz condition)\\
\label{UFCLT_smoothed_Lip}
Let $ \psi $ be a  nondecreasing Young function with $ \limsup_{x,y \to \infty} \psi(x) \psi(y) / \psi(c x y) < \infty $ for some constant $c$ and $ \| \cdot \|_{L_1} \le \| \cdot \|_{L_\psi} $. Assume that $\mathbb Z_n^s(A,f) $, $ (A,f) \in \mathcal A \times \mathcal F $, is separable for all $n$, and suppose that the following assumptions hold.
\begin{itemize}
\item[(i)] There is some constant $ C_1 $ such that for all $ A, B \in \mathcal A $ and $n \in \mathbb N,$
\[
\sup_{f \in \mathcal F} \| \mathbb{Z}_n^s(A,f) - \mathbb{Z}_n^s(B,f) \|_{L_\psi} \le C_1 \sqrt{\lambda(A \triangle B )}.
\]
\item[(ii)] There is some constant $C_2 $ such that for all $f, g \in \mathcal F $ and $n \in \mathbb N,$
\[
\sup_{A \in \mathcal A} \| \mathbb{Z}_n^s(A,f) - \mathbb{Z}_n^s(A,g) \|_{L_\psi} \le C_2 \rho(f,g).
\]
\item[(iii)] The covering numbers of $ \mathcal A \times \mathcal{F} $ satisfy 
\[
\int_0^{\Delta_s} \psi^{-1}( N(\mathcal A \times \mathcal{F}, \tau_s, \varepsilon/2 )) \, d \varepsilon < \infty.  
\]
\end{itemize}
Then for any $ \varepsilon > 0 $ there exists $ \eta = \eta( \varepsilon ) > 0 $ such that
\[
\E{ \sup_{\tau_s((A,f),(B,g)) < \eta} | {\mathbb{Z}}_n^s(A,f) - {\mathbb{Z}}_n^s( B,g) |} < \varepsilon
\]
which implies \eqref{DEF:AEC}, such that fidi-convergence of $\mathbb Z_n^s$ implies its weak convergence in $\ell^\infty(\mathcal A \times \mathcal F)$.

In particular, if $ \mathcal A \times \mathcal{F} $ is a polynomial class, i.e., $ N(\mathcal A \times \mathcal{F}, \tau_s, \varepsilon ) = O( \varepsilon^{-a} ) $, for some $a>0$, and $ \psi(x) = x^p$ with $ p > \max(a,1) $, then (iii) holds and there exists a continuous version $ \tilde{\mathbb Z}_n^s $ of $ \mathbb{Z}_n^s $ with almost all sample paths lying in the space $ \mathcal C_u(\mathcal A \times \mathcal F; \mathbb R ) $ of bounded and uniformly continuous functions.
\end{theorem}

\begin{remark}
\begin{enumerate}
\item[(a)] For any $p \geq 1$, the assumptions on $\psi$ stated in Theorem \ref{UFCLT_smoothed_Lip} are fulfilled for both the maps $x \mapsto x^p$ and $x \mapsto \psi_p(x) = \exp(x^p) - 1$, $x \geq 0,$ inducing the $L_p$-norm and the exponential Orlicz norm. Especially, $$ \norm{\cdot}_{L_1} \leq \norm{\cdot}_{L_p} \leq \norm{\cdot}_{L_{\psi_p}} $$
see \cite[p. 145]{vdVW23}. 
\item[(b)] It suffices to understand separability of $ \mathbb Z_n^s(A,f) $ in the sense that \[ \sup_{\tau_s((A,f),(B,g)) < \eta} | \mathbb{Z}_n^s(A,f) - \mathbb{Z}_n^s( B,g) | \] remains a.s. invariant, if $ \mathcal{A} \times \mathcal{F} $ is replaced by a suitable countable subset. We refer to \cite[p. 179]{vdVW23} for a precise definition. 
\item[(c)]  Note that bounding the increments $ \| \mathbb{Z}_n^s(A,f) - \mathbb{Z}^s_n(B,f) \|_{L_\psi}  $ means bounding
\begin{align*}
\left\|\frac{1}{\sqrt{n}} \sum_{i=1}^n [\lambda( R_i \cap n A ) - \lambda( R_i \cap n B ) ] ( f(X_{i,n}) - \E{f(X_{i,n})})\right\|_{L_\psi}. 
\end{align*}
If $\psi(x) = x^p$, this can be achieved by suitable uniform moment bounds for weighted sums (see, e.g., \cite[Thm. 3.1 and Lem. 4.3]{GG86} for mixing series and \cite[Lem.~2]{a_AS2024} under physical dependence). {An analogous result for $\psi_2 = \exp(x^2) - 1$ has been established in \cite[Thm.~2.8]{KM25} under a subgaussian version of the physical dependence measure.} 
\end{enumerate}
\end{remark}

If $\mathcal A = \mathcal A_{(u,v]} = \{(u,v] ~|~ 0 \leq u \leq v \leq 1 \}$, some calculations reveal that the conditions of Theorem \ref{UniformFCLT} imply that $\mathbb Z_n^s$ fulfills \eqref{DEF:AEC} and the separability condition of Theorem \ref{UFCLT_smoothed_Lip} is then superfluous:

\begin{theorem}(Asymptotic equicontinuity for $\mathcal A_{(u,v]}$)\\
\label{UFCLT_smoothed}
If all conditions of Theorem \ref{UniformFCLT} are met and thus \eqref{DEF:AEC} holds for the sequential process $\mathbb Z_n$ indexed by $\mathcal F$, then \eqref{DEF:AEC} also holds for the smoothed process $\mathbb Z_n^s(A,f), (A,f) \in \mathcal A_{(u,v]} \times \mathcal F,$ for $\tau$ from \eqref{DEF: Semimetric_smoothed} and fidi-convergence of $\mathbb Z_n^s$ implies weak convergence in $\ell^\infty(\mathcal A_{(u,v]} \times \mathcal F).$ 
\end{theorem}

\begin{remark}\label{RemarkOnUFCLT_Smoothed}
It should be possible to extend the statement of Theorem \ref{UFCLT_smoothed} to classes of sets that consist of simple combinations of intervals $(u,v],$ $0 \leq u \leq v \leq 1.$ For instance, let $A \subset [0,1]$ be representable as $$ A = \bigcup_{j = 1}^K A_j, $$ where the $A_j \subset [0,1]$ are pairwise disjoint, each $A_j$ is either an interval of the form $(u,v], 0 \leq u \leq v \leq 1,$ or the empty set, and $K \in \mathbb N$ is arbitrary, but fixed. Denote the set of all such sets by $\mathcal A_K.$ By the additivity of the Lebesgue measure, for any $(A,f)$, it holds $$ \mathbb Z_n^s(A, f) = \sum_{j = 1}^K \mathbb Z_n^s(A_j, f), ~ \text{ where } A = \bigcup_{j = 1}^K A_j, $$ from which it should be possible to obtain \eqref{DEF:AEC} for $\mathbb Z_n^s$ (indexed by $\mathcal A_K \times \mathcal F)$ by using arguments similar to those used to prove Theorem \ref{UFCLT_smoothed}.
\end{remark}

\subsection{Proofs of Theorems \ref{UFCLT_smoothed_Lip} and \ref{UFCLT_smoothed}}
\begin{proof}[Proof of Theorem~\ref{UFCLT_smoothed_Lip}]
Let $ \psi $ be a Young function with $$ \limsup_{x,y \to \infty} \psi(x) \psi(y) / \psi(c x y) < \infty $$ for some constant $c$, and $ \| \cdot \|_{L_1} \le \| \cdot \|_{L_\psi} $. The first step consists in showing that the process is Lipschitz in $ L_{\psi} $. This means, for all $ (A,f), (B,g) \in \mathcal{A} \times \mathcal{F}, $  
\begin{equation}
\label{LipLp}
\| \mathbb{Z}_n^s(A,f) - \mathbb{Z}_n^s(B,g) \|_{L_\psi} \le C \tau_s( (A,f), (B, g) )
\end{equation}
for some constant $ C $. By the triangle inequality, we decompose the increment in increments with respect to each argument, 
\begin{align*}
\| \mathbb{Z}_n^s(A,f) - \mathbb{Z}_n^s(B,g) \|_{L_\psi} 
&\le \sup_{f \in \mathcal{F}} \| \mathbb{Z}_n^s(A,f) - \mathbb{Z}_n^s(B,f) \|_{L_\psi} \\
& \qquad + \sup_{B \in \mathcal{A}} \| \mathbb{Z}_n^s(B,f) - \mathbb{Z}_n^s(B,g) \|_{L_\psi} \\
& \qquad = I_n(A,B) + II_n(f,g).
\end{align*}
By (i), the first term can be bounded by
\begin{equation}
\label{SuffCond1}
I_n(A,B) \le C_1 \sqrt{\lambda(A \triangle B)},
\end{equation}
and (ii) entails that
\begin{equation}
\label{SuffCond2}
II_n(f,g) \le C_2 \rho(f,g).
\end{equation}
Consequently,
\begin{align*}
\| \mathbb{Z}_n^s(A,f) - \mathbb{Z}_n^s(B,g) \|_{L_\psi} 
&\le C ( \sqrt{\lambda(A \triangle B)} + \rho(f,g)) \\
& = C \tau_s( (A,f), (B,g) )
\end{align*}
for $ C=\max(C_1,C_2) $,  which verifies \eqref{LipLp}. By \cite[Thm. 2.2.4]{vdVW23}, for all $n$ and $\eta, \delta > 0$, 
\begin{align*}
\left\| \sup_{ \tau_s((A,f),(B,g)) < \delta} |\mathbb{Z}_n^s(A,f) - \mathbb{Z}_n^s(B,g) | \right\|_{L_\psi} &\leq K \int_0^\eta \psi^{-1}\left(  N(\mathcal{A} \times \mathcal{F}, \tau_s, \varepsilon/2) \right) \, d \varepsilon \\
&  + K \delta 
\psi^{-1}\left( N^2(\mathcal{A} \times \mathcal{F}, \tau_s, \eta/2) \right),
\end{align*}
for some constant $ K = K(\psi,C) $. Clearly, by (iii), the right side is less than any $\varepsilon > 0$ for all sufficiently small $\delta, \eta > 0$, uniformly in $n,$ from which the assertion follows in view of $\| \cdot \|_{L_1} \le \| \cdot \|_{L_\psi}. $ An application of Markov's inequality now shows that \eqref{DEF:AEC} holds true. 

To show the second assertion, first note that \[
\int_0^{\Delta_s} \psi^{-1}( N(\mathcal A \times \mathcal{F}, \tau_s, \varepsilon/2 )) \, d \varepsilon =  O \left( \int_0^{\Delta_s} \varepsilon^{-a/p} \, d \varepsilon \right) ,  
\] and thus the integral in (iii) is finite provided $ p> a. $ Next, we show that for any measurable set $ E \subset \Omega $,
\begin{equation}
\label{SuffCondLT1}
\int_E | \mathbb{Z}_n^s(A,f) - \mathbb{Z}_n^s(B,g) | \, d {\rm P}
\le \tau_s( (A,f), (B,g) ) \Prob{E} \psi^{-1}\left( \frac{1}{\Prob{E}} \right)
\end{equation}
for $ \psi(x) = x^p $, $ x \ge 0 $,  $p > 1 $. We can argue as in \cite{MVW2013} and \cite{AS2024}. This follows by H\"older's inequality with $ q=p/(p-1) $, since
\begin{align*}
\int_E | \mathbb{Z}_n^s(A,f) - \mathbb{Z}_n^s(B,g) | \, d \rm P
& \le \| \mathbb{Z}_n^s(A,f) - \mathbb{Z}_n^s(B,g) \|_{L_p} \left( \int_E \, d \rm P \right)^{1/q} \\
&\le \| \mathbb{Z}_n^s(A,f) - \mathbb{Z}_n^s(B,g) \|_{L_p} \left(\frac{1}{\Prob{E}} \right)^{-1/q} \\
& =  \| \mathbb{Z}_n^s(A,f) - \mathbb{Z}_n^s(B,g) \|_{L_p} \Prob{E} \left(\frac{1}{\Prob{E}} \right)^{1/p} \\
& \le C \tau_s( (A,f), (B,g) ) \Prob{E} \psi^{-1}\left( \frac{1}{\Prob{E}} \right),  
\end{align*}
where the constant can be absorbed into the pseudo-metric $ \tau_s$. Combined with $p > a$, this leads to the condition $p > \max(a,1).$ 
\cite[Thm.~11.6]{LT2002} now ensures the existence of a continuous version $ \tilde{\mathbb{Z}}_n^s $ with almost all sample paths bounded and uniformly continuous, i.e., in the space $ \mathcal{C}_u( \mathcal{A} \times \mathcal{F}; \mathbb{R} )$.
\end{proof}

Similar to Theorem \ref{UniformFCLT}, Theorem \ref{UFCLT_smoothed} is a consequence of the following more general version. 

\begin{theorem}(Asymptotic equicontinuity for $\mathcal A_{(u,v]}$ - general case)\\
\label{UFCLT_smoothed-general}
If the conditions of Theorem \ref{UniformFCLT-general} are met for some $\nu > 2$, then \eqref{DEF:AEC} also holds for the smoothed process $\mathbb Z_n^s(A,f), (A,f) \in \mathcal A_{(u,v]} \times \mathcal F,$ for $\tau$ from \eqref{DEF: Semimetric_smoothed}. If, in addition, $(\mathcal F, \rho)$ is totally bounded, then fidi-convergence of $\mathbb Z_n^s$ implies weak convergence in $\ell^\infty(\mathcal A_{(u,v]} \times \mathcal F).$ 
\end{theorem}
As the proof of Theorem \ref{UniformFCLT} shows that its conditions imply those of Theorem \ref{UniformFCLT-general}, Theorem \ref{UFCLT_smoothed} is now an immediate consequence of $(\mathcal A_{(u,v]}, d_\lambda)$ being totally bounded for $d_\lambda(A,B) = \lambda(A \triangle B)$. This can be seen by covering $\mathcal A_{(u,v]}$ with the set of all intervals $(i/m, j/m]$, $0 \leq i < j \leq m, m \in \mathbb N.$

	\section{Strongly mixing arrays}
	\label{Section:Examples-NEU}
	Let us now study the sequential process and its smoothed analogue under strong mixing conditions. We extend results for stationary strongly mixing sequences obtained in \cite{MO20} and \cite{H05} and thereby provide novel results on sequential processes indexed by general function classes for nonstationary arrays that do not satisfy the restrictive $\beta$-mixing conditions imposed in \cite{PN25} (see, e.g., the introduction of \cite{BU15} for some examples of strongly mixing sequences that fail to be $\beta$-mixing). The main results of this section are Theorem \ref{SupremumInequalityUnderMixing} and \ref{SupremumInequalityUnderMixing-geometric} which ensure \eqref{DEF:AEC} under a combination of moment-, mixing- and bracketing conditions. 
	
	To proceed, we need to introduce further notation. For $t \in \mathbb N$ and $k = 1,...,n-t$ let $\sigma_{n, 1,k}^X$ and $\sigma_{n,t+k,n}^X$ be the $\sigma$-fields generated by $(X_{1,n},...,X_{k,n})$ and $(X_{t+k,n},...,X_{n,n})$. Recall the definition of the strong mixing coefficients,
	\begin{align*}
	\alpha^X_n(t) &= \begin{cases}
		\underset{1 \leq k \leq n-t}{\sup} ~ \underset{A \in \sigma_{n, 1,k}^X, B \in \sigma_{n,t+k,n}^X}{\sup} \left|P(A \cap B) - P(A)P(B) \right|, ~ t \leq n-1,\\
		0, \text{ else,} 
	\end{cases} \\
	\alpha^X(t) &= \begin{cases}
		\underset{n \in \mathbb N}{\sup} ~\alpha_n^X(t), ~t \in \mathbb N,\\
		1, ~ t = 0.
		\end{cases} \end{align*}
		$(X_{i,n})$ is said to be ``strongly mixing'', if $\alpha^X(t) \to 0$ as $t \to \infty.$ 
		
		Let us also recall the definition of the bracketing numbers. For a seminorm $\rho$ on $\mathcal F$ and $\varepsilon > 0$, let $N_{[]}(\varepsilon, \mathcal F, \rho)$ denote the smallest integer for which, i), there exists $\mathcal J \subset \mathcal F$ and a set $\mathcal K$ of maps $b: \mathcal X \to \mathbb R$ with $\#(\mathcal J) = \#(\mathcal K) = N_{[]}(\varepsilon, \mathcal F, \rho)$; ii) for each $b \in \mathcal K$, $\rho(b) \leq \varepsilon$, and for each $f \in \mathcal F,$ there exist $a \in \mathcal J$ and $b \in \mathcal K$ with $ |f - a| \leq b $ (pointwise). Below, we work with the family of seminorms defined by $$\rho_p(f) = \sup_{1 \leq t \leq n, n \in \mathbb N} \norm{f(X_{t,n})}_{L_p}, ~ f \in \mathcal F, p \geq 1.$$ 
		
		Our first result generalizes \cite[Thm.~2.5]{MO20} and provides bounds for the moduli of continuity $\norm{ S_{n,i,j} }_{\mathcal F_\delta}$ under a combination of algebraic decay conditions on the mixing coefficients, bracketing- and moment conditions. Its proof uses some arguments from the latter reference and the proof of \cite[Thm.~2.2.4]{vdVW23}. 
		
		\begin{theorem}(Algebraic mixing conditions)\\
	\label{SupremumInequalityUnderMixing}
	Let $\mathcal F$ be a set of Borel maps $\mathcal X \to \mathbb R$ and assume the following conditions to hold:
	\begin{itemize}
		\item[(i)]  There exist $\lambda> 0$ and an even integer $\nu \geq 2$ such that $$\zeta(\alpha^X,\lambda,\nu) = \sum_{s = 1}^\infty s^{\nu-2} \alpha^X(s)^{\frac{\lambda}{2+\lambda}} < \infty.$$
		
		\item[(ii)] For $\nu$ and $\lambda$ from (i), it holds $$ \int_0^1 \varepsilon^{-\frac{\lambda}{2+\lambda}} N_{[]}^{\frac{1}{\nu}}(\varepsilon, \mathcal F, \rho_2) d\varepsilon < \infty, $$ and for each $\varepsilon > 0$, the corresponding  ($\rho_2$-)set $\mathcal K$ can be chosen such that $$ \sup_{1 \leq t \leq n, n \in \mathbb N} \E{|b(X_{t,n})|^{l \frac{2+\lambda}{2}}}^{\frac12} \leq \varepsilon ~ \text{ for all } l = 2,...,\nu \text{ and } b \in \mathcal K. $$ 
	\end{itemize}
	
	Let $\rho(f) := \rho_{\nu(2+\lambda)/2}(f).$ There exists $\kappa > 0$ and a constant $C \geq 0$ that only depends on $\nu, \lambda$ and the mixing coefficients $(\alpha^X(t))_{t \in \mathbb N}$, such that, for each $\eta, \delta > 0$, $n \in \mathbb N$ and $1 \leq i \leq j \leq n$, with $m = j-i+1$, we have 
	\begin{align*}
		& \EOuter{ \sup_{\rho(f-g) \leq \delta} \left|S_{n,i,j}(f) - S_{n,i,j}(g)\right|^\nu} \nonumber \\  \label{BoundAECMixing}
		&\leq C \left[ m \left(N_{[]}^{\frac{2}{\nu}}\left(\eta, \mathcal F, \rho_2 \right) \left(m^{-\kappa} + \delta + \delta^{\frac{\nu}{2}} \right) + \int_0^{\eta} N_{[]}^\frac{1}{\nu}(\varepsilon, \mathcal F, \rho_2) \varepsilon^{-\frac{\lambda}{2+\lambda}} d\varepsilon \right)^2 \right]^{\frac{\nu}{2}}.
	\end{align*}  
	\end{theorem}
	
	The conditions (i) and (ii) of the above theorem require a trade-off between complexity of the set $\mathcal F$, moment- and decay conditions of the mixing coefficients.  Section 3 of \cite{MO20} provides two examples that satisfy all conditions of Theorem \ref{SupremumInequalityUnderMixing}. As stated in \cite{AP94}, the integral condition can be expected to hold for suitable choices of $\lambda$ and $\nu$ if $\mathcal F$ depends in a Lipschitz continuous way on a parameter $\theta$ that lives in bounded subsets of $\mathbb R^k$ for some $k \in \mathbb N$ (cf. equation (2.1) of \cite{AP94}). Similar conditions have also been imposed in \cite{BL24}, for instance. 
	
	However, it is evident that any class $\mathcal F$ that satisfies condition (ii) of Theorem \ref{SupremumInequalityUnderMixing} must be of at most polynomial complexity in some sense such as a VC-class. The following theorem, which generalizes \cite[Thm.~3]{H05} to sequential processes of nonstationary arrays, allows for exponentially growing bracketing numbers. The basic method of proof remains the same, but some extra care is needed to obtain the uniformity in the parameters $m$ and $\delta$ required in \eqref{UniformFCLT-gammabound}. 
	\begin{theorem}(Geometric mixing conditions)\\
	\label{SupremumInequalityUnderMixing-geometric}
	Let $\mathcal F$ be a set of Borel maps $\mathcal X \to \mathbb R$ that satisfies $\sup_{f \in \mathcal F} |f(x)| \leq 1$. Furthermore, suppose there exists $\beta \in (0,1)$ and $C_\beta \geq 0$ with $ \alpha^X(t) \leq C_\beta \beta^t, t \in \mathbb N_0, $ 	
	and that it holds 
	\begin{equation}\label{geometric-integral}
		\int_0^1 \log^2 N_{[]}(\varepsilon, \mathcal F, \rho_\nu) d\varepsilon < \infty
	\end{equation}  
	for some $\nu > 2.$ Then there exists $\kappa > 0$, a constant $C \geq 0$ that only depends on $\nu$ and the mixing coefficients and $\lambda,\Lambda: (0,\infty) \to [0,\infty)$ with $\Lambda(\delta) \to 0$ as $\delta \downarrow 0$ such that for each $\delta > 0$, $n \in \mathbb N$ and $1 \leq i \leq j \leq n$, with $m = j-i+1$, we have 
	\begin{align*}
		& \EOuter{ \sup_{\rho_\nu(f-g) \leq \delta} \left|S_{n,i,j}(f) - S_{n,i,j}(g)\right|^\nu} \leq C \left[ m \left( \Lambda(\delta) + \lambda(\delta) m^{-\kappa} \right)^2 \right]^{\frac{\nu}{2}}.
	\end{align*} 
	\end{theorem} 
	
	The proof of Theorem~\ref{SupremumInequalityUnderMixing-geometric} is based on the following adaption of the Rosenthal-type inequality stated in \cite[Lem.~2]{H05}, which is of independent interest.
	
	\begin{lemma}\label{MaximalInequality-geometric}
	Let $\nu > 2$ and $h: \mathcal X \to \mathbb R$ be a Borel map with $\norm{h}_\infty = \sup_{x \in \mathcal X}|h(x)| < \infty.$ Suppose there exists $\beta \in (0,1)$ and $C_\beta \geq 0$ with $ \alpha^X(t) \leq C_\beta \beta^t, t \in \mathbb N_0. $ Then, there exists a constant $C \geq 0$ that only depends on $\nu$ and the mixing coefficients such that for any $p > 2$, $n \in \mathbb N$ and $1 \leq i \leq j \leq n$, with $m = j-i+1$, it holds
	$$ \norm{S_{n,i,j}(h)}_{L_p} \leq C \sqrt{m} \left( \sqrt{p} \rho_\nu(h) + p^2 m^{-\frac12 + \frac{1}{p}}  \norm{h}_\infty \right). $$  	
	Hence, if $\mathcal H$ is a finite set of Borel maps $\mathcal X \to \mathbb R$ with $\max_{h \in \mathcal H} \norm{h}_\infty < \infty$, there exists $C_1 \geq 0$ that only depends on $\nu$ and the mixing coefficients such that 
	$$ \norm{\max_{h \in \mathcal H} |S_{n,i,j}(h)|}_{L_\nu} \leq C_1 (1 \lor \log^2 \#(\mathcal H))  \sqrt{m} \left( \max_{h \in \mathcal H} \rho_\nu(h) + m^{-\frac12 + \frac{1}{\nu}} \max_{h \in \mathcal H} \norm{h}_\infty \right). $$ 
	\end{lemma}
	
	Theorem \ref{SupremumInequalityUnderMixing} and \ref{SupremumInequalityUnderMixing-geometric} provide bounds of the form required in \eqref{UniformFCLT-gammabound}. The asymptotic tightness and weak convergence of the sequential process $\mathbb Z_n$ can now be obtained from Theorem \ref{UniformFCLT}. Thus, we obtain the following corollaries.
	
\begin{corollary}\label{FCLTApplicableUnderMixing}
Let the conditions of Theorem \ref{SupremumInequalityUnderMixing} be fulfilled for  some $\nu > 2$ and $\lambda > 0$ and suppose there is a constant $1 \leq K < \infty $ with 
\begin{equation}\label{FCLTApplicableUnderMixingMomentCondition}
\sup_{1 \leq t \leq n, n \in \mathbb N} \sup_{f \in \mathcal F} \E{|f(X_{t,n})|^{\nu\frac{2+\lambda}{2}}} \leq K.
\end{equation} 
Then \eqref{DEF:AEC} holds for $\mathbb Z_n$ for $\rho$ from Theorem \ref{SupremumInequalityUnderMixing} and fidi-convergence of $\mathbb Z_n$ implies weak convergence in $\ell^\infty([0,1]\times\mathcal F).$
\end{corollary}

\begin{corollary}\label{FCLTApplicableUnderMixing-geometric}
The assumptions of Theorem \ref{SupremumInequalityUnderMixing-geometric} imply \eqref{DEF:AEC} for $\mathbb Z_n$ for $\rho_\nu$ and fidi-convergence of $\mathbb Z_n$ implies weak convergence in $\ell^\infty([0,1]\times\mathcal F).$	
\end{corollary}

Given that, an application of Theorem \ref{UFCLT_smoothed} proves the analogous conclusion for the $\mathcal A_{(u,v]}$-indexed smoothed process $\mathbb Z_n^s$. A possible application of Theorem \ref{UFCLT_smoothed_Lip} is sketched in Section \ref{Section:Application}. 
\begin{corollary}\label{FCLTApplicableUnderMixing-Smoothed}
Under the conditions of either Corollary \ref{FCLTApplicableUnderMixing} or \ref{FCLTApplicableUnderMixing-geometric}, \eqref{DEF:AEC} holds for $\mathbb Z_n^s$ indexed by $\mathcal A_{(u,v]} \times \mathcal F$ for $\tau$ from \eqref{DEF: Semimetric_smoothed} and fidi-convergence of $\mathbb Z_n^s$ implies weak convergence in $\ell^\infty(\mathcal A_{(u,v]}\times\mathcal F).$
\end{corollary}

	\section{Application}
	\label{Section:Application}
	An important problem arising in diverse areas is to analyze whether a time series of observations has a change-point where the distribution changes. For example, depending on the application, such a change-point may represent the onset of a financial crisis, a cyber attack or a tipping point in a climate series.  
	
	Suppose one observes random vectors $X_1,...,X_n$ in $\mathbb R^D$ and wishes to test for a change in the underlying marginal distributions. For a given candidate change-point location $k$, it is natural to compare the empirical measures of the pre- and after-change period. Thus, \cite{GH} proposed the test statistic
	\begin{equation}\label{General-statistic}
		T_n^{(d)} = \sqrt{n} \max_{1 \leq k \leq n} \frac{k}{n} \frac{n-k}{n}  d\left(\mathbb P_k, \widetilde{ \mathbb P}_{n-k}  \right),
	\end{equation}
	which also has been studied by \cite{HKQ13} for i.i.d. data and in Section 4 of \cite{DDT14} for stationary sequences.
	Here, $\mathbb P_k$ is the empirical distribution of $X_1,...,X_k$, $\widetilde{\mathbb P}_{n-k}$ is that of $X_{k+1},...,X_n,$ and $d$ is any metric on the space of probability measures on the Borel sets of $\mathbb R^D$ that admits a representation of the form $$ d(\mu, \nu) = \sup_{f \in \mathcal F} \left| \int f d\mu - \int f d\nu  \right| = \sup_{f \in \mathcal F, f(0) = 0} \left| \int f d\mu - \int f d\nu  \right| $$ for some set $\mathcal F$ of maps $\mathbb R^D \to \mathbb R$. Two important examples for such metrics are the Kolmogorov distance and the family of Wasserstein distances. The latter corresponds to choosing $ \mathcal F$ as (a subset of) the class of Lipschitz continuous functions as in \cite{XH22}.
	
	By our choice of the metric $d$, we have the representation
	\begin{equation}
		\label{ReprTStat}
		T_n^{(d)} = \sup_{(t,f) \in [0,1] \times \mathcal F} \left| \mathbb Z_n(t,f) - t \mathbb Z_n(1,f) + \frac{nt - \floor{nt}}{n} \mathbb Z_n(1,f)  \right|    
	\end{equation}
	in terms of the sequential empirical process. The asymptotic distribution of $ T_n^{(d)} $ can be derived whenever the assumed model for the underlying observations ensures the fidi-convergence of $\mathbb Z_n$, since the results of this paper provide the tools to prove its tightness.  
	Here, one can rely on suitable CLTs for nonstationary processes such as \cite{BT2017} or \cite[Thm.~1]{AS2024}. Regarding the bracketing conditions of Corollary \ref{FCLTApplicableUnderMixing-geometric}, if one considers Lipschitz functions on $ [-M,M]$, $M>0$, \cite[Lem.~2]{XH22} provides a bound on the bracketing numbers in uniform norm, and arguments as thosed used to prove \cite[Thm.~2]{XH22} can then be used to bound the $\rho_\nu$-bracketing numbers. For smoother functions on unbounded domains, one can use \cite[Cor.~2.7.3]{vdVW23}. Then, by \cite[Prop.~4.1]{DDT14}, the weak convergence of $\mathbb Z_n$ implies
	\[
	T_n^{(d)} \stackrel{d}{\to} \sup_{t \in [0,1], f \in \mathcal F} \left| \mathbb Z(t,f) - t \mathbb{Z}(1,f) \right|,
	\]
	as $ n \to \infty$, for the limiting Kiefer process $ \mathbb{Z} $ that is uniquely determined by its mean and covariance function. 
	
	Note that under stationarity, the set of distributions for $ ( X_i ) $ leading to the same first and second moment structure $ \E{ f(X_1)} $ and $$ \sum_{k=1}^\infty \mathrm{Cov}( f(X_1), g(X_{k+1}) ) + \mathrm{Cov}( g(X_1), f(X_{k+1}), ~ f, g \in \mathcal{F}, $$ of the Kiefer process $ \mathbb{Z} $ cannot be distinguished by the test. But under nonstationarity, the set of indistinguishable distributions may be considerably larger, as it includes distributions such that the first and second moment structure converges sufficiently fast to a given one.
	
	One can also replace the sequential empirical process in \eqref{ReprTStat} by its smoothed $\mathcal A \times \mathcal F$-indexed version $\mathbb Z_n^s$ and make use of our Theorems \ref{UFCLT_smoothed_Lip} and \ref{UFCLT_smoothed}. To bound the weighted sums in $L_{\psi_2}$-norm required to verify the $ L_{\psi_2} $-Lipschitz property, one can rely on the techniques of \cite[Thm.~2.8]{KM25}.

	\section{Discussion and Outlook}\label{Section:Discussion} 
	The comprehensive theory of asymptotic tightness and weak convergence of the sequential empirical process  and its smoothed set-indexed analogue for nonstationary time series developed in this paper provides sufficient conditions in terms of abstract moment bounds and regularity conditions imposed on the family $ \mathcal{F}$. A key tool for verifying the tightness of $\mathbb Z_n$ and $\mathbb Z_n^s$ is a new maximal inequality for nonmeasurable processes. Alternatively, under a measurability condition, the tightness of $\mathbb Z_n^s$ can also be derived from Lipschitz properties. 
	These results avoid explicit dependency assumptions and can therefore be specialized to different notions of weak dependence. That has been exemplified in detail for strongly mixing nonstationary arrays by extending known results for the empirical process. We have shown that conditions implying the weak convergence of the empirical process need only be slightly strengthened to imply the weak convergence of the sequential process $\mathbb Z_n$ and a certain smoothed version. This enlarges the scope of applicability to decision procedures which can be represented in terms of the sequential empirical process, as illustrated in our change-point testing example. 
	
	To use those results in applications, one needs to estimate the covariance function of the limiting Kiefer process to simulate critical values, or rely on a suitable bootstrap procedure. There are only few results in this direction under nonstationarity. For nonstationary Bernoulli shifts \cite{MS23} studies a wild bootstrap for unweighted partial sums and \cite{a_AS2024} for localized partial sums of spectral statistics. \cite{BL24} consider a block bootstrap for a class of localized averages and \cite{PN25} provide a multiplier Bootstrap under $\beta$-mixing conditions.  Extending the general setting considered in \cite{PR21}, which studies processes allowing for a representation as Bernoulli shifts under physical dependence conditions, to sequential processes - including a consistent multiplier bootstrap to provide applicable approximations - is subject of current research.

\section*{Acknowledgements} The authors gratefully
acknowledge financial support from Deutsche Forschungsgemeinschaft (DFG, grant
STE 1034/13-1).\\

\bibliographystyle{apalike}
\bibliography{ref.bib}

\section{Appendix}\label{Section:Appendix}
\subsection{Preliminaries on outer expectations}\label{Section:Sub-Preliminiaries}

We start by recalling some elementary facts related to the outer expectation which are needed in the proof of Proposition \ref{MSSAdaption_general}. For a thorough introduction to the theory of outer expectations and probabilities, see Section 1.2 in \cite{vdVW23}.

For any two maps $Y,X: \Omega \to \mathbb R$, $\EOuter{Y} = \E{Y}$, if $Y$ is measurable, and $Y \leq X$ implies $\EOuter{Y} \leq \EOuter{X}$, i.e. the outer expectation is monotonic in its ``integrand'', both of which follow directly from the definition. A very useful, yet less obvious feature is that there always exists the so-called measurable cover function $Y^*$, a measurable map $Y^*: \Omega \to \ER$ that fulfills $Y \leq Y^*$ and $\EOuter{Y} = \E{Y^*}$, provided the latter expectation exists in $\ER$, the extended real line \cite[Lem.~1.2.1]{vdVW23}. Measurable covers have many useful properties that hold irrespective of the underlying probability space (see, e.g., \cite[Lem.~1.2.2]{vdVW23} or \cite[Lem.~6.8]{KO07}) and the identity $\EOuter{Y} = \E{Y^*}$ can then be used to deduce properties of outer integrals. Two further properties that will be needed throughout are the following: Firstly, $\EOuter{.}$ is subadditive in the sense that 
\begin{equation}\label{OuterExpectationSubAdd}
	\EOuter{|X|+|Y|} \leq \EOuter{|X|} + \EOuter{|Y|},
\end{equation}
and secondly, for any $a\in \mathbb R$, it holds 
\begin{align}\label{OuterExpectationHomogeneous}
	\EOuter{|aY|} = |a|\EOuter{|Y|},
\end{align}
i.e. $\EOuter{.}$ exhibits a certain homogeneity-property. Both statements can be concluded from \cite[Lem.~1.2.1 and 1.2.2]{vdVW23}. And finally, there also exist versions of Markov's, H{\"o}lder's and Minkowski's inequalities for outer expectations and probabilities. The first of these three results is proven in \cite[Lem.~6.10]{KO07}. As we were unable to find proofs for the latter two in the literature, we briefly state and prove them below. In particular, note that due to the Minkowski-analogue and the homogeneity-property stated in \eqref{OuterExpectationHomogeneous}, it makes sense to introduce the family of seminorms given by 
\[\norm{Y}_{L_p}^* = (\EOuter{|Y|^p})^{1/p}, ~ p \in [1,\infty).
\]

\begin{lemma}[H{\"o}lder's inequality]\label{HoelderForOuter}
	Let $X,Y: \Omega \to \mathbb R$ be arbitrary maps. If there exist $p,q > 1$ with $p^{-1} + q^{-1} = 1$ and $\EOuter{|X|^p}, \EOuter{|Y|^q} < \infty$, we have 
	$$ \EOuter{|XY|} \leq \norm{X}_{L_p}^* \norm{Y}_{L_q}^*.$$
\end{lemma}

\begin{proof}[Proof of Lemma \ref{HoelderForOuter}]
	Denote by $|X|^*$ and $|Y|^*$ measurable covers of $|X|$ and $|Y|$. Since $|X| \leq |X|^*$ and $|Y| \leq |Y|^*$ and since the cover functions are measurable, we have $$ \EOuter{|XY|} \leq \EOuter{|X|^* |Y|^*} = \E{|X|^* |Y|^*}.$$
	Define, for $r \geq 1$, the family of maps $\phi_r: \mathbb R \to \mathbb R$ by $$ \phi_r(x) = \begin{cases*}
		0, ~ x < 0,\\
		x^r, ~ x \geq 0,
	\end{cases*}$$ and extend each $\phi_r$ continuously to $\ER$ by setting $\phi_r(\infty) = \infty$ and $\phi_r(-\infty) = 0 .$ Then $\phi_p$ and $\phi_q$ are nondecreasing and continuous on $[-\infty,\infty)$, and hence part A of \cite[Lem.~6.8]{KO07} implies 
	\begin{equation}\label{PowerCommutesWithCover-p}
		(|X|^p)^* = (\phi_p(|X|))^* = \phi_p(|X|^*) = (|X|^*)^p
	\end{equation}
	and
	\begin{equation}\label{PowerCommutesWithCover-q}
		(|Y|^q)^* = (\phi_q(|Y|))^* = \phi_q(|Y|^*) = (|Y|^*)^q
	\end{equation}
	almost surely. This implies $|X|^* \in L_p(P)$ and $|Y|^* \in L_q(P)$, since $\E{ (|X|^*)^p } = \E{(|X|^p)^*} =  \EOuter{|X|^p} < \infty$ and, analogously, $\E{ (|Y|^*)^q } = \EOuter{|Y|^q} < \infty$. Hence, by H{\"o}lder's inequality and \eqref{PowerCommutesWithCover-p} and \eqref{PowerCommutesWithCover-q}, we obtain
	$$ \E{|X|^* |Y|^*} \leq \left( \E{(|X|^*)^p} \right)^{\frac{1}{p}} \left( \E{(|Y|^*)^q} \right)^{\frac{1}{q}} = \left( \EOuter{|X|^p} \right)^{\frac{1}{p}} \left( \EOuter{|Y|^q} \right)^{\frac{1}{q}},$$ which concludes the proof.
\end{proof}

\begin{lemma}[Minkowski's inequality]\label{MinkowskiForOuter}
	Let $X,Y: \Omega \to \mathbb R$ be arbitrary maps and let $p \in [1, \infty)$. If  $\EOuter{|X|^p}, \EOuter{|Y|^p} < \infty$, we have 
	$$ \norm{ X + Y }_{L_p}^* \leq \norm{ X }_{L_p}^* + \norm{ Y }_{L_p}^*.$$
\end{lemma}

\begin{proof}[Proof of Lemma \ref{MinkowskiForOuter}]
	We proceed as in the proof of the ``classical'' Minkowski inequality (see, e.g., \cite[Thm.~7.17]{Klenke14}). For $p = 1$, the statement is a consequence of the triangle inequality and \eqref{OuterExpectationSubAdd}. For $p > 1$, we have $ |X + Y|^p \leq 2^{p-1}(|X|^p + |Y|^p)$ by convexity, which implies $\EOuter{|X+Y|^p}<\infty$ due to the properties \eqref{OuterExpectationHomogeneous} and \eqref{OuterExpectationSubAdd}. Assuming $\EOuter{|X+Y|^p} > 0$ without loss of generality, we have, by monotonicity, subadditivity and Lemma \ref{HoelderForOuter} with q = p/(p-1),
	\begin{align*}
		(\norm{X + Y}_{L_p}^*)^p = \EOuter{|X+Y|^p} &= \EOuter{|X+Y|\cdot|X+Y|^{p-1}} \\
		&\leq \EOuter{|X|\cdot|X+Y|^{p-1}} + \EOuter{|Y|\cdot|X+Y|^{p-1}} \\
		&\leq \left(\norm{X}_{L_p}^* + \norm{Y}_{L_p}^*\right) (\norm{X + Y}_{L_p}^*)^{p-1}.
	\end{align*}
	Dividing both sides by $(\norm{X + Y}_{L_p}^*)^{p-1}$ concludes the proof.
\end{proof}

\subsection{Proofs of Section \ref{Section:UFCLT}}\label{Section:Sub:Proofs2}
We start with the following generalization of Proposition \ref{MSSAdaption}.
\begin{proposition}\label{MSSAdaption_general}
	Let $\nu \geq 1$, $\Psi \neq \emptyset$ and $W_1,...,W_n: \Omega \to \ell^\infty(\Psi)$ be arbitrary processes. If there exist $\alpha > 1$ and $g: \{1,...,n\}^2 \to \mathbb R$ which fulfill
	\begin{equation}\label{SumboundMSS_general}
		\mathrm E^*\left\{ \norm{S_{i,j}^W}_\Psi^\nu \right\} \leq g^{\alpha}(i,j)
	\end{equation}
	for all $1 \leq i \leq j \leq n,$ and if it holds
	\begin{align*}
		\mathrm{(i')}&:~ g(i,j) \geq 0, \text{ for all } (i,j) \in \{1,...,n\}^2, \nonumber\\
		\mathrm{(ii')}&:~ g(i,j) \leq g(i,j+1), \text{ for all } 1 \leq i \leq j \leq n-1, \nonumber \\
		\mathrm{(iii')}&:~ g(i,j) + g(j+1,k) \leq Q g(i,k), \text{ for all } 1 \leq i \leq j < k \leq n, \nonumber
	\end{align*}
	for some $Q \in [1,2^{(\alpha-1)/\alpha})$, then there exists a constant $A$ that only depends on $\alpha, \nu,$ and $Q$ such that
	\begin{equation}\label{MaxBoundMSS_general}
		\mathrm E^* \left\{\norm{M_{1,n}^W}_\Psi^\nu \right\} \leq A g^{\alpha}(1,n).
	\end{equation}
	One may take
	\begin{equation*}
		A = \left(1 - \frac{Q^{\alpha / \nu}}{2^{(\alpha - 1)/ \nu}} \right)^{-\nu}.
	\end{equation*}
\end{proposition}
\begin{proof}[Proof of Proposition \ref{MSSAdaption_general}.]
	First note that since $W_i(\omega) \in \ell^{\infty}(\Psi)$ for all $i = 1,...,n$ and $\omega \in \Omega$ by assumption, Lemma \ref{MinkowskiForOuter} is applicable to all of the sums and maxima $S_{i,j}^W, M_{i,j}^W, 1 \leq i \leq j \leq n$, which will be needed in the course of the proof. 
	
	We closely follow the proof of \cite[Thm.~3.1]{MSS82}. Since the authors only sketch their proof, we present some of the arguments in more detail for the sake of clarity. 
	
	The assertion is shown by induction over $n \in \mathbb N.$ For $n = 1$, the assertion is trivial. Therefore, let $n \geq 2$ and assume that the assertion is true for all $k \leq n-1$. Recalling that $S_{i,j}^W(\psi) = M_{i,j}^W(\psi) = 0$ for $j < i$ and all $\psi \in \Psi$, we put $g(1,0) = g(n+1,n) = 0.$ Then, as the first inequality in \eqref{MSS27} below holds for $ m = 1 $, since $ g \geq 0 $, and the second one for $m = n$, since $ Q < 2$,  we can find  $m \in \{1,...,n\}$ such that
	\begin{equation}\label{MSS27}
		g(1, m-1) \leq \frac{Q}{2} g(1,n) \leq g(1,m).
	\end{equation}
	For this $m$, we have 
	\begin{equation}\label{MSS28}
		g(m+1,n) \leq \frac{Q}{2} g(1,n),
	\end{equation}	
	which, for $1 \leq m \leq n-1$, follows from (iii') and \eqref{MSS27}, and clearly holds true for $m = n$.
	
	Now take any $\psi \in \Psi$. Arguing as in the proof of \cite[Thm.~1]{MO76}, for $m \leq k \leq n,$ we have
	\begin{equation*}
		\left|S_{1,k}^W(\psi) \right| \leq \left|S_{1,m}^W(\psi) \right| + M_{m+1,k}^W(\psi) \leq \left|S_{1,m}^W(\psi) \right| + M_{m+1,n}^W(\psi),
	\end{equation*}
	and for $1 \leq k \leq m-1$, it holds $|S_{1,k}^W(\psi)| \leq M_{1,m-1}^W(\psi) \leq |S_{1,m}^W(\psi) | + M_{1,m-1}^W(\psi).$ Hence, we obtain 
	\begin{align*}
		M_{1,n}^W(\psi) &\leq \left|S_{1,m}^W(\psi) \right| + \max \left\{ M_{1,m-1}^W(\psi), M_{m+1,n}^W(\psi) \right\} \\
		&\leq \sup_{\psi \in \Psi} \left|S_{1,m}^W(\psi) \right| + \max \left\{ \sup_{\psi \in \Psi} M_{1,m-1}^W(\psi), \sup_{\psi \in \Psi} M_{m+1,n}^W(\psi) \right\} \\
		&\leq \sup_{\psi \in \Psi} \left|S_{1,m}^W(\psi) \right| + \left(\sup_{\psi \in \Psi} \left| M_{1,m-1}^W(\psi)\right|^\nu + \sup_{\psi \in \Psi} \left| M_{m+1,n}^W(\psi)\right|^\nu \right)^{\frac{1}{\nu}},
	\end{align*}
	where we used the estimate $|x|_\infty \leq |x|_{\nu}, x \in \mathbb R^2,$ in the last step. We conclude that
	\begin{align*}
		\sup_{\psi \in \Psi} M_{1,n}^W(\psi) 
		&\leq \sup_{\psi \in \Psi} \left|S_{1,m}^W(\psi) \right| \\
		&+ \left(\sup_{\psi \in \Psi} \left| M_{1,m-1}^W(\psi)\right|^\nu + \sup_{\psi \in \Psi} \left| M_{m+1,n}^W(\psi)\right|^\nu \right)^{\frac{1}{\nu}}
	\end{align*}
	which by Lemma \ref{MinkowskiForOuter} entails
	\begin{align*}
		\norm{\sup_{\psi \in \Psi} M_{1,n}^W(\psi)}_{L_\nu}^{*}
		&\leq \norm{\sup_{\psi \in \Psi} \left| S_{1,m}^W(\psi) \right|}_{L_\nu}^{*} \\ 
		&+ \left(\mathrm E^*\left\{\sup_{\psi \in \Psi} \left|M_{1,m-1}^W(\psi) \right|^\nu \right\} + \mathrm E^*\left\{\sup_{\psi \in \Psi} \left|M_{m+1,n}^W(\psi) \right|^\nu \right\} \right)^{\frac{1}{\nu}}.
	\end{align*}
	By assumption \eqref{SumboundMSS_general}, 
	\begin{equation*}
		\norm{\sup_{\psi \in \Psi} \left| S_{1,m}^W(\psi) \right|}_{L_\nu}^{*} \leq g^{\frac{\alpha}{\nu}}(1,m)
	\end{equation*}
	and furthermore, by the induction hypothesis and \eqref{MSS27}, 
	\begin{equation*}
		\mathrm E^*\left\{\sup_{\psi \in \Psi} \left|M_{1,m-1}^W(\psi) \right|^\nu \right\} \leq A g^\alpha(1,m-1) \leq A \left(\frac{Q}{2} \right)^\alpha g^\alpha(1,n).
	\end{equation*}
	Finally, note that if $Y_1 = W_{m+1},...,Y_{n-m} = W_{n}$ and $g_Y(i,j) = g(m+i,m+j)$, then $g_Y$ satisfies (i')-(iii') with $Q_Y = Q$ and 
	$$ \EOuter{\sup_{\psi \in \Psi} \left|S_{i,j}^Y(\psi) \right|^\nu} \leq g_Y^{\alpha}(i,j), ~ \forall ~ 1 \leq i\leq j\leq n-m. $$ Hence, we also have
	\begin{align*}
		\mathrm E^*\left\{\sup_{\psi \in \Psi} \left|M_{m+1,n}^W(\psi) \right|^\nu \right\} &= \EOuter{\sup_{\psi \in \Psi} \left|M_{1,n-m}^Y(\psi) \right|^\nu}\\ 
		&\leq A g_Y^{\alpha}(1,n-m) \\
		&= A g^{\alpha}(m+1,n) \leq A \left(\frac{Q}{2} \right)^\alpha g^{\alpha}(1,n)
	\end{align*}
	by the induction hypothesis and \eqref{MSS28}. We conclude that
	\begin{align*}
		\norm{\sup_{\psi \in \Psi} M_{1,n}^W(\psi)}_{L_\nu}^{*} &\leq g^{\frac{\alpha}{\nu}}(1,m) + \left(A \frac{Q^{\alpha}}{2^{(\alpha-1)}} g^\alpha(1,n) \right)^{\frac{1}{\nu}}\\
		&\leq g^{\frac{\alpha}{\nu}}(1,n) \left(1 + A^{\frac{1}{\nu}} \frac{Q^{\alpha/\nu}}{2^{(\alpha-1)/\nu}} \right),
	\end{align*}
	where we used that $g(1,m) \leq g(1,n)$ by assumption (ii'). Hence, for $A$ being chosen larger than or equal to 
	$$ \left(1 - \frac{Q^{\alpha / \nu}}{2^{(\alpha - 1)/ \nu}} \right)^{-\nu}, $$ we have 
	\begin{equation*}
		\norm{\sup_{\psi \in \Psi} M_{1,n}^W(\psi)}_{L_\nu}^{*} \leq A^{\frac{1}{\nu}} g^{\frac{\alpha}{\nu}}(1,n),
	\end{equation*}
	which concludes the proof.
\end{proof}
\begin{proof}[Proof of Proposition \ref{MSSAdaption}]
	The function $\tilde g: \{1,...,n\}^2 \to \mathbb R$ defined by $g(i,j) = g(j-i+1)$, $i \leq j$, and $g(i,j) = 0$ otherwise fulfills the conditions (i')-(iii') of Proposition \ref{MSSAdaption_general} (see the first remark in \cite{MSS82}, in particular equation (1.5) therein). Hence, by arguing as in the proof of Proposition \ref{MSSAdaption_general}, for $1 \leq i \leq j \leq n,$ we have
	$$ \EOuter{\sup_{\psi \in \Psi} \left|M_{i,j}^W(\psi) \right|^\nu} \leq A\tilde g^\alpha(i,j) = Ag^\alpha(j-i+1) $$ and the constant $A$ does not depend on the pair $(i,j) \in \{1,...,n\}^2$.
\end{proof}

The next result is a main ingredient of the proof of Theorem \ref{UniformFCLT}.
\begin{lemma}(Properties of $\gamma$)\\\label{Generic-gn}
	For $\nu > 2, \kappa \in (0, 1/2 - 1/\nu), C \geq 0$ and $R,J: (0,\infty) \to [0,\infty)$ let
	\begin{equation}\label{Generic-gn-Equation}
		\gamma(m, \delta) = C m \left(R(\delta) + J(\delta) m^{-\kappa} \right)^2, ~ m \in \mathbb N, \delta > 0.
	\end{equation}
	
	\begin{itemize}
		\item[(i)] For any $\delta > 0$, the pair $(\nu/2, \gamma(.,\delta))$ fulfills condition \eqref{MSS-ConditionM} with index $Q = 2^{2\kappa} \in [1,2^{1-2/\nu}).$
		
		\item[(ii)] For any $\delta > 0$, $\gamma(., \delta)$ fulfills condition \ref{item:Uniform_FCLT_Cond4} of Theorem \ref{UniformFCLT-general}.		
		
		\item[(iii)] If $\lim_{\delta \downarrow 0} R(\delta) = 0$, then $\gamma$ fulfills condition \ref{item:Uniform_FCLT_Cond3} of Theorem \ref{UniformFCLT-general}. 
	\end{itemize}
\end{lemma}

The proof of Lemma \ref{Generic-gn} uses the following elementary result. 
\begin{lemma}\label{HoelderConstantsBound}
	Let $\delta \in (0,1)$. For $x,y \geq 0$, we have $$ x^\delta + y^\delta \leq 2^{1-\delta}(x+y)^\delta.$$
\end{lemma}

\begin{proof}[Proof of Lemma \ref{HoelderConstantsBound}]
	By convexity of the map $u \mapsto u^{\frac{1}{\delta}}, u \geq 0,$ and Jensen's inequality, we have
	\begin{align*}
		x^\delta + y^\delta = \left[ \left(x^\delta + y^\delta \right)^\frac{1}{\delta} \right]^\delta \leq \left[ 2^{\frac{1}{\delta} - 1} \left(x + y \right) \right]^\delta = 2^{1 - \delta} (x+y)^\delta.
	\end{align*}
\end{proof}

\begin{proof}[Proof of Lemma \ref{Generic-gn}]
	As the following calculations concern the growth and limiting behaviour of $\gamma$, it does not entail a loss of generality to assume $C = 1$.\\
	(i): Let $\delta > 0.$ Evidently, $\gamma(.,\delta)$ is nonnegative. By 
	\begin{align*}
		\sqrt{\gamma(m,\delta)} &= \sqrt{m} \left(R(\delta) + J(\delta) m^{-\kappa} \right) = \sqrt{m} R(\delta) + J(\delta) m^{\frac12 - \kappa}\\
		&\leq \sqrt{m+1} R(\delta) + J(\delta) (m+1)^{\frac12 - \kappa} = \sqrt{\gamma(m+1,\delta)}
	\end{align*}
	it is nondecreasing, since $\kappa < \frac12.$ And finally, by Lemma \ref{HoelderConstantsBound}, for $1 \leq i \leq j \leq n$, we have
	\begin{align*}
		&\gamma(i,\delta) + \gamma(j-i,\delta) \\
		&= R(\delta)^2 j + 2R(\delta) J(\delta) (i^{1-\kappa} + (j-i)^{1-\kappa}) + J(\delta)^2 (i^{1-2\kappa} + (j-i)^{1-2\kappa}) \\
		&\leq R(\delta)^2 j + 2R(\delta) J(\delta) 2^{\kappa} j^{1-\kappa} + J(\delta)^2 2^{2\kappa} j^{1-2\kappa} \\
		&\leq 2^{2\kappa} \gamma(j,\delta)
	\end{align*}
	with $Q = 2^{2\kappa} < 2^{1 - \nu/2}$ due to $\kappa < 1/2 - 1/\nu$.\\
	(ii): Let $\varepsilon,\delta > 0$. For all $n$ large enough, we have
	\begin{align*}
		\frac{\gamma(\floor{n \varepsilon},\delta)}{n} &\leq \varepsilon \left(R(\delta) + J(\delta) \floor{n \varepsilon}^{-\kappa} \right)^{2}\\ 
		&\leq 2 \varepsilon \left(R(\delta)^{2} + J(\delta)^{2} \floor{n\varepsilon}^{-2 \kappa} \right),
	\end{align*}
	by Jensen's inequality. It follows
	\begin{align*}
		\limsup_{n \to \infty} \frac{\gamma(\floor{n\varepsilon},\delta)}{n} &\leq 2R(\delta)^{2} \varepsilon  
	\end{align*}
	and $R(\delta) < \infty$ holds by assumption.\\
	(iii): Take $\varepsilon = 1$ in the last display and let $\delta \downarrow 0.$
\end{proof}

\begin{proof}[Proof of Theorem \ref{UniformFCLT}]
	In view of Lemma \ref{Generic-gn}, Theorem \ref{UniformFCLT} will now be proven if we verify the conditions \ref{item:Uniform_FCLT_Cond2} and \ref{item:Uniform_FCLT_Cond4} of Theorem \ref{UniformFCLT-general}. Let us start by noting that $\Delta_\rho$, the diameter of $\mathcal F$ with respect to $\rho$, is finite because $(\mathcal F, \rho)$ is totally bounded. By the triangle inequality and Minkowski's inequality (Lemma \ref{MinkowskiForOuter}), for each $n \in \mathbb N$ and $1 \leq i \leq j \leq n,$ we have
		\begin{align*}
			\left( \EOuter{ \norm{ S_{n,i,j} }_{\mathcal F}^\nu } \right)^{\frac{1}{\nu}} &\leq \left( \EOuter{ \norm{ S_{n,i,j} }_{\mathcal F_{\Delta_\rho}}^\nu } \right)^{\frac{1}{\nu}} + \norm{S_{n,i,j}(f_0)}_{L_\nu} \\ 
			&\leq \gamma^{\frac12}(j-i+1,\Delta_\rho) + C \sqrt{j-i+1},
		\end{align*}
		where the last inequality holds by assumption. Put $l(m) = C^2 m, m \in \mathbb N$, and observe that the pairs $(\nu/2, l)$ and $(\nu/2, \gamma(.,\Delta_\rho))$ fulfill condition \eqref{MSS-ConditionM} with indices $Q_l = 1$ and $Q_\gamma = 2^{2\kappa} \in [1,2^{1-2/\nu})$, respectively, where the former follows from direct computations, the latter from Lemma \ref{Generic-gn}. Furthermore, by Lemma \ref{HoelderConstantsBound}, we have $$ \gamma^{\frac12}(m,\Delta_\rho) + l^{\frac12}(m) \leq \sqrt{2} \left(\gamma(m,\Delta_\rho) + l(m) \right)^{\frac12}, ~ m \in \mathbb N.$$ And finally, it is easily seen that if, for a fixed $\alpha > 1$, two functions fulfill condition \eqref{MSS-ConditionM} with indices $Q_1, Q_2 \in [1, 2^{(\alpha - 1)/\alpha})$, then so does their sum with index $Q = \max\{Q_1,Q_2\}.$ Consequently, for $h(m) =  2\gamma(m,\Delta_\rho) + 2l(m), m \in \mathbb N,$ it holds 
		$$ \EOuter{ \norm{ S_{n,i,j} }_{\mathcal F}^\nu } \leq h^{\frac{\nu}{2}}(j-i+1), ~ \forall ~ 1 \leq i \leq j \leq n, $$
		and the pair $(\nu/2, h)$ fulfills condition \eqref{MSS-ConditionM} with index $Q_h = \max\{1, Q_\gamma\} = Q_\gamma \in [1,2^{1-2/\nu})$, thereby verifying condition \ref{item:Uniform_FCLT_Cond2}. And finally, for each $\varepsilon \in (0,1]$, by Lemma \ref{Generic-gn} and direct computations, it holds
		\begin{align*}
			\limsup_{n \to \infty} \frac{h(\floor{n\varepsilon})}{n} &\leq 2  \left( \limsup_{n \to \infty} \frac{\gamma(\floor{n\varepsilon}, \Delta_\rho)}{n}  + C^2 \varepsilon \right)\\ 
			&\leq 2 (C_{\Delta_\rho} + C^2) \varepsilon,
		\end{align*}
		where $C_{\Delta_\rho}$ is from part (ii) of Lemma \ref{Generic-gn}. This verifies condition \ref{item:Uniform_FCLT_Cond4} and, thereby, concludes the proof of Theorem \ref{UniformFCLT}.	
\end{proof}

\subsection{Proofs of Section~\ref{Sec:Smoothed}}
\begin{proof}[Proof of Theorem \ref{UFCLT_smoothed-general}]
	We start by noting that for $0 \leq u \leq v \leq 1$ and $1 \leq i \le n,$ we have 
	\begin{align*}
		(i-1,i] \cap (nu, nv] = \begin{cases}
			(i-1,i], ~ \floor{nu} + 2 \leq i \leq \floor{nv}, \\
			\left(nu, \min\{\floor{nu} + 1, nv\} \right], ~i = \floor{nu} + 1,\\
			\left(\max\{nu, \floor{nv} \}, nv \right], ~i = \floor{nv} + 1,\\
			\emptyset,~ \text{ else,}
	\end{cases} \end{align*}
	which implies
	\begin{align}\label{Identity: SmoothedProcess}
		\mathbb Z_n^s ((u,v], f) &= \frac{1}{\sqrt{n}} \sum_{i = \floor{nu} + 2}^{\floor{nv}} ( f(X_{i,n}) - \E{f(X_{i,n})}) \nonumber \\
		&+ \frac{\min\{\floor{nu} + 1, nv\} - nu}{\sqrt{n}} ( f(X_{\floor{nu} + 1,n}) - \E{f(X_{\floor{nu} + 1,n})}) \nonumber \\
		&+ \frac{nv - \max\{nu, \floor{nv} \}}{\sqrt{n}} ( f(X_{\floor{nv} + 1,n}) - \E{f(X_{\floor{nv} + 1,n})}).
	\end{align}
	Furthermore, for each $n \in \mathbb N$ and $\delta > 0$, 
	\begin{align}\label{TwoSuprema_Smoothed}
		&\sup_{\tau( (u,v],f), (w,z],g)) \leq \delta } \left| \mathbb Z_n^s((u,v], f) - \mathbb Z_n^s((w,z], g) \right| \nonumber \\
		&\leq \sup_{(u,v] \in \mathcal A} \sup_{\rho(f,g) \leq \delta} \left| \mathbb Z_n^s((u,v], f) - \mathbb Z_n^s((u,v], g)\right| \nonumber \\
		&+\sup_{\lambda((u,v] \triangle (w,z]) \leq \delta} \sup_{f \in \mathcal F} \left| \mathbb Z_n^s((u,v], f) - \mathbb Z_n^s((w,z], f) \right|.
	\end{align}
	To treat the first term on the right side of \eqref{TwoSuprema_Smoothed}, note that for each $0 \leq u \leq v \leq 1$ and $f,g \in \mathcal F$ with $\rho(f,g) \leq \delta$, by \eqref{Identity: SmoothedProcess} and the triangle inequality, we have 
	\begin{align*}
		& \left| \mathbb Z_n^s((u,v], f) - \mathbb Z_n^s((u,v], g)\right| \\
		&\qquad  \leq n^{-\frac12} \left| \sum_{i = \floor{nu} + 2}^{\floor{nv}} ( (f-g)(X_{i,n}) - \E{(f-g)(X_{i,n})}) \right| \\
		&\qquad \qquad + 4 n^{-\frac12} \max_{k = 1,...,n} \sup_{f \in \mathcal F} \left|f(X_{k,n}) - \E{f(X_{k,n})} \right| \\
		&\qquad \leq n^{-\frac12} \left| \sum_{i = \floor{nu} + 1}^{\floor{nv}} ( (f-g)(X_{i,n}) - \E{(f-g)(X_{i,n})}) \right|\\
		&\qquad \qquad + 6 n^{-\frac12} \max_{k = 1,...,n} \sup_{f \in \mathcal F} \left|f(X_{k,n}) - \E{f(X_{k,n})} \right|\\
		&\qquad = \left|\mathbb Z_n(v, f-g) - \mathbb Z_n(u, f-g) \right| \\
		&\qquad \qquad + 6 n^{-\frac12} \max_{k = 1,...,n} \sup_{f \in \mathcal F} \left|f(X_{k,n}) - \E{f(X_{k,n})} \right|\\
		&\qquad \leq 2 \sup_{u \in [0,1]} \sup_{\rho(f,g) \leq \delta} \left|\mathbb Z_n(u, f-g) \right| \\
		&\qquad \qquad + 6 n^{-\frac12} \max_{k = 1,...,n} \sup_{f \in \mathcal F} \left|f(X_{k,n}) - \E{f(X_{k,n})} \right|.
	\end{align*}
	The proof of Theorem \ref{UniformFCLT-general} shows that $\sup_{u \in [0,1]} \sup_{\rho(f,g) \leq \delta} \left|\mathbb Z_n(u, f-g) \right|$ converges to $0$ in outer probability as $n \to \infty$ followed by $\delta \downarrow 0.$ For the remaining term, we make use of a union bound, Markov's inequality, and condition \ref{item:Uniform_FCLT_Cond2} of Theorem \ref{UniformFCLT-general}, which gives  
	\begin{align*}
		&P^*\left( n^{-\frac12} \max_{k = 1,...,n} \sup_{f \in \mathcal F} \left|f(X_{k,n}) - \E{f(X_{k,n})} \right| > \varepsilon \right) \\
		&\leq n \max_{k = 1,...,n} P^*\left( n^{-\frac12} \sup_{f \in \mathcal F} \left|f(X_{k,n}) - \E{f(X_{k,n})} \right| > \varepsilon \right) \\
		&\leq n^{1 - \nu/2} \varepsilon^{-\nu} \max_{k = 1,...,n} \EOuter{ \sup_{f \in \mathcal F} \left|f(X_{k,n}) - \E{f(X_{k,n})} \right|^\nu}\\
		&= n^{1 - \nu/2} \varepsilon^{-\nu} \max_{k = 1,...,n} \EOuter{ \norm{ S_{n,k,k}}_{\mathcal F}^\nu}\\ 
		&\leq n^{1 - \nu/2} \varepsilon^{-\nu} h^{\frac{\nu}{2}}(1),
	\end{align*}
	for each $\varepsilon > 0.$ Since $\nu > 2$, the last term converges to $0$ as $n \to \infty$, independent of $\delta$.
	
	It remains to discuss the second term on the right hand side of \eqref{TwoSuprema_Smoothed}, i.e. to show that 
	\begin{equation}\label{SecondSup_Smoothed}
		\limsup_{n \to \infty} P^*\left(\sup_{\lambda((u,v] \triangle (w,z]) \leq \delta} \sup_{f \in \mathcal F} \left| \mathbb Z_n^s((u,v], f) - \mathbb Z_n^s((w,z], f) \right| > \varepsilon \right) \to 0,
	\end{equation} 
	as $\delta \downarrow 0,$ for each $\varepsilon > 0.$ To this end, let $\varepsilon, \delta > 0,$ $n \in \mathbb N$, $f \in \mathcal F$ and $0 \leq u \leq v \leq 1$ as well as $0 \leq w \leq z \leq 1$ such that $\lambda((u,v] \triangle (w,z]) \leq \delta$. We distinguish three cases to calculate the Lebesgue disjunction explicitly.
	
	\noindent
	\underline{Case 1:} $(u,v] \cap (w,z] = \emptyset.$\\
	We have $\lambda((u,v] \triangle (w,z]) = |u-v| + |w-z| \leq \delta$, which implies $|u-v|, |w-z| \leq \delta$. By \eqref{Identity: SmoothedProcess}  and the triangle inequality, it holds
	\begin{align*}
		&\left| \mathbb Z_n^s ((u,v], f) - \mathbb Z_n^s((w,z], f) \right|\\
		&\leq n^{-\frac12} \left| \sum_{i = \floor{nu} + 1}^{\floor{nv}} ( f(X_{i,n}) - \E{f(X_{i,n})}) - \sum_{i = \floor{nw} + 1}^{\floor{nz}} ( f(X_{i,n}) - \E{f(X_{i,n})}) \right| \\
		&+ 6 n^{-\frac12} \max_{k = 1,...,n} \sup_{f \in \mathcal F} \left|f(X_{k,n}) - \E{f(X_{k,n})} \right| \\
		&\leq \left| \mathbb Z_n(v, f) - \mathbb Z_n(u, f) \right| + \left| \mathbb Z_n(z, f) - \mathbb Z_n(w, f) \right| \\
		&+ 6 n^{-\frac12} \max_{k = 1,...,n} \sup_{f \in \mathcal F} \left|f(X_{k,n}) - \E{f(X_{k,n})} \right| \\
		&\leq 2 \sup_{|u-v| \leq \delta} \sup_{f \in \mathcal F} \left| \mathbb Z_n(v, f) - \mathbb Z_n(u, f) \right|\\
		&+ 6 n^{-\frac12} \max_{k = 1,...,n} \sup_{f \in \mathcal F} \left|f(X_{k,n}) - \E{f(X_{k,n})} \right|.
	\end{align*}
	
	\noindent
	\underline{Case 2:} $(u,v] \cap (w,z] \neq \emptyset$, but neither $(u,v] \subset (w,z]$ nor $(v,z] \subset (u,v]$.\\
	Assume, without loss of generality, that $0 \leq u \leq w \leq v \leq z.$ We then have\\ $\lambda((u,v] \triangle (w,z]) = |u-w| + |v-z| \leq \delta$, which implies $|u-w|, |v-z| \leq \delta$. By \eqref{Identity: SmoothedProcess}  and the triangle inequality, it holds
	\begin{align*}
		&\left| \mathbb Z_n^s ((u,v], f) - \mathbb Z_n^s((w,z], f) \right|\\
		&\leq n^{-\frac12} \left| \sum_{i = \floor{nu} + 2}^{\floor{nv}} ( f(X_{i,n}) - \E{f(X_{i,n})}) - \sum_{i = \floor{nw} + 2}^{\floor{nz}} ( f(X_{i,n}) - \E{f(X_{i,n})}) \right| \\
		&+ 4 n^{-\frac12} \max_{k = 1,...,n} \sup_{f \in \mathcal F} \left|f(X_{k,n}) - \E{f(X_{k,n})} \right| \\
		&= n^{-\frac12} \left| \sum_{i = \floor{nu} + 2}^{\floor{nw} + 1} ( f(X_{i,n}) - \E{f(X_{i,n})}) - \sum_{i = \floor{nv} + 1}^{\floor{nz}} ( f(X_{i,n}) - \E{f(X_{i,n})}) \right| \\
		&+ 4 n^{-\frac12} \max_{k = 1,...,n} \sup_{f \in \mathcal F} \left|f(X_{k,n}) - \E{f(X_{k,n})} \right| \\
		&\leq n^{-\frac12} \left| \sum_{i = \floor{nu} + 1}^{\floor{nw}} ( f(X_{i,n}) - \E{f(X_{i,n})}) - \sum_{i = \floor{nv} + 1}^{\floor{nz}} ( f(X_{i,n}) - \E{f(X_{i,n})}) \right| \\
		&+ 6 n^{-\frac12} \max_{k = 1,...,n} \sup_{f \in \mathcal F} \left|f(X_{k,n}) - \E{f(X_{k,n})} \right| \\
		&\leq \left| \mathbb Z_n(w, f) - \mathbb Z_n(u, f) \right| + \left| \mathbb Z_n(z, f) - \mathbb Z_n(v, f) \right| \\
		&+ 6 n^{-\frac12} \max_{k = 1,...,n} \sup_{f \in \mathcal F} \left|f(X_{k,n}) - \E{f(X_{k,n})} \right| \\
		&\leq 2 \sup_{|u-w| \leq \delta} \sup_{f \in \mathcal F} \left| \mathbb Z_n(w, f) - \mathbb Z_n(u, f) \right|\\
		&+ 6 n^{-\frac12} \max_{k = 1,...,n} \sup_{f \in \mathcal F} \left|f(X_{k,n}) - \E{f(X_{k,n})} \right|.
	\end{align*}
	
	\noindent
	\underline{Case 3:} Either $(u,v] \subset (w,z]$ or $(v,z] \subset (u,v]$.\\
	Assume, without loss of generality, that $0 \leq w \leq u \leq v \leq z.$ Again, we have \\$\lambda((u,v] \triangle (w,z]) = |u-w| + |v-z| \leq \delta$, which implies $|u-w|, |v-z| \leq \delta$. Similar arguments as in case 2 yield
	\begin{align*}
		&\left| \mathbb Z_n^s ((u,v], f) - \mathbb Z_n^s((w,z], f) \right|\\
		&\leq n^{-\frac12} \left| \sum_{i = \floor{nw} + 2}^{\floor{nu} + 1} ( f(X_{i,n}) - \E{f(X_{i,n})}) + \sum_{i = \floor{nv} + 1}^{\floor{nz}} ( f(X_{i,n}) - \E{f(X_{i,n})}) \right| \\
		&+ 4 n^{-\frac12} \max_{k = 1,...,n} \sup_{f \in \mathcal F} \left|f(X_{k,n}) - \E{f(X_{k,n})} \right| \\
		&\leq 2 \sup_{|u-w| \leq \delta} \sup_{f \in \mathcal F} \left| \mathbb Z_n(w, f) - \mathbb Z_n(u, f) \right|\\
		&+ 6 n^{-\frac12} \max_{k = 1,...,n} \sup_{f \in \mathcal F} \left|f(X_{k,n}) - \E{f(X_{k,n})} \right|.
	\end{align*}
	
	By combining the cases 1-3, we conclude 
	\begin{align*}
		&\sup_{\lambda((u,v] \triangle (w,z]) \leq \delta} \sup_{f \in \mathcal F} \left| \mathbb Z_n^s((u,v], f) - \mathbb Z_n^s((w,z], f) \right| \\
		&\leq 2 \sup_{|u-v| \leq \delta} \sup_{f \in \mathcal F} \left| \mathbb Z_n(v, f) - \mathbb Z_n(u, f) \right|\\
		&+ 6 n^{-\frac12} \max_{k = 1,...,n} \sup_{f \in \mathcal F} \left|f(X_{k,n}) - \E{f(X_{k,n})} \right|.
	\end{align*}
	The proof of Theorem \ref{UniformFCLT-general} shows that $\sup_{|u-v| \leq \delta} \sup_{f \in \mathcal F} \left| \mathbb Z_n(v, f) - \mathbb Z_n(u, f) \right|$ converges to $0$ in outer probability as $n \to \infty$ followed by $\delta \downarrow 0.$ The remaining term has already been discussed. This shows \eqref{SecondSup_Smoothed} and, thereby, \eqref{DEF:AEC}. 
	
	Finally, since $(\mathcal A, d_\lambda)$ is totally bounded, where $d_\lambda(A,B) = \lambda(A \triangle B)$, which can be seen by covering $\mathcal A$ with the set of all intervals $(i/m, j/m]$, $0 \leq i < j \leq m, m \in \mathbb N,$ the proof of Theorem \ref{UFCLT_smoothed-general} can now be completed in the same way as the proof of Theorem \ref{UniformFCLT-general}.
\end{proof}

\subsection{Proofs of Section \ref{Section:Examples-NEU}}
\subsubsection{Proof of Theorem~\ref{SupremumInequalityUnderMixing}}
We start with a moment bound for strongly mixing arrays. The following result generalizes \cite[Lem.~4.1]{MO20}.

\begin{lemma}\label{MaximalInequalityUnderMixing}
	Let $\mathcal H$ be a set of Borel maps $\mathcal X \to \mathbb R.$ Assume there exist $\lambda, \tau > 0$ and an even integer $\nu \geq 2$ such that 
	\begin{itemize}
		\item[(i)] $\zeta(\alpha^X, \lambda, \nu) = \sum_{s = 1}^\infty s^{\nu-2} \alpha^X(s)^{\lambda/(2+\lambda)} < \infty,$
		
		\item[(ii)] $\E{\left|h(X_{t,n}) - \E{h(X_{t,n})}\right|^{l \frac{2+\lambda}{2}}} \leq \tau^{2+\lambda}$, for all $h \in \mathcal H$, $l = 2,...,\nu$ and $1 \leq t \leq n, n \in \mathbb N.$ 
	\end{itemize}
	
	Then there exists a constant $C \geq 0$ that only depends on $\nu$, $\lambda$ and the mixing coefficients such that for each $n \in \mathbb N$ and $1 \leq i \leq j \leq n$, with $m = j-i+1$, we have 
	$$ \sup_{h \in \mathcal H} \norm{S_{n,i,j}(h)}_{L_\nu} \leq C \sqrt{m} \max\left\{m^{-\frac12}, \tau \right\}.$$
\end{lemma}

The proof of Lemma \ref{MaximalInequalityUnderMixing} uses the following elementary fact concerning the strong mixing coefficients. Its proof is straightforward and therefore omitted. 
\begin{lemma}\label{MixingInvarianceTransformations}
	Let $(X_{i,n})$ be a triangular array of $\mathcal X$-valued random variables and let $(h_{i,n}): \mathcal X \to \mathbb R$ be Borel maps. For the array $(Y_{i,n}) = (h_{i,n}(X_{i,n}))$, we have $\alpha_n^Y(t) \leq \alpha_n^X(t)$ and $\alpha^Y(t) \leq \alpha^X(t)$, $t \in \mathbb N_0$. 
\end{lemma}

The proof of Lemma \ref{MaximalInequalityUnderMixing} makes repeated use of the covariance inequality stated in \cite[Lem.~4.2]{MO20} which we - for the required special case - restate here for convenience.
\begin{lemma}\label{L42Mohr}
	Let $(\xi_{i,n})$ be a strongly mixing triangular array of real random variables with mixing coefficients $(\alpha^\xi(t))_{t \in \mathbb N}$. For $n \in \mathbb N$, $m > 1$, and integers $(i_1,...,i_m)$ with $1 \leq i_1 < ... < i_m \leq n$ let $F_{n,i_1,...,i_m}$ denote the distribution function of $(\xi_{n,i_1},...,\xi_{n,i_m})$ and let $g: \mathbb R^m \to \mathbb R$ be a Borel function such that $$ \int |g(x_1,...,x_m)|^{1+\delta} dF_{n,i_1,...,i_m}(x_1,...,x_m) \leq M_n $$ and $$ \int |g(x_1,...,x_m)|^{1+\delta} dF_{n,i_1,...,i_j}(x_1,...,x_j) dF_{n,i_{j+1},...,i_m}(x_{j+1},...,x_m)\leq M_n $$ holds for some $\delta > 0.$ Then \begin{align*}
		&\bigg| \int g(x_1,...,x_m) dF_{n,i_1,...,i_m}(x_1,...,x_m)\\ 
		&- \int g(x_1,...,x_m) dF_{n,i_1,...,i_j}(x_1,...,x_j) d F_{n,i_{j+1},...,i_m}(x_{j+1},...,x_m) \bigg| \\
		&\leq 4 M_n^{\frac{1}{1+\delta}} \alpha^\xi(i_{j+1} - i_j)^{\frac{\delta}{1+\delta}}.
	\end{align*}
\end{lemma}

\begin{proof}[Proof of Lemma \ref{MaximalInequalityUnderMixing}]
	We can follow the proof of \cite[Lem.~4.1]{MO20}, with some modifications and refinements. For the sake of clarity, we provide the details.
	Let $h \in \mathcal H$ and $n \in \mathbb N$. Let furthermore $1 \leq i \leq j \leq n$ and put $m = j-i+1$. As $h \in \mathcal H$ is fixed throughout the proof, we abbreviate $Z_{t,n} = h(X_{t,n}) - \E{h(X_{t,n})}$ and  $S_{n,i,j} = S_{n,i,j}(h).$
	
	It suffices to show that for all integers $\nu \geq 2$ (not necessarily even) for which the assumptions (i) and (ii) hold, there is a constant $C_0 \geq 0$ that depends only on $\nu$, $\lambda$ and and the mixing coefficients $(\alpha^X(t))_{t \in \mathbb N}$, such that\begin{equation}\label{WeakMultiplicativitiyForQ}
		\sum_{\mathbf t \in \mathbf{ T_{\nu;i,j}}} \left| \E{Z_{t_1,n},...,Z_{t_\nu,n}} \right| \leq C_0 \left(m \tau^2 + ... + (m\tau^2)^{\floor{\frac{\nu}{2}}} \right),
	\end{equation}
	where the sum is taken over the set $$ \mathbf{T}_{\nu;i,j} = \left\{ \mathbf t = (t_1,...,t_\nu) \in \{i,...,j\}^\nu ~|~ t_1 \leq ... \leq t_\nu \right\}.$$ Then the proof can be completed as follows: for each $l = 1,..., \floor{\nu/2}$, it holds $$ (m \tau^2)^l \leq \max\{1, (m\tau^2)^{\floor{\frac{\nu}{2}}}\} ,$$ and thus, with respect to our even $\nu \geq 2$, we have 
	\begin{align*}
		\norm{S_{n,i,j}}_{L_\nu} &= \norm{\sum_{t = i}^{j} Z_{t,n} }_{L_\nu} \\
		&\leq (\nu!)^{\frac{1}{\nu}} \left(\sum_{\mathbf t \in \mathbf{T_{\nu;i,j}}} \left| \E{Z_{t_1,n}\cdot ... \cdot Z_{t_\nu,n}} \right| \right)^{\frac{1}{\nu}} \\
		&\leq C \max\{1, \sqrt{m} \tau\} = C \sqrt{m} \max\{m^{-\frac12}, \tau\},
	\end{align*}
	for $C = (C_0 \nu! \nu/2)^{1/\nu},$ where $C_0$ is from \eqref{WeakMultiplicativitiyForQ}. 
	
	It remains to show  \eqref{WeakMultiplicativitiyForQ}, which is accomplished by induction over $\nu \geq 2$ with the help of the covariance inequality stated in Lemma \ref{L42Mohr}. So, let first $\nu = 2$. By Lemma \ref{MixingInvarianceTransformations}, the array $(Z_{t,n})$ is strongly mixing with $\alpha^{Z}(s) \leq \alpha^X(s), s \in \mathbb N$, and by H{\"o}lder's inequality and condition (ii), we have $$ \E{\left|Z_{t_1,n} Z_{t_2,n} \right|^{1 + \frac{\lambda}{2}}} = \E{|Z_{t_1,n} Z_{t_2,n}|^{\frac{2+\lambda}{2}}} \leq \tau^{2+\lambda},$$ for any $i \leq t_1 < t_2 \leq j$. Since the above estimate remains valid if either $Z_{t_1,n}$ or $Z_{t_2,n}$ is replaced by an independent copy, Lemma \ref{L42Mohr} is applicable with $g(x_1,x_2) = x_1 \cdot x_2,$ $\delta = \lambda/2$ and $M_n = \tau^{2+\lambda}$, from which we obtain (since the $Z_{t,n}$ are centered)
	\begin{align*}
		&\sum_{\mathbf t \in \mathbf{T_{2;i,j}}} \left| \E{Z_{t_1,n} Z_{t_2,n}}  \right| \\
		&= \sum_{t_1 = i}^{j} \E{|Z_{t_1,n}|^2} + \sum_{i\leq t_1 < t_2 \leq j} \left| \E{Z_{t_1,n} Z_{t_2,n}}  \right| \\
		&\leq \sum_{t_1 = i}^{j} \E{|Z_{t_1,n}|^2} + \sum_{i\leq t_1 < t_2 \leq j} 4 \tau^2 \alpha^{Z}(t_2 - t_1)^{\frac{\lambda}{2 + \lambda}} \\
		&\leq \sum_{t_1 = i}^{j} \E{|Z_{t_1,n}|^2} + 4 \tau^2 m \sum_{s = 1}^\infty \alpha^X(s)^{\frac{\lambda}{2 + \lambda}} \\
		&\leq m \tau^2 \left(1 + 4 \zeta(\alpha^X, \lambda,2) \right),
	\end{align*}
	where we have used condition (ii) in the last step.
	
	Next, let $\nu > 2$ be an arbitrary integer and assume that the assertion holds  for all $r = 2,....,\nu-1$. To show that it holds for $\nu$ as well, we decompose the sum over $\mathbf{T}_{\nu;i,j}$. For $\mathbf t = (t_1,...,t_\nu) \in \mathbf{T}_{\nu;i,j}$ let $$ G(\mathbf t) = \max \left\{t_{l+1} - t_l ~|~ l = 1,...,\nu -1 \right\}$$ 
	indicate the largest gap between any two consecutive entries $t_{l}, t_{l+1}$, and let
	$$k(\mathbf t) = \min \left\{l \in \{1,..., \nu-1\} ~|~ t_{l+1} - t_l = G(\mathbf t) \right\}$$ indicate its first occurence in $\mathbf t = (t_1,...,t_\nu)$. Note that $G(\mathbf t) = 0$ implies that all indices in $\mathbf t$ are equal. For those $\mathbf t $ with $G(\mathbf t ) > 0$, the idea is to identify the entry $t_l$ at which the largest gap appears (for the first time) and insert a zero by adding and subtracting the term $$ \E{Z_{t_1,n} ...Z_{t_l,n}}\E{Z_{t_{l+1},n}...Z_{t_\nu,n}} $$
	at this point. This results in one term to which Lemma \ref{L42Mohr} can be applied and another term that can be treated with the induction hypothesis. That is, by the triangle inequality and condition (ii), we have
	\begin{align}\label{FinalBound}
		&\sum_{\mathbf t \in \mathbf{T_{\nu;i,j}}} \left| \E{Z_{t_1,n} \cdot ... \cdot Z_{t_\nu,n}} \right| \nonumber \\
		&\quad \leq \sum_{t_1 = i}^{j} \E{\left| Z_{t_1,n}\right|^{\nu}} + \sum_{r = 1}^{\nu-1} \sum_{\mathbf t: G(\mathbf t) > 0, k(\mathbf t) = r} \left| \E{Z_{t_1,n} ... Z_{t_\nu,n}} \right| \nonumber \\
		&\quad\leq m\tau^2 \nonumber \\
		&\qquad + \sum_{r = 1}^{\nu-1} \sum_{\mathbf t: G(\mathbf t) > 0, k(\mathbf t) = r} \left| \E{Z_{t_1,n} ... Z_{t_\nu,n}} - \E{Z_{t_1,n} ...Z_{t_r,n}}\E{Z_{t_{r+1},n}...Z_{t_\nu,n}} \right| \nonumber\\ 
		&\qquad+ \sum_{r = 1}^{\nu-1} \sum_{\mathbf t: G(\mathbf t) > 0, k(\mathbf t) = r} \left| \E{Z_{t_1,n}...Z_{t_r,n}}\E{Z_{t_{r+1},n}...Z_{t_\nu,n}} \right| \nonumber\\
		&\quad= m\tau^2 + B_{1,n,i,j} + B_{2,n,i,j},
	\end{align}
	where
	\begin{align*}
		B_{1,n,i,j}	&= \sum_{r = 1}^{\nu-1} \sum_{\mathbf t: G(\mathbf t) > 0, k(\mathbf t) = r} \left| \E{Z_{t_1,n} ... Z_{t_\nu,n}} - \E{Z_{t_1,n} ...Z_{t_r,n}}\E{Z_{t_{r+1},n}...Z_{t_\nu,n}} \right| \nonumber,\\ 
		B_{2,n,i,j}	&= \sum_{r = 2}^{\nu-2} \sum_{\mathbf t: G(\mathbf t) > 0, k(\mathbf t) = r} \left| \E{Z_{t_1,n}...Z_{t_r,n}}\E{Z_{t_{r+1},n}...Z_{t_\nu,n}} \right|.
	\end{align*}
	Here, we used that for $r \in \{1,\nu-1\}$, because the $Z_{t,n}$ are centered, 
	$$ \left| \E{Z_{t_1,n}...Z_{t_r,n}}\E{Z_{t_{r+1},n}...Z_{t_\nu,n}} \right| = 0.$$ We now estimate the sums $B_{1,n,i,j}$ and $B_{2,n,i,j}$ separately. To bound $B_{1,n,i,j}$, we can use Lemma \ref{L42Mohr}. That is, if the $ t_1, ..., t_\nu $ are pairwise different, we can argue as in the proof of \cite[Lem.~4.1]{MO20} and apply Lemma \ref{L42Mohr} with $g(x_1,...,x_\nu) = x_1\cdot...\cdot x_\nu$, which gives 
	\begin{align*}
		&\left| \E{Z_{t_1,n} ... Z_{t_\nu,n}} - \E{Z_{t_1,n} ...Z_{t_r,n}}\E{Z_{t_{r+1},n}...Z_{t_\nu,n}} \right|\\
		&\leq 4 \tau^2 \alpha^X(t_{r+1} - t_r)^{\frac{\lambda}{2+\lambda}}.
	\end{align*} 
	The case of repetitions in $ t_1, ..., t_\nu$, which has not been discussed in \cite{MO20}, can be treated as follows:  We can group the indices and write $$ Z_{t_1,n}...Z_{t_\nu,n} = (Z_{k_1,n})^{p_1}...(Z_{k_L,n})^{p_L},$$ 
	for certain pairwise different indices $i \leq k_1 < ... < k_L \leq j$ and powers $p_1,...,p_L \in \{1,...,\nu-1\}$, for some $L \leq \nu$. Clearly, the gap $t_{r+1} \neq t_r$ is retained during this procedure, i.e. there is $l \in \{1,...,L-1\}$ with $k_{l+1} - k_l = t_{r+1} - t_r$ and hence $ \E{Z_{t_1,n}...Z_{t_r,n}}\E{Z_{t_{r+1},n}...Z_{t_\nu,n}}$ equals	\begin{align*}
		\E{(Z_{k_1,n})^{p_1}...(Z_{k_l,n})^{p_l}}\E{(Z_{k_{l+1},n})^{p_{l+1}}...(Z_{k_L,n})^{p_L}}.
	\end{align*} 
	Now let $$\xi_{t,n} = \begin{cases*}
		(Z_{k_l,n})^{p_l}, ~ t = k_l \text{ for an } l \in \{1,...,L\},\\
		Z_{t,n}, ~ \text{ else,}
	\end{cases*}$$
	and put $\xi_{t,N} = Z_{t,N}$ for all $N \neq n$, $t = 1,...,N$. We then have $(\xi_{t,n}) = (h_{t,n}(Z_{t,n}))$ for suitable Borel functions $(h_{t,n})$, which, by Lemma \ref{MixingInvarianceTransformations}, entails $\alpha^\xi(s) \leq \alpha^{Z}(s) \leq \alpha^X(s)$ for all $s \in \mathbb N$. Furthermore, by the generalized version of H{\"o}lder's inequality and condition (ii), we have
	$$ \E{|\xi_{k_1,n}...\xi_{k_L,n}|^{\frac{2+\lambda}{2}}} = \E{|Z_{t_1,n}...Z_{t_\nu,n}|^{\frac{2+\lambda}{2}}} \leq \prod_{l = 1}^{\nu} \norm{|Z_{t_l,n}|^{\frac{2+\lambda}{2}}}_{L_\nu} \leq \tau^{2+\lambda},$$ and the same holds if some of the $\xi_{k_l,n}$ are replaced by independent copies. Hence, Lemma \ref{L42Mohr} is applicable with $g(x_1,...,x_L) = x_1 \cdot ... \cdot x_L$, $\delta = \lambda/2$ and $M_n = \tau^{2+\lambda}$, from which we obtain
	\begin{align*}
		&\left| \E{Z_{t_1,n}...Z_{t_\nu,n}} - \E{Z_{t_1,n}...Z_{t_r,n}}\E{Z_{t_{r+1},n}...Z_{t_\nu,n}} \right| \\
		&= \left| \E{\xi_{k_1,n}...\xi_{k_L,n}} - \E{\xi_{k_1,n}...\xi_{k_l,n}}\E{\xi_{k_{l+1},n}...\xi_{k_L,n}} \right| \\
		&\leq 4 \tau^2 \alpha^\xi(k_{l+1} - k_l)^{\frac{\lambda}{2+\lambda}} \\
		&\leq 4 \tau^2 \alpha^X(t_{r+1} - t_r)^{\frac{\lambda}{2+\lambda}}.
	\end{align*}
	It follows
	\begin{align}\label{BoundforBn1}
		B_{1,n,i,j} &\leq 4\tau^2 \sum_{r = 1}^{\nu-1} \sum_{\mathbf t: G(\mathbf t) > 0, k(\mathbf t) = r} \alpha^X(t_{r+1} - t_r)^{\frac{\lambda}{2+\lambda}}.
	\end{align}
	Now, for any $\mathbf t = (t_1,...,t_\nu) \in \mathbf{T}_{\nu;i,j}$ with $G(\mathbf t) > 0$ and $k(\mathbf t) = r$ let $l(\mathbf t) \in \{i,...,j-1\}$ denote the index at which the largest gap occurs for the first time. Group the tuples $\mathbf t$ according to this location and the gap size $G(\mathbf t) = s \in \{1,...,m-1\}$, which gives
	$$ \sum_{\mathbf t: G(\mathbf t) > 0, k(\mathbf t) = r} \alpha^X(t_{r+1} - t_r)^{\frac{\lambda}{2+\lambda}} \leq \sum_{l = i}^{j} \sum_{s = 1}^{m} \sum_{\mathbf t \in \mathbf{T}_{\nu;i,j}^{r,l,s}} \alpha^X(s)^{\frac{\lambda}{2+\lambda}},$$ where we have put $$ \mathbf{T}_{\nu;i,j}^{r,l,s} = \left\{\mathbf t = (t_1,...,t_\nu) \in \mathbf{T}_{\nu;i,j} ~|~  k(\mathbf t) = r, l(\mathbf t) = l, G(\mathbf t) = s \right\}. $$ $\#(\mathbf{T}_{\nu;i,j}^{r,l,s})$ can be bounded as follows: By definition, for any $\mathbf t = (t_1,...,t_\nu) \in \mathbf{T}_{\nu;i,j}^{r,l,s}$, we have $t_1 \leq ... \leq t_\nu$, $t_r = l$ and $t_{r+1} = l+s$. Since $t_r$ is the smallest entry at which a gap of size $s$ occurs, there are at most $s$ different values that $t_{r-1}$ can take, i.e. $t_{r-1} \in \{l-s+1,...,l\}$. Analogously, there are at most $s$ different values that $t_{r-2}$ can take. Proceeding in this fashion, we conclude that there are at most $s^{r-1}$ possible values for the first $r-1$ entries $t_1,...,t_{r-1}$. Analogous reasoning shows that there are at most $(s+1)^{\nu-r-1}$ possibilities for the entries $t_{r+2},...,t_\nu$ (as gaps of size $s$ can occur among these entries). Hence, we can conclude that $\#(\mathbf{T}_{\nu;i,j}^{r,l,s}) \leq s^{r-1}(s + 1)^{\nu-r-1}$, which entails 
	\begin{align}\label{BoundforBn1Part2}
		\sum_{l = i}^{j} \sum_{s = 1}^{m} \sum_{\mathbf t \in \mathbf{T_{\nu;i,j}^{r,l,s}}} \alpha^X(s)^{\frac{\lambda}{2+\lambda}} &\leq \sum_{l = i}^{j} \sum_{s = 1}^{m} s^{r-1}(s + 1)^{\nu-r-1}  \alpha^X(s)^{\frac{\lambda}{2+\lambda}} \nonumber \\
		&\leq m \sum_{s = 1}^\infty (s+1)^{\nu-2} \alpha^X(s)^{\frac{\lambda}{2+\lambda}} \nonumber\\
		&\leq 2^{\nu-2} m \zeta(\alpha^X, \lambda, \nu).
	\end{align} 
	By inserting \eqref{BoundforBn1Part2} into \eqref{BoundforBn1}, we obtain
	\begin{equation}\label{BoundforB1}
		B_{1,n,i,j} \leq 2^{\nu} m \tau^2 (\nu-1) \zeta(\alpha^X, \lambda, \nu). 
	\end{equation}
	
	It remains to treat $B_{2,n,i,j}$. Denote $$ M_m(r) = m\tau^2 + ... + (m\tau^2)^{\floor{\frac{r}{2}}}, ~ r = 2,...,\nu,$$ and note that if $\mathbf t = (t_1,...,t_\nu) \in \mathbf{T}_{\nu;i,j}$ fulfills $G(\mathbf t) > 0$ and $k(\mathbf t) = r$, then $\mathbf t \in \mathbf{T}_{r; i,j} \times \mathbf{T}_{\nu-r; i,j}$, which entails $$\#\left( \left\{ \mathbf t \in \mathbf{T}_{\nu;i,j} ~|~ G(\mathbf t) > 0, k(\mathbf t) = r \right\} \right) \leq \#(\mathbf{T}_{r; i,j}) \#(\mathbf{T}_{\nu-r; i,j}).$$ Combined with the induction hypothesis, this implies 
	\begin{align*}
		B_{2,n,i,j} &= \sum_{r = 2}^{\nu-2} \sum_{\mathbf t: G(\mathbf t) > 0, k(\mathbf t) = r} \left| \E{Z_{t_1,n}...Z_{t_r,n}}\E{Z_{t_{r+1},n}...Z_{t_\nu,n}} \right| \\
		&\leq \sum_{r = 2}^{\nu-2} \sum_{\mathbf t \in \mathbf{T}_{r;i,j}} \left| \E{Z_{t_1,n}...Z_{t_r,n}} \right| \sum_{\mathbf t \in \mathbf{T}_{\nu-r;i,j}} \left| \E{Z_{t_{r+1},n}...Z_{t_\nu,n}} \right| \\
		&\leq \sum_{r = 2}^{\nu-2} C_r M_m(r) C_{\nu-r} M_m(\nu-r)
	\end{align*}
	for constants $C_r, r = 2,...,\nu-2$, that only depend on $r$, $\lambda$ and the mixing coefficients $(\alpha^X(t))_{t \in \mathbb N}$. For $r = 2,...,\nu-2$, $$ M_m(r) M_m(\nu-r) $$ is a polynomial in $m\tau^2$ of degree $$ \floor{\frac{r}{2}} \floor{\frac{\nu-r}{2}} \leq \floor{\frac{\nu}{2}},$$ and therefore there exist constants $C_{r,\nu}$ that only depend on $r$ and $\nu$ such that $$ M_m(r) M_m(\nu-r) \leq C_{r,\nu} M_m(\nu), ~ r = 2,...,\nu-2.$$ It follows
	\begin{equation}\label{BoundforB2}
		B_{2,n,i,j} \leq M_m(\nu) \sum_{r = 2}^{\nu-2} C_r C_{\nu-r} C_{r,\nu}.
	\end{equation} 
	
	By inserting \eqref{BoundforB2} and \eqref{BoundforB1} into \eqref{FinalBound}, we conclude
	\begin{align*}
		&\sum_{\mathbf t \in \mathbf{T_{\nu;i,j}}} \left| \E{Z_{t_1,n}...Z_{t_\nu,n}} \right| \\
		&\leq m\tau^2 \left(1 + 2^{\nu}(\nu-1) \zeta(\alpha^X,\lambda,\nu) \right) + M_m(\nu) \sum_{r = 2}^{\nu-2} C_r C_{\nu-r} C_{r,\nu}\\
		&\leq C M_m(\nu),
	\end{align*}
	for $$ C = 1 + 2^{\nu}(\nu-1) \zeta(\alpha^X,\lambda,\nu) + \sum_{r = 2}^{\nu-2} C_r C_{\nu-r} C_{r,\nu},$$ which is \eqref{WeakMultiplicativitiyForQ} with $C_0 = C.$ This concludes the proof.
\end{proof}

The proof of Theorem makes repeated use of Lemma \ref{MaximalInequalityUnderMixing} and the simple estimate
\begin{equation}\label{MaxBound}
	\norm{\max_{1 \leq i \leq N} |U_i| }_{L_p} \leq N^{\frac{1}{p}} \max_{1 \leq i \leq N} \norm{U_i}_{L_p}
\end{equation}
which is valid for any $p \geq 1$ and real random variables $U_1,...,U_N$ \cite[Lem.~2.2.2]{vdVW23}.

\begin{proof}[Proof of Theorem \ref{SupremumInequalityUnderMixing}]
	Let $\eta, \delta > 0$, $n \in \mathbb N$ and $1 \leq i \leq j \leq n$. Put $m = j-i+1$. Abbreviating $N_{[]}(\eta) := N_{[]}(\eta, \mathcal F, \rho_2)$, we shall show that
	\begin{align}\label{SupremumInequality_ETA}
		&\norm{\sup_{\rho(f-g) \leq \delta} |S_{n,i,j}(f) - S_{n,i,j}(g)|}_{L_{\nu}}^* \nonumber \\ 
		&\leq C_0 \sqrt{m} \left(N_{[]}^{\frac{2}{\nu}}\left(\eta\right) \left(m^{-\frac12} + \delta + \delta^{\frac{\nu}{2}} \right) + m^{-\frac{\lambda}{4}} + \int_0^{\eta} N_{[]}^\frac{1}{\nu}(\varepsilon) \varepsilon^{-\frac{\lambda}{2+\lambda}} d\varepsilon \right)
	\end{align}  
	for a constant $C_0$ that only depends on $\nu, \lambda$ and the mixing coefficients, from which the result follows by taking $0 < \kappa < \frac12 \land  \frac{\lambda}{4}.$
	
	For $k \in \mathbb N_0$ let $$ \eta_k = \eta 2^{-k}, ~~ \tau_k = \eta_k^{\frac{2}{2+\lambda}}, ~~ N_k = N_{[]}(\eta_k).$$ By (ii), for each $k \in \mathbb N_0$, we may choose $\mathcal J_k, \mathcal K_k$ with $\#(\mathcal J_k) = \#(\mathcal K_k) = N_k$ such that for each $f \in \mathcal F$, there exists $a_k(f) \in \mathcal J_k$ and $b_k(f) \in \mathcal K_k$ with $|f - a_k(f)| \leq b_k(f)$ and 
	\begin{equation}\label{ConditionOnBoundingClassKk}
		\max_{b \in \mathcal K_k} \left\{ \rho_2(b), \max_{l = 2,...,\nu} \sup_{1 \leq t \leq n, n \in \mathbb N} \E{|b(X_{t,n})|^{l\frac{2+\lambda}{2}}}^{\frac12} \right\} \leq \eta_k.
	\end{equation} 
	We pursue an idea similar to the one used in the proof of \cite[Thm.~2.5]{MO20} and show that, for each $f \in \mathcal F$, there exists $a_0^f \in \mathcal J_0$ (where, not necessarily, $a_0(f) = a_0^f$) with 
	\begin{align}\label{ApproxByJ0}
		&\norm{\sup_{f \in \mathcal F} |S_{n,i,j}(f) - S_{n,i,j}(a_0^f)|}_{L_\nu}^* \nonumber \\ 
		&\leq C_1 \sqrt{m} \left(N_0^{\frac{1}{\nu}} m^{-\frac12} + m^{-\frac{\lambda}{4}} + \int_{0}^{\eta} N_{[]}^\frac{1}{\nu}(\varepsilon) \varepsilon^{-\frac{\lambda}{2+\lambda}} d\varepsilon \right) 
	\end{align}
	for a constant $C_1$ that only depends on $\nu, \lambda$ and the mixing coefficients. Let us tentatively assume that \eqref{ApproxByJ0} is shown. Then the proof can be completed by taking similar steps as in the proof of \cite[Thm.~2.5]{MO20}. For the sake of completeness, we repeat some arguments. Define an equivalence relation on $\mathcal F$ by 
	$$ f \sim g \Leftrightarrow a_0^f = a_0^g, ~(f,g) \in \mathcal F^2, $$ which induces a partition of $\mathcal F$ into $N_0$ classes we denote by $\mathcal E_r, r = 1,...,N_0.$ Since $a_0^f = a_0^g,$ we then have
	\begin{align}\label{ApproxWithinEquivalenceClasses}
		&\norm{\max_{1 \leq r \leq N_0} \sup_{f,g \in \mathcal E_r} |S_{n,i,j}(f) - S_{n,i,j}(g)|}_{L_\nu}^* \nonumber \\
		&\leq 2 \norm{\sup_{f \in \mathcal F} |S_{n,i,j}(f) - S_{n,i,j}(a_0^f)|}_{L_\nu}^*,
	\end{align}
	by Lemma \ref{MinkowskiForOuter}. Now let $$ d(\mathcal E_r, \mathcal E_s) = \inf\left\{ \rho(f-g) ~|~ f \in \mathcal E_r, g \in \mathcal E_s \right\}, ~ r,s = 1,...,N_0,$$ and choose, for each $(r,s) \in \{1,...,N_0\}^2$, functions $\phi_{r,s} \in \mathcal E_r, \psi_{s,r} \in \mathcal E_s$ with $$ \rho(\phi_{r,s}- \psi_{s,r}) \leq d(\mathcal E_r, \mathcal E_s) + \delta.$$ Given that, for $f, g \in \mathcal F$ with $\rho(f-g) \leq \delta$, we then have, for $(r,s) \in \{1,...,N_0\}^2$ such that $f \in \mathcal E_r, g \in \mathcal E_s$,
	\begin{align} \nonumber 
		&|S_{n,i,j}(f) - S_{n,i,j}(g)| \\
		&\leq  |S_{n,i,j}(f) - S_{n,i,j}(\phi_{r,s})| + |S_{n,i,j}(g) - S_{n,i,j}(\psi_{s,r})| \nonumber  \\
		&+ |S_{n,i,j}(\phi_{r,s}) - S_{n,i,j}(\psi_{s,r})|\nonumber   \\
		&\leq 2 \max_{1 \leq r \leq N_0} \sup_{f,g \in \mathcal E_r} |S_{n,i,j}(f) - S_{n,i,j}(g)| \nonumber \\
		&+ \max_{1 \leq r,s \leq N_0, \rho(\phi_{r,s} - \psi_{s,r}) \leq 2\delta} |S_{n,i,j}(\phi_{r,s}) - S_{n,i,j}(\psi_{s,r})|. \label{MaxMeas}
	\end{align}
	Since $ \#\left( \left\{1 \leq r,s \leq N_0 ~|~ \rho(\phi_{r,s} - \psi_{s,r}) \leq 2\delta \right\} \right) \leq N_0^2, $ we have 
	\begin{align}\label{BoundingTheMax}
		&\norm{ \max_{1 \leq r,s \leq N_0, \rho(\phi_{r,s} - \psi_{s,r}) \leq 2\delta} |S_{n,i,j}(\phi_{r,s}) - S_{n,i,j}(\psi_{s,r})| }_{L_\nu} \nonumber \\
		&\qquad \leq N_0^\frac{2}{\nu} \max_{1 \leq r,s \leq N_0, \rho(\phi_{r,s} - \psi_{s,r}) \leq 2\delta} \norm{S_{n,i,j}(\phi_{r,s}) - S_{n,i,j}(\psi_{s,r})}_{L_\nu} 
	\end{align}
	by \eqref{MaxBound}. Hence, by applying Lemma \ref{MinkowskiForOuter} to the term \eqref{MaxMeas} and using \eqref{ApproxWithinEquivalenceClasses} and \ref{BoundingTheMax}, we obtain
	\begin{align}\label{FinalSupremum}
		&\norm{\sup_{\rho(f-g) \leq \delta} |S_{n,i,j}(f) - S_{n,i,j}(g)|}_{L_\nu}^* \nonumber \\
		&\leq 4 \norm{\sup_{f \in \mathcal F}|S_{n,i,j}(f) - S_{n,i,j}(a_0^f)| }_{L_\nu}^*  \nonumber \\
		&+N_0^{\frac{2}{\nu}} \max_{1 \leq r,s \leq N_0, \rho(\phi_{r,s} - \psi_{s,r}) \leq 2\delta} \norm{S_{n,i,j}(\phi_{r,s}) - S_{n,i,j}(\psi_{s,r})}_{L_\nu}.
	\end{align}
	It remains to estimate the maximum in the last line of the above display. Denote, for $1 \leq r,s \leq N_0$, $$Z_{t,n}(\phi_{r,s} - \psi_{s,r}) = \phi_{r,s}(X_{t,n}) - \psi_{s,r}(X_{t,n}) - \E{\phi_{r,s}(X_{t,n}) - \psi_{s,r}(X_{t,n})}.$$  By Jensen's inequality and the definition of $\rho$, we have  
	\begin{align*}
		\E{|Z_{t,n}(\phi_{r,s} - \psi_{s,r})|^{l\frac{2+\lambda}{2}}} &\leq 2^{l\frac{2+\lambda}{2}} \E{|\phi_{r,s}(X_{t,n}) - \psi_{s,r}(X_{t,n})|^{l\frac{2+\lambda}{2}}} \\
		&\leq 2^{\nu\frac{2+\lambda}{2}} \rho(\phi_{r,s} - \psi_{s,r})^{l\frac{2+\lambda}{2}} \\
		&\leq (2^\nu \max\{\delta, \delta^{\frac{\nu}{2}} \} )^{2+\lambda},
	\end{align*}
	for each $l = 2,...,\nu$. Thus, Lemma \ref{MaximalInequalityUnderMixing} with $ \tau = 2^\nu \max\{\delta, \delta^{\frac{\nu}{2}} \}$ gives 
	\begin{align}\label{FinalMaximum}
		&\max_{r,s: \rho(\phi_{r,s} - \psi_{s,r}) \leq 2\delta} \norm{S_{n,i,j}(\phi_{r,s}) - S_{n,i,j}(\psi_{s,r})}_{L_\nu}  \nonumber\\
		&\leq C_2 \sqrt{m} \max\{ m^{-\frac12}, 2^\nu \max\{\delta, \delta^{\frac{\nu}{2}} \} \} \nonumber \\ 
		&\leq C_2 \sqrt{m} \left(m^{-\frac12} + 2^\nu (\delta + \delta^{\frac{\nu}{2}}) \right),
	\end{align} 
	for a constant $C_2$ that only depends on $\nu, \lambda$ and the mixing coefficients. Inserting \eqref{FinalMaximum} and \eqref{ApproxByJ0} into \eqref{FinalSupremum}, we arrive at  
	\begin{align*}
		&\norm{\sup_{\rho(f-g) \leq \delta} |S_{n,i,j}(f) - S_{n,i,j}(g)|}_{L_\nu}^* \\
		&\leq C_2\sqrt{m} \left(N_0^{\frac{2}{\nu}} \left( m^{-\frac12} + 2^\nu (\delta + \delta^{\frac{\nu}{2}}) \right) \right) \\
		& + 4 C_1 \sqrt{m} \left(N_0^{\frac{1}{\nu}} m^{-\frac12} + m^{-\frac{\lambda}{4}} + \int_{0}^{\eta} N_{[]}^\frac{1}{\nu}(\varepsilon) \varepsilon^{-\frac{\lambda}{2+\lambda}} d\varepsilon \right) \\
		&\leq C_3 \sqrt{m} \left(N_0^{\frac{2}{\nu}} \left(m^{-\frac12} + \delta + \delta^{\frac{\nu}{2}} \right) + m^{-\frac{\lambda}{4}} + \int_0^{\eta} N_{[]}^\frac{1}{\nu}(\varepsilon) \varepsilon^{-\frac{\lambda}{2+\lambda}} d\varepsilon \right), 
	\end{align*}
	for a constant $C_3$ that only depends on $\nu,\lambda$ and the mixing coefficients, which proves \eqref{SupremumInequality_ETA} (with $C_0 = C_3$). 
	
	It therefore remains to verify \eqref{ApproxByJ0}. To do so, we distinguish two cases.
	
	\noindent
	\underline{Case 1:} $\tau_0 = \eta^{2/(2+\lambda)} \leq m^{-1/2}$.\\
	We have $$ \sqrt{m} \eta \leq \sqrt{m} (m^{-\frac12})^{\frac{2+\lambda}{2}} = m^{-\frac{\lambda}{4}}$$ and for each $f \in \mathcal F$, there exists $a_0(f) \in \mathcal J_0$ and $b_0(f) \in \mathcal K_0$ with
	\begin{align*}
		&|S_{n,i,j}(f) - S_{n,i,j}(a_0(f)) | \\
		&\leq S_{n,i,j}(b_0(f)) + 2\sum_{t = i}^{j} \E{b_0(f)(X_{t,n})} \\
		&\leq \sup_{f \in \mathcal F} |S_{n,i,j}(b_0(f))| + 2m \sup_{f \in \mathcal F} \rho_2(b_0(f)).
	\end{align*}  
	Note that both suprema in the last line of the above display are taken over finite sets, and that $\sup_{f \in \mathcal F} \rho_2(b_0(f)) \leq \eta_0 = \eta$, by \eqref{ConditionOnBoundingClassKk}. Hence, Lemma \ref{MinkowskiForOuter} and the estimate \eqref{MaxBound} entail
	\begin{align}\label{ApproxForK0}
		\norm{\sup_{f \in \mathcal F} \left|S_{n,i,j}(f) - S_{n,i,j}(a_0(f)) \right|}_{L_\nu}^* \nonumber 
		&\leq \norm{\sup_{f \in \mathcal F} |S_{n,i,j}(b_0(f))| }_{L_\nu} + 2\sqrt{m} \sqrt{m} \eta \nonumber \\
		&\leq N_0^{\frac{1}{\nu}} \sup_{f \in \mathcal F} \norm{S_{n,i,j}(b_0(f))}_{L_\nu} + 2\sqrt{m} m^{-\frac{\lambda}{4}}.
	\end{align}
	Next, we use Lemma \ref{MaximalInequalityUnderMixing} to bound $\sup_{f \in \mathcal F} \norm{S_{n,i,j}(b_0(f))}_{L_\nu}$. Let $Z_{t,n}(b_0(f)) =  b_0(f)(X_{t,n}) - \E{b_0(f)(X_{t,n})}, f \in \mathcal F$. Arguing as above, for $l = 2,...,\nu$ and $f \in \mathcal F,$ we have 
	$$ \E{|Z_{t,n}(b_0(f))|^{l\frac{2+\lambda}{2}}} \leq 2^{l\frac{2+\lambda}{2}} \E{|b_0(f)(X_{t,n})|^{l\frac{2+\lambda}{2}}} \leq 2^{\nu\frac{2+\lambda}{2}} \eta^2 = (2^{\frac{\nu}{2}} \tau_0)^{2+\lambda}$$ by \eqref{ConditionOnBoundingClassKk}, and thus Lemma \ref{MaximalInequalityUnderMixing} with $ \tau = 2^{\frac{\nu}{2}} \tau_0 $ entails	
	\begin{align*}
		\sup_{f \in \mathcal F} \norm{S_{n,i,j}(b_0(f))}_{L_\nu} &\leq C_2 \sqrt{m} \max\{m^{-\frac12}, 2^{\frac{\nu}{2}} \tau_0\} \\ 
		&\leq C_2 \sqrt{m} 2^{\frac{\nu}{2}} \max\{m^{-\frac12}, \tau_0\} = C_2 \sqrt{m} 2^{\frac{\nu}{2}} m^{-\frac12}.
	\end{align*} 
	By inserting the latter estimate into \eqref{ApproxForK0}, we conclude that for $a_0^f = a_0(f) \in \mathcal J_0$, $f \in \mathcal F,$ it holds
	\begin{align}\label{BoundForCase1}
		\norm{\sup_{f \in \mathcal F} \left|S_{n,i,j}(f) - S_{n,i,j}(a_0^f) \right|}_{L_\nu}^* 
		&\leq \sqrt{m} \left(C_2 2^{\frac{\nu}{2}} N_0^{\frac{1}{\nu}}m^{-\frac12} + 2 m^{-\frac{\lambda}{4}}  \right).
	\end{align}
	\underline{Case 2:} $\tau_0 = \eta^{2/(2+\lambda)} > m^{-1/2}$.\\
	Let $K = K(m) \in \mathbb N_0$ be the largest integer with $\tau_K \geq m^{-\frac12}$. Then $$ \sqrt{m} \eta_K = 2 \sqrt{m} \eta_{K+1} \leq 2 m^{-\frac{\lambda}{4}},$$ since $\eta_{K+1} = (\tau_{K+1})^\frac{2+\lambda}{2} \leq (m^{-\frac12})^{\frac{2+\lambda}{2}},$ and, by proceeding as in Case 1,
	\begin{align}\label{ApproxForKK}
		&\norm{\sup_{f \in \mathcal F} |S_{n,i,j}(f) - S_{n,i,j}(a_K(f))|}_{L_\nu}^* \nonumber\\
		&\leq N_K^{\frac{1}{\nu}} \sup_{f \in \mathcal F} \norm{S_{n,i,j}(b_K(f))}_{L_\nu} + 2\sqrt{m} \sqrt{m} \eta_K \nonumber \\
		&\leq N_K^{\frac{1}{\nu}} \sup_{f \in \mathcal F} \norm{S_{n,i,j}(b_K(f))}_{L_\nu} + 4 \sqrt{m} m^{-\frac{\lambda}{4}}. 
	\end{align}
	Once again, let $$Z_{t,n}(b_K(f)) = b_K(f)(X_{t,n}) - \E{b_K(f)(X_{t,n})}, f \in \mathcal F.$$ We then have $$ \E{|Z_{t,n}(b_K(f))|^{l\frac{2+\lambda}{2}}} \leq (2^{\frac{\nu}{2}} \tau_K)^{2+\lambda}$$ for each $l = 2,...,\nu$ and $f \in \mathcal F,$ and thus conclude that
	\begin{align*}
		\sup_{f \in \mathcal F} \norm{S_{n,i,j}(b_K(f))}_{L_\nu} &\leq C_2 \sqrt{m} \max\{m^{-\frac12}, 2^{\frac{\nu}{2}} \tau_K\} = C_2 \sqrt{m} 2^{\frac{\nu}{2}} \tau_K,
	\end{align*} 
	by Lemma \ref{MaximalInequalityUnderMixing}. By inserting the above estimate into \eqref{ApproxForKK}, we obtain
	\begin{align}\label{ApproxforKK2}
		&\norm{\sup_{f \in \mathcal F} \left|S_{n,i,j}(f) - S_{n,i,j}(a_K(f)) \right|}_{L_\nu}^* \nonumber\\ 
		&\leq \sqrt{m} \left(C_2 2^{\frac{\nu}{2}} N_K^{\frac{1}{\nu}} \tau_K + 4 m^{-\frac{\lambda}{4}}  \right).
	\end{align}
	We now further distinguish the cases $K = 0$ and $K > 0.$\\
	\underline{Case 2.1:} $K = 0$.\\
	The rough estimate 
	\begin{equation}\label{RoughEstimate}
		N_k^{\frac{1}{\nu}} \tau_k \leq \int_0^{\eta_k} N_{[]}^\frac{1}{\nu}(\varepsilon) \varepsilon^{-\frac{\lambda}{2+\lambda}} d\varepsilon \leq \int_0^\eta N_{[]}^\frac{1}{\nu}(\varepsilon) \varepsilon^{-\frac{\lambda}{2+\lambda}} d\varepsilon, ~ k \in \mathbb N_0,
	\end{equation}
	entails that for $a_0^f = a_0(f) \in \mathcal J_0, f \in \mathcal F,$ we have
	\begin{align}\label{BoundForCase21}
		&\norm{\sup_{f \in \mathcal F} \left|S_{n,i,j}(f) - S_{n,i,j}(a_0^f) \right|}_{L_\nu}^* \nonumber \\
		&\leq \sqrt{m} \left(C_2 2^{\frac{\nu}{2}} \int_0^\eta N_{[]}^\frac{1}{\nu}(\varepsilon) \varepsilon^{-\frac{\lambda}{2+\lambda}} d\varepsilon + 4 m^{-\frac{\lambda}{4}}  \right).
	\end{align}
	\underline{Case 2.2:} $K > 0$.\\
	We follow the steps taken in  the proof of \cite[Thm. 2.5]{MO20} and apply a chaining argument to obtain the elements $a_0^f \in \mathcal J_0, f \in \mathcal F$. To this end, pick any $\tilde a_K \in \mathcal J_K$ and let $\tilde a_{K-1}(\tilde a_K)$ be its best-approximation in $\mathcal J_{K-1}$ in the sense that
	\begin{align*}
		&\tilde a_{K-1}(\tilde a_K) \\ 
		&\in \mathrm{argmin}_{c \in \mathcal J_{K-1}} \left\{ \max_{l = 2,...,\nu} \sup_{1 \leq t \leq n, n \in \mathbb N} \E{|c(X_{t,n}) - \tilde a_K(X_{t,n}))|^{l\frac{2+\lambda}{2}}}^{\frac12}  \right\},
	\end{align*}
	which is well-defined by the nonnegativity of the objective function. Analogously, we may define the best-approximation of $\tilde a_{K-1}(\tilde a_K)$ in $\mathcal J_{K-2}$. By proceeding in this fashion for all $\tilde a_K \in \mathcal J_K$, we obtain a chain that links any $\tilde a_K \in \mathcal J_K$ to an element of $\mathcal J_0$. Hence, for each $f \in \mathcal F$ with approximating function $a_K(f) \in \mathcal J_K$, there exists a chain of best-approximations $a_{0}^f,...,a_{K-1}^f, a_K^f = a_K(f)$ running through the respective approximating classes $\mathcal J_0,...,\mathcal J_{K-1}, \mathcal J_K$ such that telescoping gives 
	\begin{align}\label{ChainFroma_Ktoa_0}
		\norm{\sup_{f \in \mathcal F} \left|S_{n,i,j}(a_K(f)) - S_{n,i,j}(a_0^f)\right|}_{L_\nu} \nonumber 
		&= \norm{\sup_{f \in \mathcal F} \left|\sum_{k = 1}^K S_{n,i,j}(a_k^f) - S_{n,i,j}(a_{k-1}^f)\right|}_{L_\nu} \nonumber \\
		&\leq \sum_{k = 1}^{K} \norm{\sup_{f \in \mathcal F} \left|S_{n,i,j}(a_k^f - a_{k-1}^f)\right|}_{L_\nu} \nonumber \\
		&\leq \sum_{k = 1}^{K} N_k^{\frac{1}{\nu}} \sup_{f \in \mathcal F}  \norm{S_{n,i,j}(a_k^f - a_{k-1}^f)}_{L_\nu},
	\end{align}
	where in the last step we used that for each $k = 1,...,K$, there exist no more than $N_k$ different pairs $(a_k^f, a_{k-1}^f) \in \mathcal J_k \times \mathcal J_{k-1}$, as each $a_k^f$ is linked to a single $a_{k-1}^f$. To estimate the last term of the above display, put again, for $f \in \mathcal F$ and $k = 1,...,K$, $$ Z_{t,n}(a_k^f - a_{k-1}^f) = (a_k^f - a_{k-1}^f)(X_{t,n}) - \E{(a_k^f - a_{k-1}^f)(X_{t,n})}.$$ For $k = 1,...,K$ and $f \in \mathcal F$, it holds $a_k^f \in \mathcal J_k \subset \mathcal F,$ and thus there exists an approximating function $c(a_k^f) \in \mathcal J_{k-1}$ and a bounding function $d(a_k^f) \in \mathcal K_{k-1}$ with $|a_k^f - c(a_k^f)| \leq d(a_k^f)$ and, by \eqref{ConditionOnBoundingClassKk}, $$ \max_{l = 2,...,\nu} \sup_{1 \leq t \leq n, n \in \mathbb N} \E{\left|d(a_k^f)(X_{t,n})\right|^{l \frac{2+\lambda}{2}}}^{\frac12} \leq \eta_{k-1},  $$ uniformly in $f \in \mathcal F$. Hence, since the $a_{k-1}^f$ are chosen to best-approximate the $a_k^f$, we have
	\begin{align*}
		&\max_{l = 2,...,\nu} \sup_{1 \leq t \leq n, n \in \mathbb N} \E{|Z_{t,n}(a_k^f - a_{k-1}^f)|^{l\frac{2+\lambda}{2}}}\\ 
		&\leq 2^{\nu\frac{2+\lambda}{2}} \max_{l = 2,...,\nu} \sup_{1 \leq t \leq n, n \in \mathbb N} \E{|(a_k^f - a_{k-1}^f)(X_{t,n})|^{l\frac{2+\lambda}{2}}} \\
		&\leq 2^{\nu\frac{2+\lambda}{2}} \max_{l = 2,...,\nu} \sup_{1 \leq t \leq n, n \in \mathbb N} \E{|(a_k^f - c(a_{k}^f))(X_{t,n})|^{l\frac{2+\lambda}{2}}} \\
		&\leq 2^{\nu\frac{2+\lambda}{2}} \eta_{k-1}^2 = \left(2^{\frac{\nu}{2}} \tau_{k-1}\right)^{2+\lambda}.
	\end{align*} 
	Again, Lemma \ref{MaximalInequalityUnderMixing} is applicable, and so, in view of $\tau_{k} \geq m^{-\frac12}$ for all $k = 0,...,K$ we obtain $$ \sup_{f \in \mathcal F}  \norm{|S_{n,i,j}(a_k^f - a_{k-1}^f)|}_{L_\nu} \leq C_2 \sqrt{m} 2^{\frac{\nu}{2}} \tau_{k-1},~ k = 1,..., K.$$ By inserting this estimate into \eqref{ChainFroma_Ktoa_0}, we arrive at the bound
	\begin{align*}
		\norm{\sup_{f \in \mathcal F} |S_{n,i,j}(a_K(f)) - S_{n,i,j}(a_0^f)|}_{L_\nu}^* &\leq \sum_{k = 1}^{K} N_k^{\frac{1}{\nu}} 2^{\frac{\nu}{2}} C_2 \sqrt{m} \tau_{k-1}\\
		&\leq 2^{\frac{\nu}{2}} C_2 \sqrt{m} \sum_{k = 1}^{\infty} N_{k}^\frac{1}{\nu} \eta_{k-1}^{\frac{2}{2+\lambda}} \\
		&\leq 2^{\frac{\nu}{2} + 2} C_2 \sqrt{m} \sum_{k = 1}^{\infty} N_{k}^\frac{1}{\nu} \eta_{k}^{-\frac{\lambda}{2+\lambda}} (\eta_{k} - \eta_{k+1}) \\
		&\leq 2^{\frac{\nu}{2} + 2} C_2 \sqrt{m} \sum_{k = 0}^{\infty} N_{k}^\frac{1}{\nu} \eta_{k}^{-\frac{\lambda}{2+\lambda}} (\eta_{k} - \eta_{k+1}) \\
		&\leq 2^{\frac{\nu}{2} + 2} C_2 \sqrt{m} \sum_{k = 0}^{\infty} \int_{\eta_{k+1}}^{\eta_{k}} N_{[]}^\frac{1}{\nu}(\varepsilon) \varepsilon^{-\frac{\lambda}{2+\lambda}} d\varepsilon \\
		&= 2^{\frac{\nu}{2} + 2} C_2 \sqrt{m}  \int_{0}^{\eta} N_{[]}^\frac{1}{\nu}(\varepsilon) \varepsilon^{-\frac{\lambda}{2+\lambda}} d\varepsilon,
	\end{align*}
	from which, by Lemma \ref{MinkowskiForOuter}, \eqref{ApproxforKK2} and \eqref{RoughEstimate}, 
	\begin{align}\label{BoundForCase22}
		&\norm{\sup_{f \in \mathcal F} |S_{n,i,j}(f) - S_{n,i,j}(a_0^f)|}_{L_\nu}^* \nonumber \\ 
		&\leq C_22^{\frac{\nu}{2}+2}  \sqrt{m} \left( N_{K}^{\frac{1}{\nu}} \tau_K + m^{-\frac{\lambda}{4}} + \int_{0}^{\eta} N_{[]}^\frac{1}{\nu}(\varepsilon) \varepsilon^{-\frac{\lambda}{2+\lambda}} d\varepsilon \right) \nonumber \\
		&\leq C_22^{\frac{\nu}{2}+2}  \sqrt{m} \left( m^{-\frac{\lambda}{4}} + 2 \int_{0}^{\eta} N_{[]}^\frac{1}{\nu}(\varepsilon) \varepsilon^{-\frac{\lambda}{2+\lambda}} d\varepsilon \right)
	\end{align}	
	follows for $a_0^f \in \mathcal J_0, f\in \mathcal F,$ defined as above.
	
	By combining the bounds in \eqref{BoundForCase1}, \eqref{BoundForCase21} and \eqref{BoundForCase22}, we conclude that in either case, we have 
	\begin{align*}
		&\norm{\sup_{f \in \mathcal F} |S_{n,i,j}(f) - S_{n,i,j}(a_0^f)|}_{L_\nu}^* \nonumber \\
		&\leq C_2 2^{\frac{\nu}{2} + 3} \sqrt{m} \left(N_0^{\frac{1}{\nu}} m^{-\frac12} + m^{-\frac{\lambda}{4}} + \int_{0}^{\eta} N_{[]}^\frac{1}{\nu}(\varepsilon) \varepsilon^{-\frac{\lambda}{2+\lambda}} d\varepsilon \right)
	\end{align*}
	for suitably chosen functions $a_0^f \in \mathcal J_0, f \in \mathcal F,$ which proves \eqref{ApproxByJ0} (with $C_1 = 2^{\frac{\nu}{2} + 3} C_2)$. This concludes the proof.
\end{proof}

\begin{proof}[Proof of Corollary \ref{FCLTApplicableUnderMixing}]
	We use Theorem \ref{UniformFCLT}. Firstly, by (ii) of Theorem \ref{SupremumInequalityUnderMixing}, $(\mathcal F, \rho)$ is totally bounded. Secondly, Theorem \ref{SupremumInequalityUnderMixing} with $\eta = \sqrt{\delta}$ gives 
	\begin{equation}\label{MixingFCLT_bound_algebraic}
		\norm{\sup_{\rho(f-g)\leq \delta} |S_{n,i,j}(f-g)|}_{L_\nu}^* \leq C \left[ m \left( R(\delta) + J(\delta) m^{-\kappa} \right)^2 \right]^{\frac{\nu}{2}} =: \gamma^{\frac{\nu}{2}}(m,\delta)
	\end{equation}
	for any $\delta > 0, n \in \mathbb N$ and $1 \leq i \leq j \leq n, m = j-i+1$, where
	\begin{align*}
	R(\delta) = N_{[]}^{\frac{2}{\nu}}\left(\sqrt{\delta}, \mathcal F, \rho_2 \right) \left(\delta + \delta^{\frac{\nu}{2}} \right) + \int_0^{\sqrt{\delta}} N_{[]}^\frac{1}{\nu}(\varepsilon, \mathcal F, \rho_2) \varepsilon^{-\frac{\lambda}{2+\lambda}} d\varepsilon
	\end{align*}
	and
	\begin{align*}
	J(\delta) = N_{[]}^{\frac{2}{\nu}}\left(\sqrt{\delta}, \mathcal F, \rho_2 \right)
	\end{align*}	
	are finite and nonnegative. Moreover, as $\delta \downarrow 0$,
	$$ N_{[]}^{\frac{2}{\nu}}\left(\sqrt{\delta} \right) \delta  = \left(N_{[]}^{\frac{1}{\nu}}(\sqrt{\delta}) \sqrt{\delta}\right)^2 \leq \left(\int_0^{\sqrt{\delta}} N_{[]}^{\frac{1}{\nu}}(\varepsilon) d\varepsilon \right)^2 \to 0, $$
	$$ N_{[]}^{\frac{2}{\nu}}\left(\sqrt{\delta} \right) \delta^{\frac{\nu}{2}} \to 0 $$ and $$ \int_0^{\sqrt{\delta}} N_{[]}^\frac{1}{\nu}(\varepsilon, \mathcal F, \rho_2) \varepsilon^{-\frac{\lambda}{2+\lambda}} d\varepsilon \to 0$$ by the dominated convergence theorem and $$ \int_0^1 N_{[]}^{\frac{1}{\nu}}(\varepsilon) d\varepsilon \leq \int_0^1 N_{[]}^\frac{1}{\nu}(\varepsilon) \varepsilon^{-\frac{\lambda}{2+\lambda}} d\varepsilon < \infty,$$ hence $R(\delta) \to 0$ as $\delta \downarrow 0.$ Finally, as the right-hand side of \eqref{MixingFCLT_bound_algebraic} is nonincreasing in $\kappa$, we can take it small enough to satisfy the condition of Theorem \ref{UniformFCLT}. This verifies condition \eqref{UniformFCLT-gammabound}. Lastly, by Jensen's inequality and \eqref{FCLTApplicableUnderMixingMomentCondition}, for each $f \in \mathcal F$ $$ \max_{l = 2,...,\nu} \sup_{1 \leq t \leq n, n \in \mathbb N} \E{\left|f(X_{t,n}) - \E{f(X_{t,n})} \right|^{\frac{l}{2}(2+\lambda)}} \leq \left(2^{\frac{\nu}{2}} K^{\frac{1}{2+\lambda}} \right)^{2+\lambda},$$ which by Lemma \ref{MaximalInequalityUnderMixing} entails that for each $n \in \mathbb N$ and $1 \leq i \leq j \leq n$,
	$$ \sup_{f \in \mathcal F} \norm{S_{n,i,j}(f)}_{L_\nu} \leq 2^{\frac{\nu}{2}} K^{\frac{1}{2+\lambda}} C_1 \sqrt{m}. $$ The constant $C_1$ only depends on $\nu, \lambda$ and the mixing coefficients, thereby proving \eqref{UniformFCLT-f0bound}. The result follows.
\end{proof}

\subsubsection{Proof of Theorem~\ref{SupremumInequalityUnderMixing-geometric}}
We first prove the maximal inequality.
\begin{proof}[Proof of Lemma \ref{MaximalInequality-geometric}]
	Let $p > 2$, $n \in \mathbb N$ and $1 \leq i \leq j \leq n$ with $m = j-i+1$. By \cite[Thm.~6.3]{Rio17}, there exist constants $a(p), b(p)$ that only depend on $p$ such that 
	\begin{align}\label{Maximal-geometric-initial}
		&\E{ \left| S_{n,i,j}(h) \right|^p } \nonumber \\
		&\leq a(p) \left( s_{n,i,j}^2 \right)^{\frac{p}{2}} + m b(p) \int_0^1 \left[\alpha_{n,i,j}^{-1}(u) \land m \right]^{p-1} \max_{i \leq k \leq j} Q_{k,h}^p(u) du.
	\end{align}
	Here, $\alpha_{n,i,j}(t)$ are the mixing coefficients of $(X_{k,n})_{i \leq k \leq j}$, the function $\alpha_{n,i,j}^{-1}(u)$ is defined in eq. (1.21) of \cite{Rio17}, $$ s_{n,i,j}^2 = \sum_{k_1, k_2 = i}^j |\mathrm{Cov}(h(X_{k_1,n}), h(X_{k_2,n}))|, $$ and $Q_{k,h}$ is the quantile function of $|h(X_{k,n}) - \E{h(X_{k,n})}|$. Observe that for all $t \geq 0$, $\alpha_{n,i,j}(t) \leq \alpha^X(t)$ and $ \max_{i \leq k \leq j} Q_{k,h} \leq 2 \norm{h}_\infty$. So, by the results in Appendix C of \cite{Rio17}, we have 
	\begin{align*}
		&\int_0^1 [\alpha_{i,j}^{-1}(u) \land m]^{p-1} du \max_{i \leq k \leq j} Q_{k,h}^p(u) du \\ 
		&\leq 2 \norm{h}_\infty^p (p-1) \sum_{k = 0}^{m-1} (k+1)^{p-2} \alpha_{n,i,j}(k) \\
		&\leq 2 \norm{h}_\infty^p (p-1) \sum_{k = 0}^{m-1} (k+1)^{p-2} \alpha^X(k).
	\end{align*}
	Next, as $\alpha^X(k) \leq C_\beta \beta^k$ and $$(k+1)^{\floor{p}} \leq \floor{p}! \binom{k+\floor{p}}{\floor{p}}, $$
	\begin{align*}
		\sum_{k = 0}^{m-1} (k+1)^{p-2} \alpha^X(k) &\leq C_\beta \sum_{k = 0}^{\infty} (k+1)^{\floor{p}} \beta^k \\ 
		&\leq \floor{p}! C_\beta \sum_{k = 0}^{\infty} \binom{k+\floor{p}}{\floor p} \beta^k = \floor{p}! \frac{1}{(1 - \beta)^{\floor{p}+1}}
	\end{align*} 
	so that, since $$ \limsup_{p \to \infty} \frac{(\floor{p}!)^{\frac{1}{p}}}{\floor{p}} < \infty, $$
	$$ \left[ \sum_{k = 0}^{\infty} (k+1)^{p-2} \alpha^X(k) \right]^{\frac{1}{p}} \lesssim p C_\beta (1 - \beta)^{-2}. $$
	Furthermore, by \cite[Cor.~1.1]{Rio17} and similar arguments,
	$$ s_{n,i,j}^2 \leq 4 m \int_0^1 [\alpha^X]^{-1}(u) \sup_{1 \leq k \leq n, n \in \mathbb N} Q_{k,h}^2(u) du =: 4 m \norm{h}_\alpha^2. $$	
	Finally, since $a(p)^{1/p} \lesssim \sqrt{p}$ and $b(p)^{1/p} \lesssim p$ (see \cite[Lem.~2]{H05}), we have thus shown the existence of a constant $C = C(\beta) \geq 0$ that only depends on the mixing coefficients and fulfills
	\begin{align}\label{Maximal-geometric-second}
		\norm{ S_{n,i,j}(h) }_{L_\nu} &\leq C \sqrt{m} \left(\sqrt{p} \norm{h}_{\alpha} + p^2 m^{-\frac12 + \frac{1}{p}} \norm{h}_\infty \right).
	\end{align}
	To further estimate \eqref{Maximal-geometric-second}, note that by Markov's inequality,  
	$$ P\left(\left|h(X_{k,n}) - \E{h(X_{k,n})} \right| > x \right) \leq \left( \frac{ 2 \rho_\nu(h) }{x} \right)^\nu, ~ x > 0. $$ The results in Appendix C of \cite{Rio17} hence give that $Q_{k,h}(u) \leq 2 \rho_\nu(h) u^{-1/\nu}$, uniformly in $k$. So, similar to the proof of \cite[Thm.~3]{H05}, by H\"older's inequality, for any $0 < 2/\nu < \theta < 1$, we have
	\begin{align}\label{geometric-normbound}
		\norm{h}_\alpha^2 &\leq \left( \int_0^1 \left( [ \alpha^X]^{-1} (u) \right)^{\frac{1}{1- \theta}} du \right)^{1-\theta} \left(\int_0^1 \sup_{1 \leq k \leq n, n \in \mathbb N} Q_{k,h}^{\frac{2}{\theta}}(u) du \right)^\theta \nonumber \\
		&\leq 4 \rho_\nu^2(h) \left(\frac{1}{1- \theta} \sum_{k = 0}^{\infty} (i+1)^{\frac{\theta}{1- \theta}} \alpha^X(k) \right)^{1-\theta} \left( \int_0^1 u^{-\frac{2}{\theta \nu}} \right)^\theta \nonumber \\
		&\leq C_1 \rho_\nu^2(h),
	\end{align}
	where $C_1$ does only depend on $\nu$ (via $\theta$) and the mixing coefficients. Here, we have applied eq. (C.5) of \cite{Rio17} in the second last inequality and that $2/(\theta \nu) < 1$ by our choice of $\theta$. By plugging \eqref{geometric-normbound} into \eqref{Maximal-geometric-second}, we conclude the first part of Lemma \ref{MaximalInequality-geometric}.
	
	The second assertion now follows from standard arguments (see, e.g., Section 2 in \cite{DL02}): by the first part of Lemma \ref{MaximalInequality-geometric}, for any $p \geq \nu$, 
	\begin{align*}
		\norm{\max_{h \in \mathcal H} \left|S_{n,i,j}(h) \right| }_{L_\nu} &\leq \norm{\max_{h \in \mathcal H} \left|S_{n,i,j}(h) \right| }_{L_p} \\
		&\leq \left( \sum_{h \in \mathcal H} \norm{ S_{n,i,j}(h) }_{L_p} \right)^{\frac{1}{p}} \\
		&\leq \#(\mathcal H)^{\frac{1}{p}} C \sqrt{m} \left( \sqrt{p} \max_{h \in \mathcal H} \rho_\nu(h) + p^2 m^{-\frac12 + \frac{1}{p}} \max_{h \in \mathcal H} \norm{h}_\infty \right).
	\end{align*} 
	So, if we choose $p = \nu \left(1 \lor \log \#(\mathcal H) \right)$, then, since $m^{1/p} \leq m^{1/\nu}$,
	\begin{align*}
		&\norm{\max_{h \in \mathcal H} \left|S_{n,i,j}(h) \right| }_{L_\nu} \\ 
		&\leq C e \nu^2 \left(1 \lor \log \#(\mathcal H) \right)^2 \sqrt{m} \left( \max_{h \in \mathcal H} \rho_\nu(h) + m^{-\frac12 + \frac{1}{\nu}} \max_{h \in \mathcal H} \norm{h}_\infty \right),
	\end{align*} 
	as claimed.
\end{proof}

\begin{proof}[Proof of Theorem \ref{SupremumInequalityUnderMixing-geometric}]
	Let $\delta > 0$, $n \in \mathbb N$ and $1 \leq i \leq j \leq n$. Denote $m = j-i+1$ and, for any $k \in \mathbb N_0$, $\eta_k = 2^{-k}$, $N_k = N_{[]}(\eta_k, \mathcal F, \rho_\nu)$. Recall that by \eqref{geometric-integral}, for any $k \in \mathbb N_0$ and $f \in \mathcal F$ there exist $a_k(f) \in J_k \subset \mathcal F$ and $b_k(f) \in \mathcal K_k$ such that $ |f - a_k(f) | \leq b_k(f)$, $\rho_\nu(b_k(f)) \leq \eta_k$, and $\#(\mathcal J_k) = \#(\mathcal K_k) = N_k < \infty$. 
	
	We start by defining some quantities to be used below. Let 
		\begin{equation}\label{geometric-DEF-s}
			s(k) = 2^{-k} \left(\sqrt{k} + \log^{\frac12} N_k + k^2 + \log^2 N_k \right), ~ k \in \mathbb N. 
		\end{equation} 
		By \eqref{geometric-integral}, this is a summable sequence, and so there exists a smallest positive integer $K_0 \in \mathbb N$ that only depends on $\mathcal F$ with 
		\begin{equation}\label{geometric-K0}
			0 < \sum_{k = K_0}^\infty s(k) \leq \frac{1}{2} ~~~ \text{ and } K_0 > \nu.
		\end{equation}
		Furthermore, let 
		\begin{equation}\label{geometric-K(D)}
			K(\delta) = \max\left\{K_0, \max\{z \in \mathbb N ~|~ N_z \leq \delta^{-1/2} \} \right\}
		\end{equation} 
		and note that $K(\delta) \geq 1$ and that we can assume that $K(\delta) \to \infty$, if $\delta \downarrow 0.$ Finally, let 
		\begin{equation}\label{geometric-DEF-eps-function}
			\varepsilon(\delta) = \sum_{k = K(\delta)}^\infty s(k)
		\end{equation} 
		and $$ K(m,\delta) = \left[ \frac{1}{2 \log 2} \log \left( \frac{m}{\varepsilon(\delta)} \right) \right], $$ where $[x]$ denotes the integer closest to $x \in \mathbb R.$ Note that $K(m,\delta)$ is well-defined, since $\varepsilon(\delta) > 0$, and also that $K(m,\delta) > 0,$ since 
		\begin{equation}\label{geometric-eps-function}
			\varepsilon(\delta) \leq \sum_{k = K_0}^\infty s(k) \leq \frac{1}{2} < 1,
		\end{equation} 
		by \eqref{geometric-K0}. Moreover, if $\delta \downarrow 0,$ then $K(\delta) \to \infty$ and therefore $\varepsilon(\delta) \to 0$. 
	
	Now, our strategy is the same as in the proof of Theorem \ref{SupremumInequalityUnderMixing}. The first step is to prove that for each $f \in \mathcal F$, there exists $a(K(\delta), f) \in \mathcal J_{K(\delta)}$ with 
		\begin{align}\label{geometric-mainbound}
			&\norm{\sup_{f \in \mathcal F} \left|S_{n,i,j}(f) - S_{n,i,j}(a(K(\delta), f)) \right|}_{L_\nu}^* \nonumber \\ 
			&\leq C \sqrt{m} \left( \varepsilon^{\frac12}(\delta) + \log^2 N_{K(\delta)} 2^{-K(\delta)} + \log^2 N_{K(\delta)} m^{-\left(\frac12 - \frac{1}{\nu} \right)}  \right),
		\end{align}
		where the constant $C \geq 0$ does only depend on $\nu$ and the mixing coefficients. Having shown this, we can argue exactly as in the proof of Theorem \ref{SupremumInequalityUnderMixing} (but with $\rho_\nu$ in place of $\rho$) to find that 
		\begin{align*}
			&\norm{ \sup_{\rho_\nu(f-g) \leq \delta} \left| S_{n,i,j}(f) - S_{n,i,j}(g) \right| }_{L_\nu}^* \\
			&\leq 4 \norm{\sup_{f \in \mathcal F} \left|S_{n,i,j}(f) - S_{n,i,j}(a(K(\delta), f)) \right|}_{L_\nu}^* \\ 
			&+ \norm{ \max_{1 \leq r,s \leq N_{K(\delta)}, ~ \rho_\nu(\phi_{r,s} - \psi_{s,r}) \leq 2\delta } \left|S_{n,i,j}(\phi_{r,s}) - S_{n,i,j}(\psi_{s,r}) \right|}_{L_\nu},
		\end{align*}
		where the $\phi_{r,s}$ and $\psi_{s,r}$ are functions in $\mathcal F$ that are defined in the proof of Theorem \ref{SupremumInequalityUnderMixing} (between eq. \eqref{ApproxWithinEquivalenceClasses} and \eqref{MaxMeas}). As the maximum in the rightmost term of the above display runs over at most $N_{K(\delta)}^2$ many functions, Lemma \ref{MaximalInequality-geometric} entails that 
		\begin{align}\label{geometric-remainder}
			&\norm{ \max_{1 \leq r,s \leq N_{K(\delta)}, ~ \rho_\nu(\phi_{r,s} - \psi_{s,r}) \leq 2\delta } \left|S_{n,i,j}(\phi_{r,s} - \psi_{s,r}) \right|}_{L_\nu} \nonumber \\ 
			&\leq 4 C_0 \log^2 N_{K(\delta)} \sqrt{m} \left( 2 \delta + 2 m^{-(\frac12 - \frac{1}{\nu})} \right)
		\end{align}			 			
	for a constant $C_0$ that only depends on $\nu$ and the mixing coefficients. Here, we have used that since $\sup_{f \in \mathcal F} |f| \leq 1$, $|\phi_{r,s} - \psi_{s,r}| \leq 2$. The bound asserted by Theorem \ref{SupremumInequalityUnderMixing-geometric} then follows from \eqref{geometric-mainbound}, \eqref{geometric-remainder} and by choosing $0 < \kappa \leq 1/2 - 1/\nu$ and 
	$$ \Lambda(\delta) = \varepsilon^{\frac12}(\delta) + \log^2 N_{K(\delta)} \left( 2^{-K(\delta)} + \delta \right), ~~~ \lambda(\delta) = \log^2 N_{K(\delta)}. $$ 
	Moreover, it is clear that $\Lambda$ and $\lambda$ are finite and nonnegative. Since $K(\delta) \to \infty$ if $\delta \downarrow 0$, we then also have $\varepsilon(\delta) \to 0,$ 
	$$ \log^2 N_{K(\delta)} \delta \leq  \log^2 \left( \delta^{-\frac12} \right) \delta \to 0 $$ and
	$$ \log^2 N_{K(\delta)} 2^{-K(\delta)} \leq \int_0^{2^{-K(\delta)}} \log^2 N_{[]}(\varepsilon, \mathcal F, \rho_\nu) d\varepsilon \to 0, $$ by our choice of $K(\delta)$, \eqref{geometric-integral} and dominated convergence, from which the result follows.
	
	It therefore suffices to prove \eqref{geometric-mainbound}. To do so, we distinguish the two cases $K(m,\delta) < K(\delta)$ and $K(m,\delta) \geq K(\delta)$. In preparation, note that since $[x] - x \in [-1/2, 1/2]$, it holds 
		\begin{equation}\label{geometric-upper}
			2^{-K(m,\delta)} \leq \sqrt{2} \exp\left( \frac{1}{2} \log \left( \frac{m}{\varepsilon(\delta)} \right) \right) = \sqrt{2} \varepsilon^{\frac12}(\delta) m^{-\frac12},
		\end{equation} 
		\begin{equation}\label{geometric-upper2}
			m^{-\frac12} 2^{K(m,\delta)} \leq \sqrt{2} m^{-\frac12} \sqrt{m} \varepsilon^{-\frac12}(\delta) = \sqrt{2} \varepsilon^{-\frac12}(\delta)
		\end{equation}
		and
		\begin{equation}\label{geometric-lower}
			m^{-\frac12} 2^{K(m,\delta)} \geq \frac{1}{\sqrt{2}} m^{-\frac12} \sqrt{m} \varepsilon^{-\frac12}(\delta) = \frac{1}{\sqrt{2}} \varepsilon^{-\frac12}(\delta) \geq 1,
		\end{equation}
		where the last inequality is a consequence of \eqref{geometric-eps-function}. 	\\
	
		\noindent \underline{Case 1:} $K(m,\delta) < K(\delta)$. \\
		For each $f \in \mathcal F$, let $a(K(\delta), f)$ be a function in $\mathcal J_{K(\delta)}$ with $|f - a(K(\delta), f)| \leq b_{K(\delta)}(f) \in \mathcal K_{K(\delta)}.$ Then, by Lemma \ref{MinkowskiForOuter}, we have 
		\begin{align*}
			&\norm{\sup_{f \in \mathcal F} \left|S_{n,i,j}(f) - S_{n,i,j}(a(K(\delta), f)) \right|}_{L_\nu}^* \nonumber \\ 
			&\leq 2 m \sup_{f \in \mathcal F} \rho_\nu(b_{K(\delta)}(f)) + \norm{\sup_{f \in \mathcal F} \left|S_{n,i,j}(b_{K(\delta)}(f)) \right|}_{L_\nu}.
		\end{align*}
		So, since $K(\delta) > K(m,\delta)$, by construction and \eqref{geometric-upper}, it holds $$ 2 m \sup_{f \in \mathcal F} \rho_\nu(b_{K(\delta)}(f)) \leq 2 m 2^{-K(\delta)} \leq 2 m 2^{-K(m,\delta)} \leq 2^{\frac32} \sqrt{m} \varepsilon^{\frac12}(\delta). $$ Furthermore, as in the proof of \cite[Thm.~3]{H05}, since $\sup_{f \in \mathcal F} |f| \leq 1$, we may assume that $\sup_{f \in \mathcal F} |b_{K(\delta)}(f)| \leq 1$. Hence, Lemma \ref{MaximalInequality-geometric} entails 
		\begin{align*}
			\norm{\sup_{f \in \mathcal F} \left|S_{n,i,j}(b_{K(\delta)}(f)) \right|}_{L_\nu} &\leq C_0 \log^2 N_{K(\delta)} \sqrt{m} \left( \sup_{f \in \mathcal F} \rho_\nu(b_{K(\delta)}(f)) + m^{-(\frac12 - \frac{1}{\nu})}  \right) \\
			&\leq C_0 \log^2 N_{K(\delta)} \sqrt{m} \left( 2^{-K(\delta)} + m^{-(\frac12 - \frac{1}{\nu})}  \right).
		\end{align*} 
		By combining the above two estimates, we obtain
		\begin{align}\label{geometric-firstcase}
			&\norm{\sup_{f \in \mathcal F} \left|S_{n,i,j}(f) - S_{n,i,j}(a(K(\delta), f)) \right|}_{L_\nu}^* \nonumber \\
			&\leq 2^{\frac32} C_0 \sqrt{m} \left( \varepsilon^{\frac12}(\delta) + \log^2 N_{K(\delta)} \left( 2^{-K(\delta)} + m^{-(\frac12 - \frac{1}{\nu})} \right) \right).
		\end{align}
	
		\noindent \underline{Case 2:} $K(m,\delta) \geq K(\delta)$. \\
		We can argue as in Case 2.2 of the proof of Theorem \ref{SupremumInequalityUnderMixing}. That is, for each $f \in \mathcal F$, we can construct a chain of $\rho_\nu$-best-approximations $a(K(\delta),f),...,a(K(m,\delta),f) = a_{K(m,\delta)}(f)$, i.e., for any $k = K(\delta)+1,...,K(m,\delta)$, we have $$ \rho_\nu \left( a(k,f) - a(k-1,f) \right) = \min_{a \in \mathcal J_{k-1}} \rho_\nu \left( a(k,f) - a \right).$$ Then, by $|f - a_{K(m,\delta)}| \leq b_{K(m,\delta)}(f)$ and Lemma \ref{MinkowskiForOuter}, it holds 
		\begin{align}\label{geometric-secondcase-bound}
			&\norm{\sup_{f \in \mathcal F} \left|S_{n,i,j}(f) - S_{n,i,j}(a(K(\delta), f)) \right|}_{L_\nu}^* \nonumber \\
			&\leq \norm{\sup_{f \in \mathcal F} \left|S_{n,i,j}(f) - S_{n,i,j}(a_{K(m,\delta)}(f)) \right|}_{L_\nu}^* \nonumber \\  &+\norm{\sup_{f \in \mathcal F} \left|S_{n,i,j}(a_{K(m,\delta)}(f)) - S_{n,i,j}(a(K(\delta), f)) \right|}_{L_\nu}^* \nonumber \\
			&\leq 2 m \sup_{f \in \mathcal F} \rho_\nu(b_{K(m,\delta)}(f)) + \norm{\sup_{f \in \mathcal F} \left|S_{n,i,j}(b_{K(m,\delta)}(f)) \right|}_{L_\nu} \nonumber \\
			&+ \sum_{k = K(\delta) + 1}^{K(m,\delta)} \norm{\sup_{f \in \mathcal F} \left| S_{n,i,j}\left(a(k,f) - a(k-1,f) \right) \right|}_{L_\nu}
		\end{align}
		Note that if $K(\delta) = K(m,\delta)$, the rightmost term in the above display would not appear. Furthermore, recall from Case 2.2 of the proof of Theorem \ref{SupremumInequalityUnderMixing} that our construction ensures that the $\sup_{f \in \mathcal F}$ in the rightmost term runs over at most $N_k$ many functions. 
		Now, to estimate the right-hand side of \eqref{geometric-secondcase-bound}, first observe that by construction and \eqref{geometric-upper}, 
		\begin{equation}\label{geometric-secondcase-bound-1}
			2 m \sup_{f \in \mathcal F} \rho_\nu(b_{K(m,\delta)}(f)) \leq 2 m 2^{-K(m,\delta)} \leq 2^{\frac32} \sqrt{m} \varepsilon^{\frac12}(\delta). 
		\end{equation} 
		To bound the remaining two terms, we can argue similar to \cite[Thm.~3]{H05}.	We only discuss the rightmost term in \eqref{geometric-secondcase-bound} in detail, the other one can be handled similarly. By \eqref{MaxBound} and Lemma \ref{MaximalInequality-geometric}, for each $k = K(\delta) + 1,...,K(m,\delta)$ and $p > \nu$, it holds
		\begin{align}\label{geometric-secondcase-bound-2-prog}
			&\norm{\sup_{f \in \mathcal F} \left| S_{n,i,j}\left(a(k,f) - a(k-1,f) \right) \right|}_{L_\nu} \nonumber \\ 
			&\leq N_k^{\frac{1}{p}} \sup_{f \in \mathcal F} \norm{S_{n,i,j}\left(a(k,f) - a(k-1,f) \right) }_{L_\nu} \nonumber \\
			&\leq C_1 N_k^{\frac{1}{p}} \sqrt{m} \left(\sqrt{p} \sup_{f \in \mathcal F} \rho_\nu(a(k,f) - a(k-1,f)) + p^2 2 m^{-(\frac12 - \frac{1}{p})}  \right) \nonumber \\
			&\leq 2 C_1 N_k^{\frac{1}{p}} \sqrt{m} \left(\sqrt{p} 2^{-k} + p^2 m^{-(\frac12 - \frac{1}{p})}  \right), 
		\end{align} 
		where the constant $C_1$ does only depend on $\nu$ and the mixing coefficients. Here, we have used that since $a(k,f) \in \mathcal J_k \subset \mathcal F$, it holds
		$$ \sup_{f \in \mathcal F} \rho_\nu \left(a(k,f) - a(k-1,f) \right) = \sup_{f \in \mathcal F} \min_{a \in \mathcal J_{k-1}} \rho_\nu \left(a(k,f) - a \right) \leq  2^{-k+1}. $$
		Furthermore, by \eqref{geometric-lower} and since $p > \nu > 2$ and $k \leq K(m,\delta)$, we have
		\begin{align*}
			m^{\frac{1}{p} - \frac12} &= \left(m^{-\frac12} 2^{k} \right)^{1-2/p} 2^{\frac{2k}{p}} 2^{-k} \\
			&\leq \left(m^{-\frac12} 2^{K(m,\delta)} \right)^{1-2/p} 2^{\frac{2k}{p}} 2^{-k} \\
			&\leq \left(m^{-\frac12} 2^{K(m,\delta)} \right) 2^{\frac{2k}{p}} 2^{-k}.
		\end{align*} 
		So, in view of \eqref{geometric-secondcase-bound-2-prog} and \eqref{geometric-lower}, we have
		\begin{align*}
			&\norm{\sup_{f \in \mathcal F} \left| S_{n,i,j}\left(a(k,f) - a(k-1,f) \right) \right|}_{L_\nu} \\ 
			&\leq 2 C_1 N_k^{\frac{1}{p}} \sqrt{m} \left(\sqrt{p} 2^{-k} + p^2 \left(m^{-\frac12} 2^{K(m,\delta)} \right) 2^{\frac{2k}{p}} 2^{-k} \right) \\
			&\leq 2 C_1 N_k^{\frac{1}{p}} \sqrt{m} \left(m^{-\frac12} 2^{K(m,\delta)} \right) \left(\sqrt{p} 2^{-k} + p^2 2^{\frac{2k}{p}} 2^{-k} \right),
		\end{align*} 
		for any $p>\nu$ and $k = K(\delta)+1,...,K(m,\delta).$ Following the proof of \cite[Thm.~3]{H05}, the choice $p = k + \log N_k$ (which fulfills $p \geq k \geq K(\delta) + 1 \geq K_0 + 1 > \nu$, by \eqref{geometric-K0}) now entails
		\begin{align}\label{geometric-secondcase-bound-2-prog-2}
			&\norm{\sup_{f \in \mathcal F} \left| S_{n,i,j}\left(a(k,f) - a(k-1,f) \right) \right|}_{L_\nu} \nonumber \\ 
			&\leq C_2 \sqrt{m} \left(m^{-\frac12} 2^{K(m,\delta)} \right) 2^{-k} \left(\sqrt{k} + \log^{\frac12} N_k + k^2 + \log^2 N_k \right),
		\end{align}
		where $C_2$ is a constant multiple of $C_1$, since $N_k^{1/p} \leq N_k^{1/(\log N_k)} \leq e$. 
		The estimate \eqref{geometric-secondcase-bound-2-prog-2} holds true for all $k = K(\delta) + 1,...,K(m,\delta)$, and by arguing analogously, we also find that 
		\begin{align}\label{geometric-secondcase-bound-2-prog-3}
			&\norm{\sup_{f \in \mathcal F} \left| S_{n,i,j}\left(b_{K(m,\delta)}(f) \right) \right|}_{L_\nu} \nonumber \\ 
			&\leq C_2 \sqrt{m} \left(m^{-\frac12} 2^{K(m,\delta)} \right) 2^{-K(m,\delta)} \bigg(\sqrt{K(m,\delta)} + \log^{\frac12} N_{K(m,\delta)} \nonumber \\
			&+ K(m,\delta)^2 + \log^2 N_{K(m,\delta)} \bigg) \nonumber \\
			&\leq C_2 \sqrt{m} \left(m^{-\frac12} 2^{K(m,\delta)} \right) \sum_{k = K(m,\delta)}^\infty s(k) \nonumber \\
			&\leq C_2 \sqrt{m} \left(m^{-\frac12} 2^{K(m,\delta)} \right) \varepsilon(\delta)
		\end{align}
		(recall from \eqref{geometric-DEF-s} and \eqref{geometric-DEF-eps-function} the definition of $s(k)$ and $\varepsilon(\delta)$). In the last step, we have used that $K(m,\delta) \geq K(\delta)$. In view of \eqref{geometric-secondcase-bound-2-prog-3}, \eqref{geometric-secondcase-bound-2-prog-2}, \eqref{geometric-secondcase-bound-1} and \eqref{geometric-secondcase-bound}, we have thus shown that 
		\begin{align}\label{geometric-secondcase-bound-3}
			&\norm{\sup_{f \in \mathcal F} \left|S_{n,i,j}(f) - S_{n,i,j}(a(K(\delta), f)) \right|}_{L_\nu}^* \nonumber \\ 
			&\leq C_3 \sqrt{m} \left( \varepsilon(\delta) \left(m^{-\frac12} 2^{K(m,\delta)} \right) + \varepsilon^{\frac12} (\delta) \right),
		\end{align}
		where $C_3$ depends on $\nu$ and the mixing coefficients only. By \eqref{geometric-upper2}, this entails 
		\begin{align}\label{geometric-secondcase}
			&\norm{\sup_{f \in \mathcal F} \left|S_{n,i,j}(f) - S_{n,i,j}(a(K(\delta), f)) \right|}_{L_\nu}^* \nonumber \\ 
			&\leq 2^{\frac32} C_3 \sqrt{m} \varepsilon^{\frac12} (\delta).
		\end{align}
		Taken together, the statements \eqref{geometric-secondcase} and \eqref{geometric-firstcase} prove \eqref{geometric-mainbound}.
\end{proof}

\begin{proof}[Proof of Corollary \ref{FCLTApplicableUnderMixing-geometric}]
	We use Theorem \ref{UniformFCLT}. By \eqref{geometric-integral}, $(\mathcal F, \rho_\nu)$ is totally bounded, and condition \eqref{UniformFCLT-gammabound} is an immediate consequence of Theorem \ref{SupremumInequalityUnderMixing-geometric} as we can take $\kappa > 0$ as small as desired. Finally, for any $f_0 \in \mathcal F,$ since $|f_0| \leq \sup_{f \in \mathcal F}|f| \leq 1$, Lemma \ref{MaximalInequality-geometric} gives
	$$ \norm{ S_{n,i,j}(f_0) }_{L_\nu} \leq C \sqrt{m} \left( \sqrt{\nu} \rho_\nu(f_0) + \nu^2 m^{-\left(\frac12 - \frac{1}{\nu}\right)} \right) \leq 2 \nu^2 C \sqrt{m}, $$ for any $1 \leq i \leq j \leq n$ with $m = j-i+1.$ This proves \eqref{UniformFCLT-f0bound} and concludes the proof.
\end{proof}

\end{document}